\documentclass{article}
\usepackage[utf8]{inputenc}
\usepackage{amsmath,amsthm,amssymb}
\usepackage{mathtools, tikz-cd}
\usepackage{color}
\usepackage{cite}

\usepackage{url}

\theoremstyle{definition}
\newtheorem{definition}{Definition}[section]

\theoremstyle{definition}
\newtheorem{proposition}{Proposition}[section]
\theoremstyle{definition}
\newtheorem{theorem}{Theorem}[section]
\theoremstyle{definition}
\newtheorem{lemma}{Lemma}[section]
\theoremstyle{definition}
\newtheorem{corollary}{Corollary}[section]
\theoremstyle{definition}

\theoremstyle{definition}
\newtheorem{exercise}{Exercise}[section]

\theoremstyle{remark}
\newtheorem{observation}{Observation}[section]
\theoremstyle{definition}
\newtheorem{remark}{Remark}[section]

\newtheorem{example}{Example}[section]

\DeclareMathOperator{\Mat}{\operatorname{Mat}}
\DeclareMathOperator{\Hom}{\operatorname{Hom}}
\DeclareMathOperator{\End}{\operatorname{End}}
\DeclareMathOperator{\Aut}{\operatorname{Aut}}

\DeclareMathOperator{\Lie}{\operatorname{Lie}}
\DeclareMathOperator{\SO}{\operatorname{SO}}
\DeclareMathOperator{\Sl}{\operatorname{Sl}}
\DeclareMathOperator{\Ad}{\operatorname{Ad}}
\DeclareMathOperator{\ad}{\operatorname{ad}}

\DeclareMathOperator{\im}{\operatorname{im}}
\DeclareMathOperator{\g}{\mathfrak{g}}

\DeclareMathOperator{\so}{\mathfrak{so}}

\DeclareMathOperator{\pr}{\operatorname{pr}}
\DeclareMathOperator{\Nbb}{\mathbb{N}}

\DeclareMathOperator{\Rbb}{\mathbb{R}}
\DeclareMathOperator{\Cbb}{\mathbb{C}}
\DeclareMathOperator{\Kbb}{\mathbb{K}}
\DeclareMathOperator{\del}{\partial}
\DeclareMathOperator{\A}{\mathcal{A}}
\DeclareMathOperator{\B}{\mathcal{B}} 
\DeclareMathOperator{\J}{\mathcal{J}} 
\DeclareMathOperator{\I}{\mathcal{I}} 
\DeclareMathOperator{\U}{\mathcal{U}} 
\DeclareMathOperator{\D}{\mathcal{D}}
\DeclareMathOperator{\Pol}{\operatorname{Pol}}
\DeclareMathOperator{\PC}{\operatorname{PC}}
\DeclareMathOperator{\PSpec}{\operatorname{PSpec}}

\let\S\relax \DeclareMathOperator{\S}{\mathcal{S}}
\let\d\relax \DeclareMathOperator{\du}{\underline{\operatorname{d}}}
\let\sl\relax \DeclareMathOperator{\sl}{\mathfrak{sl}} 
\DeclareMathOperator{\Poiss}{\operatorname{Poiss}}
\DeclareMathOperator{\Cas}{\operatorname{Cas}}
\DeclareMathOperator{\Ham}{\operatorname{Ham}}
\DeclareMathOperator{\gr}{\operatorname{gr}}
\DeclareMathOperator{\tr}{\operatorname{tr}}
\DeclareMathOperator{\Diff}{\operatorname{Diff}}
\DeclareMathOperator{\Smbl}{\operatorname{Smbl}}
\DeclareMathOperator{\Alt}{\operatorname{Alt}}
\DeclareMathOperator{\Sym}{\operatorname{Sym}}

\DeclareMathOperator{\Jac}{\operatorname{Jac}}

\let\sgn\relax\DeclareMathOperator{\sgn}{\operatorname{sgn}}
\DeclareMathOperator{\s}{\operatorname{s}}

\usepackage{pb-diagram}
\usepackage[matrix,arrow,curve]{xy}

\newcommand{\db}[1]{\mathopen{\{\!\!\{}#1\mathopen{\}\!\!\}}}

\title{Lectures on Poisson algebras}
\author{V. Roubtsov$^*$, R. Suchánek$^{\bullet}$ \\$^*$ volodya@univ-angers.fr \ \ $^{\bullet}$ r.suchanek.r@gmail.com }
\date{\today}

\begin{document}

\maketitle

\tableofcontents

\newpage

\section{Introduction}

The notion of a Poisson algebra was probably introduced in the first time by A.M. Vinogradov and J. S. Krasil'shchik in 1975 under the name "canonical algebra" and 
by J. Braconnier in his short note  "Algèbres de Poisson" (Comptes rendus 
Ac.Sci) in 1977.

It was a natural "algebraic interpretation" of the notion of Poisson structure and Poisson brackets appeared in XIX century in the framework of Classical Mechanics.

Nowadays, Poisson algebras have proved to have a very rich mathematical structure
(see the beautiful short article of Y. Kosmann-Schwarzbach in "Encyclopedia of Mathematics" (Poisson algebra, Encyclopedia of Mathematics,  \url{http://encyclopediaofmath.org/index.php?title=Poisson_algebra&oldid=51639}).

There are many good books and lecture notes devoted to Poisson structures and 
much more less sources concerning the algebraic side of the story. We should mention the extensive book of C. Laurent-Gengoux, A. Pichereau and P. Vanhaecke "Poisson structures" which covers also many algebraic aspects and structures associated with the notion of a Poisson variety \cite{Poisson-Structures}. 

These notes are slightly extended reproduction of the mini-course of the same title (5 lectures) given by the first author during the virtual (online) Winter School and Workshop "Wisla 20-21" and handled and written by the second. 

The virtual character of lectures has imposed few constraints and has defined a specific choice of material and chosen subjects. The authors had to pass between
Scylla of unknown audience level and tantalizing Charybdis of their will to exit out of the standard set of content in numerous existed lecture notes about Poisson structures in algebra and in geometry.

This contradiction can probably explain some strange "jumps" of a difficulty level between various chapters of our notes. We hope to come back to these notes and to extend, to improve, or even, to write a new comprehensive book covering the subject. Right now we were on a serious time (and the Covid pandemic) limits pressure and it was also a reason of some non-homogeneous choice of the lecture notes matters and its details.

There are (almost) no "new results" in the lectures. The only exception is a~content of the last brief chapter where we provide our classification results (joint with A. Odesskii and V. Sokolov) on the "low rank" double quadratic Poisson brackets. We had kept this chapter almost in it's specific form of a~computer presentation.

The lectures are based on many various well- and less known sources, which we had tried to carefully quot in the main body of the text. But we need to specify most influential: the beautiful paper of A.M. Vinogradov and J.S. Krasil'shchik \cite{Vinogradov1975} "What is the hamiltonian formalism?", lectures of P. Cartier "Some fundamental techniques in the theory of integrable systems" \cite{Cartier}, a short book of  K. H. Bhaskara and K. Viswanath "Poisson algebras and Poisson manifolds" \cite{Bhaskara1988PoissonAA}.

\subsubsection{Acknowledgment} 

V. Rubtsov is thankful to the Institut des Hautes Études Scientifiques Université Paris-Saclay for hospitality during the last stages of this work. R. Suchánek is grateful to the University of Angers and to the Masaryk University for their hospitality. The work of R. Suchánek was financially supported under the project GAČR EXPRO GX19-28628X and by the Barrande Fellowship program organized by The French Institute in Prague (IFP) and the Czech Ministry of Education, Youth and Sports (MYES). Authors are also grateful to the organizers of the Winter School and Workshop Wisla 20-21 for a wonderful online event, organized during a difficult time of a covid pandemic.

\newpage
\section{Motivation}

 \subsection{Lagrangian and Hamiltonian mechanics}

A fundamental discovery of Lagrange: a \emph{Lagrangian} function $L \colon TM \to \Rbb$
    \begin{align*}
        L = T - V \ ,
    \end{align*}
describes the motion of a particle. The equation of motion is described by the Euler-Lagrange equation
    \begin{align*}
        \frac{d}{dt} (\frac{\del L}{\del \Dot{q^i}}) = \frac{\del L}{\del q^i} \ . 
    \end{align*}
Here $(\Bar{q}) = (q^1, \ldots, q^n)$ are coordinates on $M$, $(\Bar{q}, \Bar{\Dot{q}}) = (q^1, \ldots, q^n,\Dot{q}^1, \ldots , \Dot{q}^n )$ are the induced coordinates on the tangent bundle $TM$. The above Euler-Lagrange equation follows from the variational principle 
    \begin{align*}
        \delta \int L(\Bar{q}, \Bar{\Dot{q}}) dt = 0 \ , 
    \end{align*}
where $\delta$ (in this context) refers to the variational derivative. Using the Legendre transformation, one can pass from the tangent bundle to the cotangent bundle $T^*M$ with the coordinates $(\Bar{q}, \Bar{p})$, and reformulate the Euler-Lagrange equations equivalently by the Hamiltons equations
    \begin{align}\label{Hamilton equations}
          \Dot{\Bar{q}} = \frac{\del H}{\del \Bar{p}},  \ \ \ 
          \Dot{\Bar{p}} = - \frac{\del H}{\del \Bar{q}} 
    \end{align}
The link between Lagrangian and Hamiltonian is given by 
    \begin{align*}
        H (\Bar{q},\Bar{p})= <\Bar{p}, \Bar{\Dot{q}}> - L (\Bar{q}, \Bar{\Dot{q}}) \ .
    \end{align*}


\subsection{Hamiltonian mechanics and Poisson brackets}

Instead of starting with the Lagrangian $L$ function, one can start with the Hamiltonian function $H \colon T^*M \to \Rbb$ and build the mechanics independently of the Lagrangian. We now go to the simpler setup of flat spaces, i.e. consider $M \cong \Rbb^N$. We will get back to the setup of smooth manifolds later on.     

Let $F,G$ be two differentiable functions on $\Rbb^{2n}$ with coordinates $(\Bar{q}, \Bar{p})$, where $\Bar{p} = (p_1, \dots, p_n), \Bar{q} = (q_1, \dots, q_n)$. Introduce a new functions as the result of the following skew-symmetric operation on $F$ and $G$. 
    \begin{align}
        \{ F,G\} : = \sum_{k = 1}^n (\frac{\del F}{\del q_k} \frac{\del G}{\del p_k} - \frac{\del G}{\del q_k} \frac{\del F}{\del p_k}) = \frac{\del F}{\del \Bar{q}} \frac{\del G}{\del \Bar{p}} - \frac{\del G}{\del \Bar{q}} \frac{\del F}{\del \Bar{p}} \in C^{\infty}(\Rbb^{2n}) \  .
    \end{align}
This operations allows to write the Hamilton's equations \eqref{Hamilton equations} associated with a~Hamiltonian function $H = H(\Bar{q},\Bar{p})$ in the following manner (putting $\Bar{q}$ and $\Bar{p}$ on an equal footing):
    \begin{align}\label{Hamilton system}
        \Dot{\Bar{q}} = \{ \Bar{q}, H \} \ , \ \ \
        \Dot{\Bar{p}} = \{ \Bar{p}, H \} \ . 
    \end{align}
Indeed we have 
    \begin{align*}
        \{ q_i , H \} = \sum_{k = 1}^n (\frac{\del q_i}{\del q_k} \frac{\del H}{\del p_k} - \frac{\del H}{\del q_k} \frac{\del q_i}{\del p_k} ) = \frac{\del H}{\del p_i} \ , \\
        \{ p_i , H \} = \sum_{k = 1}^n ( \frac{\del p_i}{\del q_k} \frac{\del H}{\del p_k} - \frac{\del H}{\del q_k} \frac{\del p_i}{\del p_k} ) = - \frac{\del H}{\del q_i} \ ,
    \end{align*}
since $\frac{\del q_i}{\del q_k} = \frac{\del p_i}{\del p_k} = \delta^i_k$ and $\frac{\del q_i}{\del p_k} = \frac{\del p_i}{\del q_k }= 0$.

Note that for each solution $\{ (q(t), p(t)\}$ of the Hamilton system \eqref{Hamilton system} and for any $F \in C^{\infty}(\Rbb^{2n})$
    \begin{align*}
        \frac{d}{d t} F (\Bar{q}(t), \Bar{p}(t)) & : = \sum_{k=1}^n (\frac{\del F}{ \del q_i} \Dot{q_i} + \frac{\del F}{ \del p_i} \Dot{p_i} ) \\
            & = \sum_{k=1}^n (\frac{\del F}{ \del q_i} \frac{\del H}{\del p_i} - \frac{\del F}{ \del p_i} \frac{\del H}{\del q_i} ) \\
            & = \{ F, H \} (\Bar{q}(t), \Bar{p}(t))  \ . 
    \end{align*}
In particular, for the Hamiltonian function $H$,
    \begin{align*}
        \frac{d}{d t} H (\Bar{q}(t), \Bar{p}(t)) = \{ H,H\}  (\Bar{q}(t), \Bar{p}(t)) = 0 \ , 
    \end{align*}
which is the \emph{conservation law} for $H$ (i.e. $ H (\Bar{q}(t), \Bar{p}(t))$ is constant along the trajectories). 
D. Poisson had observed that if $\{ F,H\}$ and $\{ G,H\}$ vanish, then $\{ \{F, G\}, H\}$ vanishes (Poisson bracket of two constants of motion is again a~constant of motion). This statement can be explained abstractly with the help of the \emph{Jacobi identity}
    \begin{align}\label{Jacobi identity}
        \{ \{F, G\}, H\} + \{\{G, H\}, F\} + \{ \{H, F\}, G\} = 0  \ ,
    \end{align}
since if $\{ F,H\} = \{ G,H\} = 0$, then the above equality gives $\{ \{F, G\}, H\} = 0$. Now let $\pi_{ij}$ be arbitrary smooth functions on $\Rbb^d$, $x_1, \dots, x_d$ cooridnates on $\Rbb^d$. Define 
    \begin{align}\label{to be a bracket}
        \{ F,G\} : =  \sum_{k=1}^d \pi_{ij} \frac{\del F}{\del x_i}\frac{\del G}{\del x_j}  \ .
    \end{align}
Then \eqref{to be a bracket} satisfies the Jacobi identity \eqref{Jacobi identity} if and only if
    \begin{align}
         \sum_{k=1}^d (\pi_{lk} \frac{\del \pi_{ij}}{\del x_l} + \pi_{li} \frac{\del \pi_{jk}}{\del x_l}  + \pi_{lj} \frac{\del \pi_{ki}}{\del x_l} ) = 0 \  . 
    \end{align}
Skew-symmetry and the Jacobi identity imply that the operation $\{ \cdot, \cdot \}$ defines a \emph{Lie algebra} structure on $C^{\infty}(\Rbb^d)$. Moreover this structure is compatible with the product operation $\cdot$ on $C^{\infty}(\Rbb^d), U \subset \Rbb^d$, meaning that it satisfies the \emph{Leibniz rule}
    \begin{align*}
        \{ FG, H\} = F\{ G, H\} + \{ F, H\}G  \ .
    \end{align*}

\section{Poisson algebras}    
There are two ways to think about a Poisson structure (given by a Poisson bracket) on a manifold $M$ (algebraic, smooth, analytic, etc.). 

\textbf{A geometric viewpoint}. To each function $H$ (smooth, holomorphic, etc.) on $M$, one can associate a vector field $X_H$, where $H$ is the Hamiltonian in the mechanical interpretation. The vector field is given by 
    \begin{align}
        X_H : = \{ - , H\}  \  .
    \end{align}

\textbf{An algebraic viewpoint}. Consider a vector space (typically infinite dimensional) $\A$, with two algebra structures: 
    \begin{enumerate}
        \item Structure of an associative and commutative algebra with a unit.
        \item A Lie algebra structure. 
    \end{enumerate}
\emph{The Poisson structure:} commutative structure + Lie algebra structure + compatibility conditions. 


Extracting the algebraic aspects of the above ideas, one is lead to the following definition.

\begin{definition}[Poisson algebra] Let $(\A, \cdot)$ be an associative, unital and commutative $\Kbb$-algebra ($\operatorname{char} \Kbb = 0$), with the unit denoted $1$. Let 
    \begin{align*}
        \{-,-\} \colon \A \times \A \to \A
    \end{align*}
be a Lie bracket defining a $\Kbb$-Lie algebra structure on $\A$. Then $\A$ is called \emph{a~Poisson algebra} if the operations $\cdot$ and $\{-,-\}$ are compatible in the following sense
    \begin{align}\label{Poisson alg - compatib cond of two structures}
        \{ a \cdot b, c\} = a \cdot \{ b, c\} + \{ a, c\} \cdot b \ .
    \end{align}
The Lie bracket is then called \emph{a Poisson bracket}.
\end{definition}

When the context is clear, we will usually omit the $\cdot$ symbol of the associative operation in $\A$.

\begin{remark}
Note that the Lie bracket $\{-,-\}$ in the above definition is not necessarily given by the commutator. In fact, if it is given by the commutator, i.e. $\{a,b\} = a \cdot b - b \cdot a$, then $\{a,b\} = 0$ for all $a,b \in \A$, since the $\cdot$ is assumed to be a commutative operation on $\A$. 
\end{remark}


\subsection{Subalgebras and ideals}

\begin{definition}[Subalgebras] A vector subspaces $\B \subset \A$ is \emph{a Poisson subalgebra} in a Poisson algebra $(\A, \{-,-\}, \cdot)$ if $\B \cdot \B subset \B$ and $\{ \B, \B \} \subset \B$.
\end{definition}

\begin{definition}[Ideals] A vector subspaces $\J \subset \A$ is \emph{a Poisson ideal} of a~Poisson algebra $(\A, \{-,-\}, \cdot)$ if $\J \cdot \A \subset \J$ and $\{ \J, \A \} \subset \J$.
\end{definition}

\begin{observation}
In the first definition, $\B$ is itself a Poisson algebra with operations given by restriction of operations in $\A$. The inclusion map $\iota \colon \B \subset \A$ is a Poisson algebra morphism. In the second definition, the quotient algebra $\A / \J$ inherits a Poisson bracket from $\A$ by the requirement that the canonical projection on the quotient $p \colon \A \to \A/\J$ is a~Poisson morphism. 
\end{observation}

\begin{remark}
For a fixed field $\Kbb$ ($\operatorname{char}\Kbb = 0)$, $\Poiss_{\Kbb}$ denotes a category, whose objects are the Poisson $\Kbb$-algebras, and whose morphisms are $\Kbb$-morphisms of Poisson algebras. Group object in $\Poiss_{\Kbb}$ is the Poisson-Lie group $(G, \cdot, \{-,-\})$ s.t. the bracket on $G$ is \emph{multiplicative}, i.e. the multiplication $\cdot G \times G \to G$ is a~Poisson morphism, where $G \times G $ is considered with a Poisson bracket given by $\{ -,- \}_{G \times G}: = \{ -,- \}_G \oplus \{ -,- \}_G$.
\end{remark}

\begin{definition}[Prime]
A Poisson ideal $\mathcal{P} \subset \A$ is \emph{Poisson prime}, if for all Poisson ideals $\I, \J \subset \A$
    \begin{align*}
        \I \J \subset \mathcal{P} \Rightarrow \I \subset \mathcal{P} \text{ or } \J \subset \mathcal{P}
    \end{align*}
\end{definition}

\begin{remark}
If $\A$ is \emph{Noetherian}\footnote{Given an increasing sequence of ideals in $\A$, $\I_1 \subseteq \I_2 \subseteq \ldots $, there always exists $k \in \Nbb$ such that $\I_k = \I_{k+n}  $ for all $n \in \Nbb$.}, then the above definition is equivalent to $\mathcal{P}$ being both prime ideal and Poisson ideal.
\end{remark}

\begin{definition}[Spectrum]
The \emph{Poisson spectrum of} $\A$, denoted $\PSpec(\A)$, is the set of Poisson prime ideals of $\A$.  
\end{definition}

\begin{definition}[Maximal ideal]
A maximal ideal $\mathcal{M} \subset \A$ is a \emph{Poisson maximal ideal} if it is a Poisson ideal. 
\end{definition}

\begin{remark}
A Poisson maximal ideal is not the same thing as a maximal Poisson ideal.
\end{remark}

\begin{definition}[Core]
The \emph{Poisson core} of an ideal $\I \subset \A$, denoted $\PC(\I)$, is the largest Poisson ideal of $\A$ contained in $\I$. 
\end{definition}

\begin{remark}
If $\I$ is a prime ideal of $\A$, then $\PC(\I)$ is Poisson prime.
\end{remark}

\begin{definition}[Primitive ideal]
A Poisson ideal $\mathcal{P}$ of $\A$ is a \emph{Poisson primitive} if $\mathcal{P}$ is the Poisson core of some maximal ideal $\mathcal{M}$, i.e. $\PC( \mathcal{M} ) = \mathcal{P} $
\end{definition}

\begin{remark}
Each Poisson primitive ideal is Poisson prime. 
\end{remark}


\begin{definition}[Localization]
Let $R$ be a commutative ring and $S \subset R$ a multiplicative subset\footnote{$s_1,s_2 \in S \Rightarrow s_1 s_2 \in S$}. A \emph{localization of $R$ at $S$}, denoted $S^{-1}R$, is the commutative ring of equivalence classes in $R \times S$ defined by 
    \begin{align*}
        (\Tilde{r},\Tilde{s}) \in [r,s]  \iff \exists k \in S: (\Tilde{r}s - r\Tilde{s}) k = 0 \ .
    \end{align*}
One usually denotes the equivalence class $[r,s]$ by $\frac{r}{s}$ or by $rs^{-1}$ (this reflects the similarity with the construction of the field of fractions). The ring operations are given by 
    \begin{align*}
        \frac{\tilde{r}}{\tilde{s}} + \frac{r}{s} & : = \frac{\tilde{r}s + r\tilde{s}}{\tilde{s}s} \\
          \frac{\tilde{r}}{\tilde{s}} \cdot \frac{r}{s} & : = \frac{\tilde{r}r}{\tilde{s}s}
    \end{align*}
\end{definition}

\begin{exercise}
Show that the above construction indeed yields a ring structure on the set $S^{-1}R$. Describe the identity elements with respect to both ring operations. 
\end{exercise}


If $S \subset \A$ is a multiplicatively closed subset of a Poisson algebra $\A$, then the localization $S^{-1}\A$ is also a Poisson algebra with the brackets 
    \begin{align*}
        \{ as^{-1}, bt^{-1} \} : = (st\{a,b\} - sb\{a,t\} - at\{s,b\} + ab\{s,t\})(s^2t^2)^{-1},
    \end{align*}
for all $a,b \in \A$ and all $s,t, \in S$. For any Poisson ideal $\mathcal{P} \subset \A$, the localization $S^{-1} \mathcal{P} $ is a Poisson ideal in $S^{-1}\A$. For two multiplicative sets $S,T \subset \A$ such that $S \subset T$, there is a Poisson morphism $ \colon S^{-1}\A \to T^{-1}\A$.

\begin{exercise}
Let $s,t \in \A$ and consider the multiplicative subsets of $\A$
    \begin{align*}
        S: = \{ s^n| \ n \in \Nbb\} && T: = \{ t^n| \ n \in \Nbb\} \ ,
    \end{align*}
where $s,t \in \A$ are such that $t = s u$ for apropriate $u \in \A$. Show that the map $\varphi_{s,t} \colon  S^{-1}\A \to T^{-1}\A$ defined by  $\varphi_{s,t} (\frac{a}{s^m}) : = \frac{au^n}{t^n}$ is independent of $u$ and is a~Poisson morphism. 
\end{exercise}

\begin{example}
Any Poisson maximal ideal is maximal Poisson, but the converse is not true. A countre-example is the Weyl-Poisson algebra $\Cbb [x,y]$, with $\{ x,y\} = 1$. This is a simple Poisson algebra but not a simple associative algebra (consider the trivial ideal $\{0 \}$).
\end{example}

\subsection{Morphisms and derivations}

\begin{definition}[Morphisms and isomorphisms] Let $(\A_i, \{ -,-\}_i, \cdot_i)_{i = 1,2}$ be two Poisson algebras over $\Kbb$ and let $\varphi \colon \A_1 \to \A_2$ be a $\Kbb$-linear map satisfying for all $a,b \in \A_1$
    \begin{enumerate}
        \item $\varphi (a \cdot_1 b) = \varphi (a) \cdot_2 \varphi (b)$,
        \item $\varphi (\{a,b\}_1) = \{\varphi (a) , \varphi (b)\}_2$
    \end{enumerate}
Then $\varphi$ is called \emph{a morphism of Poisson algebras}. If $\varphi$ is a bijective morphism of Poisson algebras s.t. the inverse map $\varphi^{-1} \colon \A_2 \to \A_1$ is also a Poisson algebra morphism, then $\varphi$ is called \emph{an isomorphism of Poisson algebras}. 
\end{definition}


\begin{definition}[Derivation]
A $\Kbb$-linear map $\varphi \colon \A \to \A$ is called \emph{a $\Kbb$-derivation} on $\A$ if
    \begin{align}
        \varphi(a\cdot b) = a \varphi(b) + \varphi(a) \cdot b \ .
    \end{align}
The set of all derivations of $\A$ is denoted by $\mathcal{D}(\A)$, $\mathcal{D}^1(\A)$ or $\mathcal{D}(\A, \A)$. A bilinear map $\varphi \colon \A \times \A \to \A$ is called \emph{a bi-derivation} if 
    \begin{equation}
        \begin{gathered}
            \varphi (ab, c) = a \varphi (b,c) + b \varphi (a,c), \\
            \varphi (a, bc) = \varphi (a,b) c + b \varphi (a,c) 
        \end{gathered}
    \end{equation}    
\end{definition}

\begin{observation}
Directly from the definition, $\mathcal{D}(\A) \subset \End_{\Kbb} (\A) = \Hom_{\Kbb} (\A, \A)$. The $\Kbb$-linear map, given for every $b \in \A$, $a \mapsto \{ a,b\}$ is a~derivation on $\A$.
\end{observation}

\begin{observation}
The biderivation associated with derivation on $\A$ given by a~Poisson bracket is always skew-symmetric.
\end{observation}

From now on, we will often abbreviate the notation and refer to a Poisson algebra $(\A, \{-,-\}, \cdot)$ simply by $\A$, if the context is clear. Also, all the vector spaces, algebras (Poisson, Lie), and the corresponding linear maps (morphisms, derivations), will be considered over a general field $\Kbb$, with $\operatorname{char} \Kbb = 0$, unless stated otherwise.

\section{Hamiltonian derivations and Casimirs}

\begin{definition}[Hamiltonian derivations and Casimirs]\label{Hamiltonian derivations and Casimirs}
Let $\A$ be a Poisson algebra and $a \in \A$. The derivation $X_a : = \{ - , a \} \in \mathcal{D}(\A)$ is called a Hamiltonian derivation with a Hamiltonian (or a Hamilton function) $a$ associated to $X_a$. We denote by
    \begin{align}\label{the set of Hamiltonian derivations}
        \Ham(\A) : = \{ X_a \ |\  a \in \A\} 
    \end{align}
the set of all Hamiltonian derivations. We have a linear map $\A \to \Ham(\A)$, $a \mapsto X_a$. Let $a \in \A$ be s.t. $X_a = 0$, then $a$ is called a \emph{Casimir element}. We denote by
    \begin{align}\label{Casimir elements}
        \Cas(\A ) : = \{ a \in \A| X_a = 0\}
    \end{align}
the set of all Casimir elements. 
\end{definition}

\begin{definition}\label{Poisson center}
The (Poisson) \emph{center of $\A$} is $\operatorname{Z}_{P}(\A) : = \{ a \in \A \ |\  \{a,b\} = 0 \text{ for all } b \in \A \}$. We have $ \Cas(\A) = \operatorname{Z}_{P}(\A)$. 
\end{definition}

\begin{definition}
\emph{Augmentation} of an (associative) $\Kbb$-algebra $\A$ is a $\Kbb$-algebra homomorpshism $\alpha \colon \A \to \Kbb$. The pair $(\A, \alpha)$ is called an \emph{augmented algebra}. The kernel $\ker \alpha$ is called the augmentation ideal of $\A$.
\end{definition}

\begin{lemma}
For arbitrary monic polynomial $p(x) \in \Cas[x]$ and arbitrary $b \in \A$
    \begin{align}\label{polynomial derivation-bracket formula}
        \{p(x), b \} = p^\prime(x) \{ x,b\} \ ,
    \end{align}
where $ p^\prime(x) = \frac{\del p (x)}{\del x}$.
\end{lemma}

\begin{proof}
By induction. Suppose $\tau : = \deg p = 0$. Then $p(x) = c_0$, where $ c_0 \in \Cas(\A)$, and hence for all $b \in \A$
    \[ 
        \{ p(x), b \} = \{ c_0, b \} = 0 =  p^\prime(x) \{ x,b\}  \ . 
    \]
Let the induction hypothesis hold for monic polynomials of $\deg \leq \tau - 1$ and suppose $\tau > 0$. Then $p(x) = x^{\tau} + c_{\tau-1} x{\tau-1} + \ldots + c_1 x + c_0$, where all $c_i \in \Cas (\A)$. We can rewrite $p(x) = x^{\tau} + c_{\tau-1} q(x)$, where $q(x)$ is monic and $\deg q = \tau - 1$. Then
    \begin{align*}
        \{ p(x), b \} & = \{ x x^{\tau-1} , b\} + c_{\tau-1} \{ q(x) , b\} \\
            & =  x^{\tau-1} \{ x , b\} + x \{ x^{\tau-1} , b\} + c_{\tau-1} q^\prime(x)\{ x , b\} \\
            & = (x^{\tau-1} + x (\tau-1) x^{\tau-2} + c_{\tau-1} q^\prime(x) ) \{ x , b\} \\
            & = p^\prime(x) \{ x,b\}  \ . \qedhere
    \end{align*}
\end{proof}

\begin{proposition}
Let $\A$ be a Poisson algebra. 
    \begin{enumerate}
        \item $\Cas (\A)$ is a subalgebra of $(\A )$. Moreover, $\Kbb$ can be naturally identified with a subset of $\Cas (\A)$.
        \item If $\A$ is an integral domain (i.e. has no zero divisors), then $\Cas (\A)$ is integrally closed in $\A$.
        \item[3] Not every representation of $\A$ on $\End(\A)$ defines an $\A$-module structure on $\Ham(\A) \subset \End(\A)$.
        \item $\Ham(\A)$ is a $\Cas(\A)$-module.
        \item The map $\varphi \colon \A \to \mathcal{D}(\A)$, defined by $H \mapsto - X_H$ is a morphism of Lie algebras. 
        \item There is a short exact sequence of Lie algebras 
            \[ 0 \to \Cas (\A) \to \A \to \Ham (\A) \to 0  \ . \]
    \end{enumerate}
\end{proposition}

\begin{proof}
    \begin{itemize}
        \item[1.] That $\Cas(\A)$ is a subalgebra of $\A$ is obvious. Now we show that scalars correspond to Casimir elements. Notice that since $\A$ is an algebra with the unit, $1$, and because $\Kbb$ has an action on $\A$ (denoted also by juxtaposition), there is a natural injective morphism $\iota \colon \Kbb \to \A$, given by $k \mapsto k1$. This means that we can identify $\Kbb$ with $\iota (\Kbb) = \Kbb 1 \subset \A$. We will simply write $\iota(k) = k$.\footnote{In this notation, $k a$ can be interpreted equivalently as action of $\Kbb$ on $\A$, or as a multiplication in $\A$ after identifying $\Kbb$ with $\iota(\Kbb)$. Of course we have $k 1 = 1  k = k$.} Now consider arbitrary $a \in \A$ and $k \in \Kbb$. Then we have 
            \[ 
                \{ a, k1\} = \{a , k\} 1 + k\{ a, 1\} ,
            \]
        which implies that $k \{ a, 1\} = 0$ (otherwise $\{ a, k 1\} \neq \{a , k\}$) and hence $\{ a, 1\} = 0$.\footnote{since for $a \neq 0: k a = 0 \Rightarrow k = 0$} But if $\{ a, 1\} = 0$ then the $\Kbb$-bilinearity of $\{-,-\}$ implies $k \{ a, 1\} = \{ a, k\} = 0$. So arbitrary $k \in \Kbb \cdot 1$ satisfies $ \{ a, k\} = 0$ for all $a \in \A$, hence $\Kbb 1 \subset \Cas (\A)$.
        \item[2.] Suppose $a$ is an integral element over $\Cas(\A)$, i.e. $a \in \A$ and there is a~monic polynomial $p(x) \in \Cas(\A) [x]$ s.t. $p(a) = 0$. We need to show that $a \in \Cas(\A)$. Without loss of generality, suppose $p$ is the smallest degree polynomial s.t. $p(a) = 0$. If $\deg p = 1$, then $p(x) = x-a$. Hence $a \in \Kbb \subset \Cas(\A)$ by the previous statement. Suppose $\tau : = \deg p > 1$. Using \eqref{polynomial derivation-bracket formula}, for any $b \in \A$ we have
            \[
                0 = \{ p(a), b\} = p^\prime (a) \{a,b \}  \ .
            \]
        Now $p^\prime (a) \neq 0$, otherwise there exists $k \in \Kbb$ s.t. $\frac{p^\prime}{k}$ is a monic polynomial, which satisfies $\frac{p^\prime}{k}(a) = 0$ and $\deg (\frac{p^\prime}{k}) = \deg p^\prime = \tau - 1$. But this would be a contradiction with the assumption of $p$ being the smallest degree monic polynomial with the property $p(a) = 0$. Because $\A$ has no zero divisors, $\{a,b \} = 0$. As $b \in \A$ is arbitrary, this shows that $a \in \Cas(\A)$.
        \item[3] Consider the left action $l \colon \A \to \End (\End(\A)), x \mapsto l_x$, where $l_x(a) $. Suppose $l$ defines an $A$-module structure on $\Ham(\A)$. Then for any $x \in \A$ and arbitrary $X_a \in \Ham (\A)$, we have $l_x X_a \in  \Ham (\A)$. This means there is $d \in \A$ s.t. for every $b \in \A: l_x X_a (b) = X_d (b)$. But this is not true, in general, since
            \[
                l_x X_a (b) = x\{ b,a\} = \{ b,xa\} - \{b,x\}a  \ ,
            \]
        while 
            \[
                X_d (b) = \{ b, d\} \  .    
            \]
        In other words,
            \begin{align}\label{intermediate step 1}
                 (l_x X_a - X_d ) (b) = 0 \iff d =  xa \text{ and } \{b,x\}a  = 0 \ ,
            \end{align}
        for every $a,b,x \in \A$, which is obviously not always the case. Notice though, if we assume that $x \in \Cas (\A)$, then $ \{b,x\} = 0$ for every $b \in \A$ and thus \eqref{intermediate step 1} can always be satisfied. This shows that the action of $\A$ on $\End(\A)$, if restricted to $\Cas (\A)$, defines an action on $\Ham(\A)$, i.e. statement 4 follows.
        \item[4.] See the proof of statement $3$.
        \item[5.] By definition, $ \varphi(a) (c) = - X_a (c) = - \{ c, a\}$. We want to show that $\varphi(\{a ,b \}) = [\varphi(a), \varphi(b)]$, where \[ [\varphi(a), \varphi(b)] = [X_a, X_b] := X_a \circ X_b - X_b \circ X_a \ . \] From the Jacobi identity \eqref{Jacobi identity}, we have for arbitrary $a,b,c \in \A$ 
        \[ - \{ c , \{ a,b \} \} = \{ \{ c,b \}, a \} - \{  \{ c,a \}, b \} = X_a (X_b (c) ) - X_b( X_a(c) )  \ . \]
        Since $\varphi(\{a ,b \}) = - X_{\{ a,b\}} = - \{ c , \{ a,b \} \} $, this shows that $\varphi$ is a Lie algebra morphism.
        \item[6.] The first map is an inclusion, $\iota \colon \Cas (\A) \to \A $, which is obviously a Lie algebra morphism. The second map is $\varphi \colon \A \to \Ham (\A)$, and is given by $a \mapsto - X_a$. The previous statement shows that it is a Lie algebra morphism. Now consider $a \in \A$. Then $X_a = 0 $ iff $ \forall b \in \A: X_a (b) = \{ b, a\} = 0$. Thus $\ker \varphi = \im \iota $, meaning that the short sequence 
           \begin{equation*}
                \begin{tikzcd}
                0 \arrow{r}{0} & \Cas (\A) \arrow{r}{\iota} & \A \arrow{r}{\varphi} &  \Ham (\A) \arrow{r}{0} & 0
                \end{tikzcd}
            \end{equation*}  
        is exact. \qedhere
    \end{itemize}
\end{proof}

\begin{remark}
Corollary of the fifth statement of the above proposition is that $\Ham(\A) \subset \mathcal{D}(\A)$ is a Lie subalgebra. 
\end{remark}



\subsection{Exterior algebra of a commutative algebra} 

\emph{In the following constructions, $\A$ is always commutative, associative, unital $\Kbb$-algebra with the unit $1$.} 

\textbf{Module of Kahler differentials of $\A$}. Let $\Omega_{\A / \Kbb} (\A)$ be a $\Kbb$-vector space generated by elements of the form $b\du (a)$, for all $a,b \in \A$, satisfying relations
    \[ 
        \du(ab) = \du(a) b + a\du(b) \ . 
    \]
where $\du\colon \A \to \Omega_{\A / \Kbb} (\A)$ is an $\A$-linear map called the \emph{universal derivation}. More precisely, for any derivation $\del \in \mathcal{D}(\A)$, there is a unique $\A$-linear map $\tilde{\del} \colon \Omega_{\A / \Kbb} (\A) \to \A$, such that the following diagram commutes 
    \begin{equation}\label{universal derivative}
        \begin{tikzcd}
        A \arrow{r}{\du} \arrow[swap]{dr}{\del} & \Omega_{\A / \Kbb} (\A) \arrow[dashed]{d}{\exists! \tilde{\del}}  \\
        & \A
        \end{tikzcd}
    \end{equation}  
The $\A$-module structure on $\Omega_{\A / \Kbb} (\A)$ is given by multiplication from the left. Using the universal property of $(\Omega_{\A / \Kbb} (\A), \du)$, we have a canonical $\A$-module isomorphism 
    \begin{align}\label{duality between derivations and Kahler differentials}
        \mathcal{D}(\A) \cong \Hom_{\A} (\Omega_{\A / \Kbb} (\A), \A) \ . 
    \end{align}
This $\A$-isomorphism can be understood as duality between $\mathcal{D}(\A)$ and $\Omega_{\A / \Kbb} (\A)$. Moreover, it can be extended to polyderivations and exterior forms on the algebra $\A$. 


\textbf{Exterior algebra of $\A$}. Let $\Omega^\bullet (\A)$ be the (graded) exterior algebra of the vector space of Kahler differentials of $\A$
    \begin{align*}
        \Omega^\bullet (\A) : = \oplus_{k \geq 0} \Omega^k (\A) : = \oplus_{k \geq 0} \Lambda^k (\Omega_{\A/\Kbb} (\A) ) ,
    \end{align*}
together with the universal derivative $\du\colon \A \to \Omega^1 (\A)$, which extends to
    \begin{align*}
        \du\colon \Omega^k (\A) \to \Omega^{k+1}(\A)
    \end{align*}
for all $k$. The universal derivative satisfies all the properties of the exterior differential on forms. In particular, $\du$ is a boundary operator
    \begin{align*}
        \du^2  = 0 \ , 
    \end{align*}
and acts with respect to the wedge product as
\begin{align*}
    \du(\omega_1 \wedge \omega_2) = \du\omega_1 \wedge \omega_2 + (-1)^{k} \omega_1 \wedge \du \omega_2 \ ,
\end{align*}
where $\omega_1 \in \Omega^k (\A)$ and $\omega_2 \in \Omega^\bullet (\A)$ are arbitrary. The $\A$-duality \eqref{duality between derivations and Kahler differentials} between $\D(\A)$ and $\Omega_{\A /\Kbb}$ can be written as 
    \begin{align*}
        \mathcal{D}(\A) \times \Omega_{\A / \Kbb} (\A) \to \A \ ,\\
        <x, \du a> : = (\du a) (X) : = X(a) \in \A \ ,
    \end{align*}
and may be extended to the duality 
    \[
        \Lambda^k(\mathcal{D}(\A) ) \cong \Hom_{\A} (\Omega^k (\A), \A) \  . 
    \]
Moreover, one can define the symmetric algebra of derivations as 
    \[ 
    S(\mathcal{D}(\A)) : = \Hom_{\A} (S ( \Omega_{\A / \Kbb} (\A) , \A) \ , 
    \]
which, as was already shown, is a graded Poisson algebra of degree $1$. 

Take $k = 2$, then we have 
    \begin{align}\label{inter. duality}
        \Lambda^2(\mathcal{D}(\A) ) \cong \Hom_{\A} (\Omega^2 (\A), \A) \  . 
    \end{align}
Suppose now that $\A$ is a Poisson algebra equipped with a Poisson bracket
    \begin{align*}
        \{-,-\} \colon \A \times \A \to \A \ .
    \end{align*}
Then the above duality \eqref{inter. duality} provides an element $\pi \in \Lambda^2 (\D(\A))$ by the formula
    \begin{align}\label{Poisson biderivation}
        <\pi, \du a \wedge \du b > = \{ a,b\} \ ,
    \end{align}
for all $a,b \in \A$. Then $\pi$ is called a \emph{Poisson biderivation}. The Jacobi identity for the Poisson bracket $\{-,- \}$ imposes an important constraint on the biderivation $\pi$, which we will discuss below (see \eqref{equation for the Poisson biderivation}.


\section{Homology and cohomology}

We will discuss two types of (co)homologies that can be defined for an algebra $\A$. Firstly we define (co)homology of commutative, associative, unitary algebras (Hochschild), then for Poisson algebras (Lichnerowicz-Poisson). 

\subsection{Hochschild (co)homology}
We use the following notation $\otimes = \otimes_{\Kbb}$ and $\A^{\otimes k} = \underbrace{\A \otimes \ldots \otimes\A}_{\text{$k$-times}}$. We define $\A^{\otimes 0} : = \Kbb$.

\textbf{Hochschild chain complex}. Let $\A$ be a commutative, associative, unitary $\Kbb$-algebra, and let $\mathcal{M}$ be an $\A$-bimodule\footnote{That is, $\A$ acts on $\mathcal{M}$ from left and right.}. Then there is the following chain complex of $\A$-bimodules $\mathcal{M} \otimes \A^{\otimes k}, k \in \Nbb$
   \begin{equation}\label{Hochschild chain complex}
        \begin{tikzcd}
        0 & \mathcal{M} \arrow[swap]{l}{0} & \mathcal{M} \otimes \A  \arrow[swap]{l}{d_1} & \mathcal{M} \otimes \A^{2} \arrow[swap]{l}{d_2} & \arrow[swap]{l}{d_3} \ldots  
        \end{tikzcd} 
    \end{equation}  
The boundary operators ${d_k}, k \in \Nbb, k > 0$ are defined via the "face maps" $\del_i$ as
    \begin{align*}
        d_k  : = \sum_{i = 0}^k (-1)^i \del_i \ , 
    \end{align*}
where the maps $\del_i $ are 
    \begin{align*}
        \del_i (m \otimes a_1 \otimes \ldots \otimes a_k) : = 
            \begin{cases}
            m a_1 \otimes \ldots \otimes a_k , & \text{if } i = 0 \\
            m \otimes a_1 \otimes \ldots \otimes a_i a_{i+1} \otimes \ldots \otimes a_k , & \text{if } 0 < i < k \\
            a_k m \otimes a_1 \otimes \ldots \otimes a_{k-1} , & \text{if } i = k
            \end{cases}
    \end{align*}
for all $m \in \mathbf{M}$ and $a_i \in \A $. This is the \emph{Hochschild chain complex}. 

\begin{exercise}
Check that the above defined boundary maps $d_k$ are 
    \begin{enumerate}
        \item $\Kbb$-multilinear, 
        \item well-defined,
        \item satisfy ${d_k}^2 = 0$
    \end{enumerate}
for all $k \in \Nbb_0$.     
\end{exercise}

\begin{definition}[Hochschild homology]
The homology of the chain complex \eqref{Hochschild chain complex}, denoted $HH_* (\A, \mathcal{M})$, is called the Hochschild homology of $\A$ with coefficients in $\mathcal{M}$. The $k$-th Hochschild homology $\Kbb$-module is 
    \begin{align*}
        HH_k (\A, \mathcal{M}) : = \ker d_k / im_{d_{k+1}} \ . 
    \end{align*}
\end{definition}


\textbf{Hochschild cochain complex}. 
Using the $\Hom_{\Kbb}$-functor, we can construct the Hochschild cochain complex 
   \begin{equation}\label{Hochschild cochain complex}
        \begin{tikzcd}
        0  \arrow{r}{0} &  \mathcal{M} \arrow{r}{d^1} & \Hom_{\Kbb} (\A, \mathcal{M}) \arrow{r}{d^2} & \Hom_{\Kbb} (\A^{\otimes 2}, \mathcal{M}) \arrow{r}{d^3} & \ldots  
        \end{tikzcd}
    \end{equation}  
The coboundary operators ${d^k}$ are defined as
    \begin{align*}
        d^k  : = \sum_{i = 0}^k (-1)^i \del^i \ , 
    \end{align*}
where the maps $\del^i$ are defined for $f \in \Hom_{\Kbb} (\A^{\otimes k}, \mathcal{M})$
    \begin{align*}
        (\del^i f) (a_0 \otimes a_1 \otimes \ldots \otimes a_k) : = 
            \begin{cases}
            a_0 f( a_1 \otimes \ldots \otimes a_k) , & \text{if } i = 0 \\
            f(a_0 \otimes a_1 \otimes \ldots \otimes a_i a_{i+1} \otimes \ldots \otimes a_k) , & \text{if } 0 < i < k \\
            f (a_0 \otimes a_1 \otimes \ldots \otimes a_{k-1})a_k , & \text{if } i = k
            \end{cases}
    \end{align*}
where $a_i \in \A $ for all $i$. This is the \emph{Hochschild cochain complex}.

\begin{definition}[Hochschild cohomology]
The cohomology of the cochain complex \eqref{Hochschild cochain complex}, denoted $HH^* (\A, \mathcal{M})$, is called the Hochschild cohomology of $\A$ with coefficients in $\mathcal{M}$. The $k$-th Hochschild cohomology $\Kbb$-module is 
    \begin{align*}
        HH^k (\A, \mathcal{M}) : = \ker d^{k} / \im d^{k-1} \ . 
    \end{align*}
\end{definition}


\begin{proposition}
The $0$-th Hochschild homology of $\A$ with coefficients in $\mathcal{M}$ satisfy
    \begin{align*}
        HH_0 (\A, \mathcal{M}) & = \mathcal{M} / [\mathcal{M}, \A] \ ,
    \end{align*}
where $[\A, \mathcal{M}] = \{ am - ma| \ a \in \A, m \in \mathcal{M} \}$. In particular 
     \begin{align*}
          HH_0 (\A, \A) = \A / [\A,\A] \ .
     \end{align*}
The $0$-th Hochschild cohomology of $\A$ with coefficients in $\mathcal{M}$ satisfy
    \begin{align*}
        HH^0 (\mathcal{M}, \A) & = \{ m \in M| \ am = ma , \ \forall a \in \A \} \ .
    \end{align*}
In particular,  
    \begin{align*}
        HH^0(\A, \A) = \operatorname{Z}(A) \ ,
    \end{align*}
where $\operatorname{Z}(A)$ is the center of $\A$. 
\end{proposition}

\begin{proposition}
Let $\A$ be a commutative $\Kbb$-algebra and $\mathcal{M}$ an $\A$-bimodule. Then the $1$st Hochschild homology satisfies
    \begin{align*}
        HH_1(\A, \mathcal{M} ) \cong M \otimes_{\A} \Omega_{\A / \Kbb} (\A) \ . 
    \end{align*}
\end{proposition}

Denote by $\D(\A, \mathcal{M})$ the space of $\Kbb$-linear functions $f \colon \A \to \mathcal{M} $ such that
    \begin{align*}
        f (ab) = af(b) + f(a)b \ . 
    \end{align*}
Denote by $\mathcal{PD} (\A, \mathcal{M})$ the space of $\Kbb$-linear functions $f_m \colon \A \to \mathcal{M} $, which are given by 
    \begin{align*}
        f_m (a) = ma - am \ .
    \end{align*}

\begin{exercise}
Check that the above defined $f_m$ satisfies $f_m \in \D(\A, \mathcal{M})$. 
\end{exercise}

\begin{proposition}
$HH^1(\A, \mathcal{M}) = \D(\A, \mathcal{M}) / \mathcal{PD} (\A, \mathcal{M}) $. 
\end{proposition}

\subsection{Lichnerowicz-Poisson cohomology}

Let $\A$ be a commutative Poisson algebra with the Poisson bracket $\{-,-\}$ and consider $\Lambda( \D(\A)) = \bigoplus_{k \geq 0} \Lambda^k( \D(\A))$ defined as follows. For $k = 0$, we define $\Lambda^0 (\D(\A)) = \A$. An element $X \in \Lambda^k( \D(\A) ), k > 0$, is a multilinear, antisymmetric mapping
    \begin{align*}
        X \colon \A^k \to \A 
    \end{align*}

and the mapping $\del_X \colon \A \to \A $ given by
    \begin{align*}
        \del_X (a) : = X(a, a_1, \ldots, a_{k-1}) 
    \end{align*}
is a derivation. 
    

\textbf{Schouten-Nijenhuis bracket}. Let $X,Y \in \Lambda(\D(\A))$ be decomposable,
    \begin{align*}
        X = x_1 \wedge \ldots \wedge x_k && Y = y_1 \wedge \ldots \wedge y_l \ ,
    \end{align*}
where all $x_i, y_j \in \D(\A)$. The Schouten-Nijenhuis bracket $ [\![-  , - ]\!]$ is given by
    \begin{align*}
        [\![X , Y ]\!] : =  \sum_{i,j} (-1)^{i+j} [x_i, y_j] x_1 \wedge \ldots \wedge \hat{x_i} \wedge \ldots \wedge x_k \wedge y_1 \wedge \ldots \wedge \hat{y_j} \wedge \ldots \wedge y_l \ ,
    \end{align*}
where $[-,-]$ is the commutator of differential operators, and the above definition is extended linearly on the whole $\Lambda(\D(\A))$. 

\begin{exercise}
Show that the Schouten-Nijenhuis bracket $ [\![-  , - ]\!]$ satisfies for all $P,Q \in \Lambda(\D(\A))$
    \begin{align}\label{graded anticom. of Schouten-Nijenhuis}
         [\![P  , Q ]\!] = (-1)^{\deg P \deg Q}   [\![Q  , P]\!]
    \end{align}
\end{exercise}

\textbf{Poisson operator}. Let $\pi \in \Lambda^2 (\D(\A))$ be the \emph{Poisson biderivation}, defined by \eqref{Poisson biderivation}. Then the Schouten-Nijenhuis bracket gives an element $[\![\pi , \pi ]\!] \in \Lambda^3 (\D(\A))$. 

\begin{proposition}
The Jacobi identity for the Poisson bracket \eqref{Jacobi identity of the Poisson bracket} implies the triviality of the $3$-derivation: 
    \begin{align}\label{equation for the Poisson biderivation}
        [\![\pi , \pi ]\!] = 0 \ .
    \end{align}
\end{proposition}

\begin{exercise}
Prove the above proposition.  
\end{exercise}

The equation \eqref{equation for the Poisson biderivation} is sometimes called \emph{Poisson Master equation}. It can be interpreted as a nilpotency condition (${\delta_\pi}^2 = 0$) for the operator $\delta_\pi \colon \Lambda (\D(\A)) \to \Lambda (\D(\A))$, defined by 
    \begin{align*}
        \delta_\pi X  = [\![\pi , X ]\!] \ . 
    \end{align*}

The operator $\delta_\pi$ is called \emph{Lichnerowicz-Poisson operator} and leads to the following notion of cohomology.

\begin{definition}[Lichnerowicz-Poisson cohomology]
The cohomology of the chain complex 
     \begin{equation*}
        \begin{tikzcd}
        \ldots \arrow{r}{\delta_\pi} & \Lambda^{k-1}(\D(\A)) \arrow{r}{\delta_\pi} & \Lambda^{k}( \D(\A)) \arrow{r}{\delta_\pi} & \Lambda^{k+1}(\D(\A)) \arrow{r}{\delta_\pi} & \ldots \ , 
        \end{tikzcd}
    \end{equation*} 
is called the Lichnerowicz-Poisson cohomology, denoted $HP^*(\A)$.    
\end{definition}

In terms of the Poisson bracket, the Lichnerowicz-Poisson operator can be rewritten as
    \begin{align}
       \left( \delta_\pi X \right) (a_0, a_1, \ldots, a_k) &  = \sum_{i} (-1)^i \{a_i , X (a_0, a_1, \ldots, \hat{a_i}, \ldots, a_k) \} \\
       & \ \ \ + \sum_{0 \leq i < j} X(\{a_i, a_j \} , a_1, \ldots , \hat{a_i}, \ldots, \hat{a_j}, \ldots,a_k) \ .
    \end{align}


\subsection{Low-dimensional Poisson cohomology}
 
 $k = 0$: the operator $\delta_\pi \colon \A \to \D(\A)) \cong \Lambda^1(\D(\A))$ acts as $ a \mapsto \del_a = \{ - , a\} $, where $a\in \A$ can be seen as a Hamiltonian for a Hamiltonian derivation $\{ - , a\}$. We have
    \begin{align*}
        HP^0 \cong \Cas (\A)
    \end{align*}
since $\delta_\pi (a) (b) = 0 $ for all $b \in \A$ if and only if $a \in \Cas (\A)$. Elements of $ HP^0 \cong \Cas (\A)$ are called Hamiltonians with zero dynamics. 

$k = 1$: $1$-coboundary is a derivation $X \in \D(\A)$ which is a Hamiltonian derivation $\del_a$ for some element $a \in \A$. $1$-cocycle is an element  $X \in \D(\A)$ s.t. 
    \begin{align*}
        \delta_\pi (X) = 0 \Rightarrow  [\![\pi , X ]\!] = 0 \ .
    \end{align*}
Derivations $X \in \D(\A)$ which satisfies $ \delta_\pi (X) = 0$ are called \emph{Poisson, or canonical, derivations}. The set of all such elements is denoted $ \operatorname{Can} (\A)$. The equation $ [\![\pi , X ]\!] = 0 $ represents a conservation of the Poisson structure along $X$. If we denote $L_X \pi = [\![ \pi, X ]\!]$, one can easily show that
    \begin{align*}
        < L_X \pi, \du a \wedge \du b > = L_X <\pi ,\du a \wedge \du b > \ .
    \end{align*}

\begin{theorem}[basic theorem of classical mechanics]
    \begin{align*}
        HP^1 (\A) \cong \operatorname{Can} (\A) / \Ham(\A) \ , 
    \end{align*}
where $\Ham(\A)$ is given by \eqref{the set of Hamiltonian derivations}.
\end{theorem}

\begin{example}[Hamiltonian derivations on polynomial algebra \cite{Vinogradov1975}]
Consider $\A = \Kbb[x_1,x_2]$ with the Poisson bracket $\{ x_1,x_2 \} = x_2$. Then $\del \in \Ham (\A) \subset \D(\A) $ if
    \begin{align}\label{interm. def. of del}
        \del = \alpha_1 \frac{\del }{\del x_1} + \alpha_2 \frac{\del }{\del x_2}
    \end{align}
and exists $h \in \A$ such that 
    \begin{align*}
        \del = \del_h = \{ - , h \} \ . 
    \end{align*}
We will compute the coefficients $\alpha_i$ so that the above equation holds. We have
    \begin{align*}
        \{ x_1 , h \} = \frac{\del x_1 }{\del x_1}  \frac{\del h }{\del x_2} -  \frac{\del x_1 }{\del x_2} \frac{\del h }{\del x_1} = \frac{\del h }{\del x_2} \ .
    \end{align*}
Similarly we get 
    \begin{align*}
        \{ x_2, h \} = - \frac{\del h }{\del x_1} \ . 
    \end{align*}
For $f \in \A$, the bracket with $h$ is
    \begin{align*}
         \{ f , h \} =  \frac{\del f }{\del x_1}  \{ x_1 , h \} +  \frac{\del f }{\del x_2}  \{ x_2 , h \} = \frac{\del f_1 }{\del x_1}  \frac{\del h }{\del x_2} -  \frac{\del f_1 }{\del x_2} \frac{\del h }{\del x_1}  \ .
    \end{align*}
From \eqref{interm. def. of del} we have
    \begin{align*}
        \del (f) = \alpha_1 \frac{\del f}{\del x_1} + \alpha_2 \frac{\del f}{\del x_2} \ , 
    \end{align*}
so to have $\del = \del_h$, one must have
    \begin{align*}
        \alpha_1 = \frac{\del h }{\del x_2} \ , && \alpha_2 = - \frac{\del h }{\del x_1} \ .
    \end{align*}
This means that the set of Hamiltonian derivations on $\A$ is given by 
    \begin{align*}
        \Ham (\A ) = \{ \frac{\del h }{\del x_2}\frac{\del }{\del x_1} - \frac{\del h }{\del x_1} \frac{\del }{\del x_2}| \ h \in \A \} \ .
    \end{align*}
\end{example}

\begin{exercise}
Consider $\A$ from the previous example. Show that
    \begin{align*}
        \operatorname{Can} (\A) / \Ham (\A) \cong \Kbb
    \end{align*}

\noindent Indication: check that if $\del = \alpha_1  \frac{\del}{\del x_1} + \alpha_2  \frac{\del}{\del x_2}  \in \operatorname{Can}(\A) $, then
    \begin{align*}
        \alpha_1 & = c + x_2 \frac{\del h}{\del x_2} \\
        \alpha_2 & = - x_2 \frac{\del h}{\del x_1}
    \end{align*}
where $c\in \Kbb$.
\end{exercise}

\begin{example}
Consider the set
    \[ 
        \Sym (\Hom_{\Kbb} (\D^p(\A), \A) ) = \{ \varphi \colon \D^p(\A) \to \A \ | \ \varphi  \text{ multilinear} \} \ .
    \]
An element of the above set is not only $\Kbb$-linear in all $p$-arguments but also a $\Kbb$-derivation in every argument. Recall that $\mathcal{D}(\A)$ is the $\A$-module of $\Kbb$-derivations of $\A$. 
\end{example}

\subsubsection{Compatible Poisson structures}

As a useful application of Lichnerowicz-Poisson cohomology, consider the following description of \emph{compatible Poisson structures}.

\begin{definition}
Two Poisson structures given by biderivations $\pi$ and $\theta$ are compatible if for all $\lambda \in \Kbb$, $\pi + \lambda \theta $ is again a Poisson structure. 
\end{definition}

As an immediate consequence of the above definition, we have the following proposition.
\begin{proposition}\label{prop: two compatible poisson structures}
If two Poisson structures $\pi , \theta \in \Lambda^2 \left( \D(\A) \right)$ are compatible, then 
    \begin{align*}
        [\![\pi , \theta ]\!] = 0
    \end{align*}
Equivalently, if $\pi$ and $\theta$ are compatible, then they are closed with respect to coderivations in the following sense
    \begin{align*}
        \delta_\pi \theta = \delta_\theta \pi = 0 \ .
    \end{align*}
\end{proposition}

\begin{exercise}
Prove the proposition \ref{prop: two compatible poisson structures}. 
\end{exercise}

Notice that if $[\![\pi , \theta ]\!] = 0$ then for all $\lambda \in \Kbb$ 
    \begin{align*}
        [\![\pi + \lambda \theta, \pi + \lambda \theta ]\!] = 0 \ . 
    \end{align*}

Let us fix a biderivation $\pi$. Consider the \emph{Lie derivative} of $\pi$ along $X \in \D (\A)$ 
    \begin{align*}
        L_X \pi : = [\![\pi , X ]\!] \in \Lambda^2 \left( \D(\A) \right) \ .
    \end{align*}
Let $\pi$ be a Poisson biderivation (i.e. it corresponds to a Poisson structure). Then $L_X \pi$ is called an \emph{infinitesimal deformation of the Poisson structure}. Now consider a deformation of $\pi$ of the form
    \begin{align*}
        \pi \mapsto \pi + L_X \pi \ . 
    \end{align*}
By assumption, we have $[\![\pi , \pi ]\!] = 0$ since $\pi$ yields a Poisson structure. If we assume that the deformation also yields a Poisson structure, then the graded antisymmetry of the bracket (see \eqref{graded anticom. of Schouten-Nijenhuis} yields
    \begin{align*}
          0 = [\![\pi +  L_X \pi, \pi +  L_X \pi ]\!] = 2 [\![\pi, L_X \pi ]\!]  + [\![ L_X \pi, L_X \pi ]\!] \ . 
    \end{align*}
Let us denote $\gamma: = L_X \pi$ and recall that $[\![ \theta, X ]\!] = \delta_\theta (X) $. Then the above can be written as 
    \begin{align*}
        \delta_\pi \gamma + \frac{1}{2} [\![ \gamma, \gamma ]\!] = 0 \ . 
    \end{align*}
This equation is called \emph{Maurer-Cartan equation} (for a differential graded algebra $ \Lambda \left( \D(\A) \right)$). 

\subsubsection{Interpretation of $HP^2 (\A)$}

Let $\pi \in \Lambda^2 \left(\D (\A) \right) $ and consider the image of $\delta_\pi \colon \Lambda^1 \left(\D(\A) \right) \to \Lambda^2 \left(\D(\A) \right)$, given by 
    \begin{align*}
        {B_\pi}^2 = \{ \theta \in \Lambda^2 \left(\D (\A) \right) | \ \theta = [\![\pi, X ]\!] \text{ for some $X \in \D (\A)$} \} \ , 
    \end{align*}
The kernel of $\delta_\pi \colon \Lambda^2 \left(\D(\A) \right) \to \Lambda^3 \left(\D(\A) \right)$ is 
    \begin{align*}
        {Z_\pi}^2 = \{ \theta \in \Lambda^2 \left(\D (\A) \right) | \ [\![\pi, \theta ]\!]  = 0 \!] \} \ .
    \end{align*}
By definition 
    \begin{align*}
        HP^2 (\A) = {Z_\pi}^2  /  {B_\pi}^2 \ .
    \end{align*}
We have the following proposition, which yields an interpretation of the second Lichnerowicz-Poisson cohomology. 

\begin{proposition}[\cite{Krasil'shchik1988}]
If $\pi \in \Lambda^2 \left(\D (\A) \right)$ is a Poisson biderivation ( $[\![\pi, \pi]\!] = 0$) and $HP^2 (\A) = 0$, then the set of all structures compatible with $\pi$
    \begin{align}\label{defn. of Comm}
        \operatorname{Comm}(\pi) : = \{ \theta \in \Lambda^2 \left(\D (\A) \right) | \  [\![\theta, \theta ]\!] = [\![\pi,\theta]\!] = 0 \}
    \end{align}
is a set of infinitesimal deformations of $\pi$ along $X \in \D (\A)$. That is, $\theta = L_X \pi$ such that 
    \begin{align*}
        {L_X}^2 (\theta ) = L_Y \pi
    \end{align*}
for some $Y \in \D(\A)$, where ${L_X}^2 = L_X \circ L_X$.
\end{proposition}

\begin{proof}
Let $\theta \in \operatorname{Comm}(\pi)$, then $\delta_\pi \theta = 0$ so 
    \begin{align*}
        \operatorname{Comm}(\pi) \subset \ker \delta_\pi \ .
    \end{align*}
If $\theta = \delta_\pi X$ then $0 = [\theta] \in HP^2(\A)$. Suppose $\theta \in \operatorname{Comm}(\pi)$ and $HP^2(\A) = 0$. Then $\theta = \delta_\pi X = [\![\pi, X ]\!] = L_X \pi$. But at the same time $[\![\theta, \theta ]\!] = [\![L_X \pi, L_X \pi ]\!] = 0$. This is equivalent to $(L_X \circ L_X) (\theta) \in \ker \delta_\pi$ since
    \begin{align*}
        L_X \left(L_X \pi  \right) = L_X \left( [\![\pi, X ]\!] \right)  = [\![ [\![\pi, X ]\!] , X  ]\!] = [\![ \theta , X ]\!] \ , 
    \end{align*}
and 
    \begin{align*}
        \delta_\pi \left( [\![ \theta , X ]\!] \right)  = \delta_\pi \left( \delta_\theta X \right) = - \delta_\theta \left( \delta_\pi X \right) =  [\![ \delta_\pi X , \delta_\pi X ]\!] = 0 \ . 
    \end{align*}
Hence if $HP^2 (\A) = 0$ then there is $Y \in \D (\A) $ such that ${L_X}^2 \theta = L_Y (\pi) = \delta_\pi (Y)$.
\end{proof}

From the above, we obtain a mapping
    \begin{align*}
        \tau \colon \operatorname{Comm} (\pi ) \to HP^2(\A) 
    \end{align*}
and 
    \begin{align*}
        \tau (\theta) = 0 \iff \theta = \delta_\pi (X)  
    \end{align*}
for some $X \in \D (\A)$. The image of $\tau$ is described by the following proposition. 

\begin{proposition}
Let $\del \in HP^2 (\A)$. Then
    \begin{align*}
        \del \in \im \tau \iff \del = \Delta + \delta_\pi \left( \D (\A) \right) \ , 
    \end{align*}
where the representative $\Delta \in \Lambda^2 \left( \D (\A) \right)$ of the class $\del$ satisfies 
    \begin{enumerate}
        \item  $[\![ \del , \del ]\!] = [\![ \Delta , \Delta ]\!] = \del_\pi (\alpha)$, where $\alpha \in \Lambda^2 \left( \D (\A) \right) $, 
        \item there is $X \in \D(\A)$ such that $\alpha + 2 L_X (\Delta) - {L_X}^2(\pi) = \ker \delta_\pi$. 
    \end{enumerate}
\end{proposition}

\begin{proof}
Let $\Delta \in \Lambda^2 \left( \D (\A) \right)$ be such that 
    \begin{align*}
        \del = \Delta + \delta_\pi \left( \D (\A) \right)  \in HP^2(\A) \ .
    \end{align*}
If $\del \in \im \tau$, then there is a $\theta \in \operatorname{Comm}(\pi)$ (see \eqref{defn. of Comm}) such that 
    \begin{align*}
        \tau (\theta) = \Delta + \delta_\pi \left( \D (\A) \right) \ . 
    \end{align*}
Because
    \begin{align*}
        0 =  [\![ \del , \del ]\!] \Rightarrow  [\![ \Delta , \Delta ]\!] = \del_\pi (\alpha) \in HP^3(\A)
    \end{align*}
and $\tau [\![ \theta , \theta ]\!] = 0 $, we have
    \begin{align*}
         [\![ \tau(\theta) , \tau(\theta) ]\!] & = [\![ \Delta + \delta_\pi (X)  ,  \Delta + \delta_\pi (X)  ]\!] \\
         & =  [\![ \Delta, \Delta ]\!] + [\![ \Delta, \delta_\pi (X) ]\!] + [\![ \delta_\pi (X), \Delta ]\!] +  [\![ \delta_\pi (X) , \delta_\pi (X) ]\!] \\
         & = \delta_\pi (\alpha) - [\![ \Delta, L_X \pi ]\!] - [\![ L_X \pi , \Delta]\!] - \delta_\pi {L_X}^2 (\pi) \\
         & = \delta_\pi \left( \alpha -  {L_X}^2 (\pi) \right) - 2  [\![ L_X \pi , \Delta]\!] \\
         & =  \delta_\pi \left( \alpha -  {L_X}^2 (\pi) \right) - 2  [\![  [\![\pi, X ]\!], \Delta ]\!] \\
         & = \delta_\pi \left( \alpha + 2 L_X (\Delta) - {L_X}^2(\pi) \right) = 0 \ ,
    \end{align*}
since $[\![ \tau(\theta) , \tau(\theta) ]\!] = 0$.    
\end{proof}

\subsection{Poisson homology}

The following construction is due to Brylinski \cite{Brylinski88} and \cite{Koszul1985}. Let $\A$ be a commutative $\Kbb$-algebra and $\Omega (\A) = \Lambda (\Omega_{\A/\Kbb})$ be the exterior algebra of $\A$. The boundary morphism $d_\pi \colon \Omega^k (\A) \to \Omega^{k-1} (\A) $ is defined by the "homotopy-like" formula 
    \begin{align}\label{boundary operator d_pi}
        d_\pi \omega = (i_\pi \circ \du) (\omega) - (\du \circ i_\pi) (\omega) \ ,
    \end{align}
where $i_\pi$ is the contraction operator with respecet to $\pi$, i.e .$i_\pi : = <\pi ,- > $. It is straightforward to check that this operator satisfies $ d_\pi \circ d_\pi = 0$. In coordinates, we can describe $d_\pi$ on a decomposable $\omega = a_0, \du a_1 \wedge \ldots \wedge a_k $ by the general formula
    \begin{align*}
        d_\pi \omega & = \sum_{l = 1}^{k+1}(-1)^{l+1} \{ a_0, a_i \} \wedge \du a_1 \wedge  \ldots \wedge \hat{\du a}_l\wedge \ldots \wedge \du a_{k+1} \\
            &\ \ \ - \sum_{i < j} (-1)^{i+j} \du \left(\{ a_i, a_j\} \right) \wedge \du a_1 \wedge  \ldots \wedge \hat{\du a}_i \wedge \ldots \wedge \hat{\du a}_j \wedge \ldots \wedge \du a_{k+1} \ .
    \end{align*}
By definition, 
    \begin{align*}
        d_\pi (a_0 da_1) = i_\pi (da_0 \wedge da_1) = \{a_0, a_1 \} \ .
    \end{align*}
Note that the Jacobi identity for the triple $a_0, a_1, a_2$
    \begin{align*}
        \{ a_0, \{ a_1, a_2 \} \} + \{ a_1, \{ a_2, a_O \} \} + \{ a_2, \{ a_0, a_1 \} \} = 0 
    \end{align*}
implies $(d_\pi)^2 = 0$.


\subsection{Duality}

Let $\pi$ be a \emph{symplectic} (or \emph{non-degenerate}) Poisson structure on $\A$, meaning that the Hamiltonian map
    \begin{align*}
        \Gamma_\pi \colon \Omega^1_{\A /\Kbb} \to \D (\A) 
    \end{align*}
defined by
    \begin{align*}
        \Gamma_\pi (\alpha) & = <\pi , \alpha> \in \D (\A)
    \end{align*}
is an isomorphism, and there exists the inverse ${\Gamma_\pi}^{-1} \colon \D (\A) \to \Omega^1 (\A)$. The inverse is given by
    \begin{align*}
        {{\Gamma}_\pi}^{-1} (\del) & = \alpha_ {\del} \text{ such that } <\pi , \alpha_{\del}> = \del \ .
    \end{align*}
    
    
In this case, one can check that
    \begin{align*}
        \delta_\pi = \Gamma_\pi \circ \du \circ {\Gamma_\pi}^{-1}  \ , 
    \end{align*}
where $\delta_\pi$ is the Lichnerowicz-Poisson operator. The above equation is symbolical. It can be described more precisely by the following commutative diagram 
   \begin{equation*}
        \begin{tikzcd}
            \Omega^1 (\A)  \arrow[swap]{d}{\du} \arrow{r}{\Gamma_\pi} & \Lambda^1 \left(\D(\A) \right)  \arrow{d}{\delta_\pi} \\ 
            \Omega^2 (\A) \arrow{r}{ \Lambda^2\Gamma_\pi} & \Lambda^2 \left(\D(\A) \right) 
        \end{tikzcd}
    \end{equation*}
which holds for any Poisson biderivation $\pi$.  

\begin{proposition}\label{duality proposition}
For general $k$, the following diagram commutes
    \begin{equation*}
        \begin{tikzcd}
            \Omega^k (\A)   \arrow[swap]{d}{\du} \arrow{r}{\Lambda^k \Gamma_\pi} & \Lambda^k \left(\D(\A) \right)  \arrow{d}{\delta_\pi} \\ 
            \Omega^{k+1} (\A) \arrow{r}{\Lambda^{k+1} \Gamma_\pi } & \Lambda^{k+1} \left(\D(\A )\right) 
        \end{tikzcd}
    \end{equation*}
\end{proposition}

\begin{proof}
To check the commutativity of the diagram, consider decomposable $\omega = a_0 \du a_1 \wedge \ldots \wedge  \du a_k \in \Omega^k (\A) $. Then  
    \begin{align*}
        \Lambda^{k+1} \Gamma_\pi (\du \omega) =  \Lambda^{k+1} \Gamma_\pi (\du a_0 \wedge \du a_1 \wedge \ldots \wedge  \du a_k ) = \del_{a_0} \wedge \del_{a_1} \wedge \ldots \wedge \del_{a_k} \ , 
    \end{align*}
and 
    \begin{align*}
        \delta_\pi (\Lambda^k \Gamma_\pi (\omega) ) & =  \delta_\pi (a_0 \del_{a_1} \wedge \ldots \wedge \del_{a_k} ) \\
         & = - [\![ a_0 \del_{a_1} \wedge \ldots \wedge \del_{a_k}, \pi ]\!] \\
         & = - [\![ a_0 , \pi ]\!] \wedge \del_{a_1} \wedge \ldots \wedge \del_{a_k} \\
         & = \del_{a_0} \wedge \del_{a_1} \wedge \ldots \wedge \del_{a_k} \ .
    \end{align*}
Since the above can be extended linearly, the diagram commutes.    
\end{proof}

\begin{remark}
If $\A = C^\infty (M)$ is the algebra of smooth function on a smooth manifold $M$, then it was shown in \cite{Brylinski88} that
    \begin{align*}
        HP_k (\A) \cong HP^{2n-k} (\A) \ .
    \end{align*}
Moreover, by the proposition \ref{duality proposition}, the latter group is isomorphic with the de Rham cohomology 
    \begin{align*}
        HP^{2n-k} (\A) = H_{DR}^{2n-k} (M) \ .
    \end{align*}
\end{remark}

In the following lemma, we use the notion of a \emph{graded commutator} in a~graded algebra 
    \begin{align*}
        [\del_1 , \del_2] : = \del_1 \circ \del_2 - (-1)^{\deg \del_1 \deg \del_2} \del_2 \circ \del_1 \ . 
    \end{align*}
We have $\deg d_\pi = -1, \deg \du = 1, \deg i_\pi = - 2, L_\pi = 1$

\begin{lemma}
The operator $d_\pi \colon \Omega^k (\A) \to \Omega^{k-1} (\A) $ commutes in graded sense with $\du $ and $i_\pi$, i.e.
    \begin{align*}
        [d_\pi , \du ] & = d_\pi \circ \du + \du \circ d_\pi = 0 \ , \\
        [d_\pi , i_\pi ] & = d_\pi \circ i_\pi - i_\pi \circ d_\pi = 0  \ .
    \end{align*}
\end{lemma}

\begin{proof}
To prove the first equation, we use the definition of $d_\pi$ (see \eqref{boundary operator d_pi}) and that $\du \circ \du = 0$. Then
    \begin{align*}
       (i_\pi \circ \du - \du \circ i_\pi )\circ \du + \du \circ (i_\pi \circ \du - \du \circ i_\pi ) = - \du \circ i_\pi \circ \du + \du \circ i_\pi \circ \du = 0
    \end{align*}
Hence $ [d_\pi , \du ] = 0$. Similarly for the second equation
    \begin{align*}
        [d_\pi , i_\pi ] = [L_\pi, i_\pi] = i_{[\pi, \pi]} = 0 \ .
    \end{align*}
\end{proof}

\begin{remark}
When $\A = C^\infty(M)$, where $M$ is a smooth Poisson manifold, we will see all the above, and more general, identities in the later section as well. 
\end{remark}

\begin{example}[$0$th Poisson homology]
For the kernel of $ d^0_\pi \colon \Omega^0 (\A) \cong \A \to 0 $ we obvously have $\ker d^0_\pi \cong \A$. For the image of $ d^1_\pi \colon \Omega^1 (\A) \to \Omega^0 (\A) \cong \A $ we have $\im d^1_\pi = \{ \A, \A \}$. This is because $\Omega^1 (\A)$ consists of element $a_0 da_1$, where $a_0, a_1 \in \A$ are arbitrary, and 
    \begin{align*}
        d^1_\pi (a_0 da_1) = \{a_0, a_1 \} \in \{ \A, \A\} \ .
    \end{align*}
Hence 
    \begin{align*}
        HP_0 (\A) \cong A / \{ \A, \A \}
    \end{align*}
There is no simple interpretation for higher homology groups, $k > 0$. 
\end{example}

\section{Polynomial Poisson Algebras}

Let $\A = \Cbb[x_1, \ldots, x_n] $ be the polynomial algebra over complex numbers. If $x_i$ is a generator of $\A$, then $\del_i : = \frac{\del}{\del x_i} \in \mathcal{D} (\B)$. The vector operator $\nabla : = (\del_1 , \ldots, \del_n) $ is called \emph{gradient} One can define the \emph{Jacobian matrix} of $n$ elements $f_1, \ldots, f_n \in \A$ as
    \begin{align*}
        \Jac(f_1, \ldots, f_n) : = \left(\frac{\del f_i}{\del x_j} \right) .
    \end{align*}
The determinant $\det ( \Jac(f_1, \ldots, f_n) )$ is called \emph{Jacobian}. It is clear the the $i$-th row of $ \Jac(f_1, \ldots, f_n)$ is $\nabla (f_i)$.

\subsection{Nambu-Jacobi-Poisson algebras}.

Now we fix $f_1, \ldots, f_{n-2} \in \A$ and define the following bilinear operation $\{ - , - \} \colon \A \times \A \to \A$, which yields a Poisson algebra structure on the polynomial algebra $\A$, called \emph{Nambu-Poisson-Jacobi structure}.
\begin{definition}
The \emph{Nambu-Jacobi-Poisson bracket} of $F,G \in \A$ is 
    \begin{align}\label{Nambu-Jacobi-Poisson bracket}
        \{ F,G \} : = \det \Jac(F,G, f_1, \ldots, f_{n-2}) \in \A \ .
    \end{align}
\end{definition}

When $\A = \Cbb[x_1,x_2,x_3]$, there is only one $f$ determining the Nambu-Jacobi-Poisson bracket, and we will denote the bracket by $\{-,-\}_f$. We proceed with the following elementary lemma.    
\begin{lemma}\label{lemma for Leibniz}
For $1 \leq i \leq n$ and $f_1, \ldots, f_{i-1}, f_{i+1}, \ldots, f_n \in \B$, the operation $D \colon \B \to \B$ given by $D(g) : = \Jac (f_1, \ldots, f_{i-1}, g, f_{i+1}, \ldots, f_n) $ is a derivation of $\B$.
\end{lemma}    

\begin{theorem}\label{Nambu bracket is a Poisson bracket}
The bracket \eqref{Nambu-Jacobi-Poisson bracket} is a Poisson bracket on $\A$. 
\end{theorem}

\begin{proof}
To prove this theorem for any $n$, we observe that skew symmetry is evident from the skew-symmetry of $\det$, the Leibniz rule follows from the lemma \ref{lemma for Leibniz}, so the only non-trivial statement is, as usual, the Jacobi identity. But this follows from the Fundamental identity for the Nambu bracket $\{ f_1, \ldots, f_n\} : = \Jac (f_1, \ldots, f_n)$ (see \eqref{Fundamental identity forthe Nambu bracket} below).
\end{proof}

An interesting property of the algebraic bracket structure on the polynomial algebra is given by the following theorem. 

\begin{proposition}
If $f_1, \ldots, f_{n-2}$ are algebraically dependent over $\Cbb$, then $\{-,-\} = 0$.
\end{proposition}

\begin{proof}
The proof is left as an easy exercise for the reader. 
\end{proof}

\begin{remark}
If we consider $F,G$ as rational functions, the result of bracket \eqref{Nambu-Jacobi-Poisson bracket} is still a polynomial. Consider a quotient by (convenient) ideal $\B : = \Cbb[x_1, \ldots, x_n] / <p_1, \ldots, p_k>$ where $ p_i \in \A$. Then the theorem \ref{Nambu bracket is a Poisson bracket} still ohlds, i.e. the bracket \eqref{Nambu-Jacobi-Poisson bracket} yields a Poisson algebra structure on $\B$. This is valid in an even more general setup of power series rings \cite{Odesskii-Rubtsov}.
\end{remark}

The Nambu-Jacobi-Poisson bracket is a special case of a $(n-m)$-ary operation
    \begin{align*}
        \{ F_1,\ldots, F_{n-m} \} : = \lambda \det \Jac(F_1,\ldots, F_{n-m}, f_1, \ldots, f_m) \in \A \ ,
    \end{align*}
where $\lambda, F_i \in \A$ for all $i$. The above multibracket $\{ -,\ldots , - \} \colon \A^{\otimes (n-m)} \to \A$ satisfies for every permutation $\sigma$ the antisymmetry condition
    \begin{align*}
        \{ F_1,\ldots, F_{n-m} \} = (-1)^\sigma \{ F_{\sigma(1)},\ldots, F_{\sigma(n-m)} \ .
    \end{align*}
It also satisfies the Leibniz rule in every argument
    \begin{align*}
         \{ hF_1,\ldots, F_{n-m} \} = F_1\{ h,\ldots, F_{n-m} \} + h\{ F_1,\ldots, F_{n-m} \} \ ,
    \end{align*}
and the \emph{Fundamental identity}
    \begin{align}\label{Fundamental identity forthe Nambu bracket}
         \{ F_1, \ldots, F_{n-m-1},\{G_1, \ldots , G_{n-m} \} \} & = \\ 
         = \sum_k \{ G_1, \ldots , G_{k-1},  \{ F_1, \ldots, F_{n-m-1}, G_k \}  , G_{k+1}, \ldots,  G_{n-m} \} \  , 
    \end{align}
which is a generalization of the Jacobi identity (the Jacobi identity and Nambu-Jacobi-Poisson bracket is restored in the case $n-m = 2$). For more details about the Nambu structures and their generalizations, see \cite{Odesskii-Rubtsov}.

Consider now a $(n-2)\times 2$ matrix over $\A$
    \begin{align*}
    M = \begin{pmatrix}
            a_{1,1} & \ldots  & a_{1,n} \\
            \vdots & & \vdots \\
            a_{n-2, 1} & \ldots & 0 a_{n-2,n}
        \end{pmatrix} .
    \end{align*}
Suppose $i \neq j$ and denote by $\hat{M}_{ij}$ the matrix given by deleting the $i$-th and $j$-th column (thus the result being $(n-2) \times (n-2)$ matrix. If we choose
    \begin{align*}
    M = \begin{pmatrix}
            \frac{\del f_1}{\del x_1} & \ldots  & \frac{\del f_1}{\del x_n} \\
            \vdots & & \vdots \\
            \frac{\del f_{n-2}}{\del x_1} & \ldots & \frac{\del f_{n-2}}{\del x_n}
        \end{pmatrix} ,
    \end{align*}
then for the generators of $\A$ we have
    \begin{align*}
        \{ x_i, x_j \} = (-1)^{i+j-1} \det \hat{M}_{ij} 
    \end{align*}
and we define $ \{ x_i, x_i \} = 0$.    


\begin{example}
Let $n = 3, f = \frac{1}{3}({x_1}^3 + {x_2}^3 + {x_3}^3) + \tau x_1 x_2 x_3$, where $\tau \in \Cbb$. Then $\{ x_i, x_j \} = \Jac(x_i, x_j, f) = z x_i x_j + {x_k}^2$, where $(i,j,k)$ is a permutation of $\{ 1,2,3 \}$. 
\end{example}

\begin{example}[Sklyanin elliptic Poisson brackets]
Let $n = 4$ and consider $f_1 = q_1(x_1, x_2, x_3, x_4), f_2 = q_2(x_1, x_2, x_3, x_4)$, where $q_1,q_2$ are \emph{quadratic} polynomials. Choosing $q_1 = {x_1}^2 + {x_2}^2 + {x_3}^2 + {x_4}^2$ and $q_2 = \alpha {x_2}^2 + \beta {x_3}^2 + \gamma {x_4}^2$ such that $\alpha \beta \gamma + \alpha + \beta + \gamma \neq 0$, we obtain the original Sklyanin-Poisson structure \cite{Sklyanin1982, Odesskii-Rubtsov, Smith94}.
\end{example}


\subsection{Poisson-Calabi-Yau algebra}

This is a Jacobian algebra $\A = \Cbb[x_1,x_2,x_3]$ with $f = -{x_1}^2 x_3$ \cite{Berger-Picherau-2014}. The Nambu-Jacobi-Poisson bracket \eqref{Nambu-Jacobi-Poisson bracket} is
    \begin{align*}
        \{x_1, x_2 \}_f = - {x_1}^2,  &&   \{x_2, x_3 \}_f = - 2 x_1 x_3, && \{x_1, x_3 \}_f = 0 \ . 
    \end{align*}
It is interesting that the algebra $\Omega^1 (\A) = \A <dx_1, dx_2, dx_3>$ (think of a free algebra over $A$) is also a Poisson algebra with 
    \begin{align*}
        \{dx_1, dx_2 \}_\Omega & = d \{x_1, x_2 \}_f = -2x_1 dx_1 \ , \\
        \{dx_2, dx_3 \}_\Omega & = d \{x_2, x_3 \}_f = -2x_3 dx_1 - 2x_1 dx_3 \ , \\
        \{dx_1, dx_3 \}_\Omega & = d \{x_1, x_3 \}_f = 0 \ .
    \end{align*}
There is a corresponding sequence \cite{Lichnerowicz77} 
    \begin{equation*}
        \begin{tikzcd}
        0  \arrow{r}{} & \Omega^0 (\A) \arrow{r}{\du} \arrow{d}{\cong} & \Omega^1 (\A) \arrow{r}{\du} \arrow{d}{\cong} & \Omega^2 (\A) \arrow{r}{\du} \arrow{d}{\cong} & \Omega^3 (\A)\arrow{r}{} \arrow{d}{\cong} &  0  \\
             0  & \Lambda^0 ((\mathcal{D}(\A)) \arrow[swap]{l}{} & \Lambda^1((\mathcal{D}(\A)) \arrow[swap]{l}{d_\pi}  & \Lambda^2 ((\mathcal{D}(\A)) \arrow[swap]{l}{d_\pi}  & \Lambda^3 ((\mathcal{D}(\A)) \arrow[swap]{l}{d_\pi} &  0 \arrow[swap]{l}{}
        \end{tikzcd}
    \end{equation*}
where $\du$ is the universal derivative (see \eqref{universal derivative}) and $d_\pi$ is given by \eqref{boundary operator d_pi}. The isomorphism are coming from the duality \eqref{duality between derivations and Kahler differentials}.

\subsubsection{Low-dimensional cohomology of the PCY algebra}.

Considering the above sequence, we firstly notice that $\Lambda^0 ((\mathcal{D}(\A)) \cong \A$ and $ \Lambda^1((\mathcal{D}(\A)) \cong  \mathcal{D}(\A)$, so that the first map amounts to mapping $\A \to \mathcal{D}(\A)$, and $\delta_\pi (a) = [\![\pi , a ]\!] = \del_a \in \Ham (\A)$ is a Hamiltonian derivation. Thus $\delta_\pi (a) = 0$ iff $a$ is a Casimir element (see def. \eqref{Casimir elements}) of $\A$ and we have
    \begin{align*}
        HP^0 (\A) \cong \Cas(\A) = <{x_1}^2 x_3>_{\A} \ ,
    \end{align*}
where $\Cas(\A)$ are the Cassimir elements of $\A$. For the first cohomology, consider $\del \in \mathcal{D}(\A)$, we have $\delta_\pi (\del) = [\![\pi , \del ]\!] = - [\![ \del , \pi ]\!] = \mathcal{L}_{\del} \pi$. So we see that $\del \in \ker \delta_\pi $ iff $\mathcal{L}_{\del} \pi = 0$, meaning that $\pi$ is invariant with respect to $\del$. We have already met these operators in the section in which we computed the low-dimensional Poisson cohomology for more general algebras: the set of such operators is denoted $\operatorname{Can}(\A)$ and $\del$ is called 
Poisson canonical. The first cohomology is 
    \begin{align*}
        HP^1 (\A) = \operatorname{Can}(\A) / \Ham(\A) \ .
    \end{align*}

\subsection{Dual Poisson complex}

Consider the chain complex
    \begin{equation*}
        \begin{tikzcd}
           \Lambda^3 ((\D(\A)) \arrow{r}{d_\pi}  & \Lambda^2((\mathcal{D}(\A)) \arrow{r}{d_\pi} & \Lambda^1 ((\mathcal{D}(\A)) \arrow{r}{d_\pi}  & \Lambda^0 ((\mathcal{D}(\A))  \ .
        \end{tikzcd}
    \end{equation*}
This complex was introduced by Brylinski in \cite{Brylinski88}. It has highly non-trivial (Poisson) homology. For example, the lowest homology is  
    \begin{align*}
        HP_0 (\A) \cong \A /<\{a,b\}_f|\ a,b \in \A>_{\A} \ , 
    \end{align*}
since $\Lambda^0 ((\mathcal{D}(\A)) = \A$ and the image of $\delta_\pi \colon \Lambda^1 ((\mathcal{D}(\A)) \to \Lambda^0 ((\mathcal{D}(\A))$ is given by $\delta_\pi (a db) = \{a,b \}_f$. Following the definition of the Nambu-Jacobi-Poisson bracket on $\A$ we get
    \begin{align*}
        \{a,b \}_f = 
           \det \begin{pmatrix}
                \del_1 a & \del_2 a & \del_3 a \\
                \del_1 b & \del_2 b & \del_3 b \\
                \del_1 f & \del_2 f & \del_3 f 
                \end{pmatrix} 
    \end{align*}
Writing $\nabla a = (\del_1 a,  \del_2 a ,\del_3 a)$, we can express $\{a,b \}_f $ as 
    \begin{align*}
       \{a,b \}_f = \nabla f \cdot (\nabla a \times \nabla b) \  ,
    \end{align*}
where $\cdot$ is the dot product and $\times$ is the vector product. More details on Poisson (co)homology of the Dual Poisson complex can be found in \cite{Poisson-Structures}.

\begin{exercise}
Describe all differentials in the complex above in terms of vector alanysis operations: $\nabla, \operatorname{curl}, \times, (-,-), \operatorname{div}$.
\end{exercise}

\textbf{Generalized SPDUNR Poisson algebra}. The following is the generalized Sklyanin-Painlevé-Dubrovin-Ugaglia-Nelson-Regge Poisson algebra: $\A_f= (\Cbb[x_1,x_1,x_3], \{-,-\}_f)$, where $\{-,-\}_f$ is the Jacobian Poisson-Nambu structure on $\Cbb^3$ (see \eqref{Nambu-Jacobi-Poisson bracket}), and $F,G \in \Cbb[x_1,x_1,x_3]$. Let $M_f$ be the zero locus of 
    \begin{align*}
        f = x_1 x_1 x_3 + \sum_{i = 1}^3 a_i {x_i}^3 - \sum_{i = 1}^3 \epsilon_i {x_i}^2 + \sum_{i = 1}^3 c_i x_i + \omega \ ,
    \end{align*}
where  $\epsilon_i \in \{ 0,1\}$ and $a_i, c_i, \omega \in \Cbb$. The bracket is given by 
    \begin{align*}
        \{ f, x_i\}_f : = 0, \text{for } i = 1,2,3 \ ,
    \end{align*}
and
    \begin{align*}
        \{ x_1, x_2\}_f : = x_1 x_2 + 3a_3 {x_3}^2 - 2 \epsilon_3 x_3 + b_3 \ ,
    \end{align*}
the result being cyclic in $(1,2,3)$ for other $x_i,x_j$. For a generic set of constraints, the bracket is nowhere vanishing on $M_f$.


\section{Graded Poisson algebras}

Let $\A$ be a Poisson algebra. We shall suppose that $\A$ is an associative \emph{gradded algebra}, that is, $\A$ contains a set of vector subspaces $(\A^k)_{k \in \Nbb_0}$ s.t. $A = \bigoplus_{k \in \Nbb_0} \A^k$ and $\A^k \cdot \A^l \subset \A^{k+l}$ for all $k,l \in \Nbb_0$. Moreover, we assume $\A_0 : = \Kbb$. 

\begin{definition} 
Let $d \in \Nbb_0$ be arbitrary. $\A$ is called a \emph{graded Poisson algebra of degree $d$} if $\forall a \in \A^k, b \in \A^l: \{ a,b\} \in \A^{k+l-d}$ (for $n < 0$ define $A^n = 0$). 
\end{definition}

The graded Poisson algebras can be constructed from non-commutative, associative, unital algebras $\U$, which are \emph{filtered}: 
    \[ 
        \U = \bigcup_{k \in \Nbb_0} \U_k, \text{ where }1 \in \Kbb = \U_0 \subset \U_1 \subset \dots \subset \U_k \subset \dots.
    \]
To every such algebra $\U$, we can define the associated graded algebra $\S = \gr (\U)$, $\S = \bigoplus_{k \in \Nbb_0} \S_k$, where $\S_k : = \U_k / U_{k-1}$, $k \geq 1$, and $\S_0 : = \U_0 = \Kbb$. We denote by $\gr_k \colon \U_k \to \S_k $ the canonical projections. Then $\U$ is a graded Poisson algebra of degree $d \geq 1$ if $uv - vu \in \U_{k+l-d}$, for all $u \in \U_k, v \in \U_l$ (define $\U_k = 0$ for $k < 0$). 

\begin{proposition}
The associated graded algebra $\S$ given by the graded algebra $\U$ of degree $d$ is a graded Poisson algebra of degree $d$. 
\end{proposition}

\begin{proof}
Since the canonical projection $\gr_k \colon \U_k \to S_k $ is a surjection, to each $a \in \S$ exists $k \in \Nbb_0$ so that $a \in \S_k$, thus there exists $u \in \U_k$ s.t. $a = \gr_k(u)$. Let $b \in \S_l$ and $v \in \U_l$ s.t. $b = gr_l (v)$. The product in $\S$ is defined by 
    \[ 
        a b = gr_{k+l} (uv) \in \S_{k+l} \ .
    \]
The product is well-defined, because if we pick different representatives, say $\tilde{a} = a + x, x \in \U_{k-1}$ and $\tilde{b} = b + y, y \in \U_{l-1}$, then 
    \[ 
        \tilde{a}\tilde{b} =  (a + x) ( b + y) = ab + \underbrace{ay + xb + xy}_{\in \U_{k+l-1}} = ab  \ .
    \]
The unit in $\S$ is the same as in $\U$, and the associativity of $\S$ is also inherited from the associativity of $\U$. So we see that $\S$ is a graded algebra. 
The Lie bracket is
    \[ 
        \{ a,b\} = \{ \gr_k (u), \gr_l(v) \}: = gr_{k+l-d}(uv - vu) 
    \]
and is of degree $d$. By a similar argument as for the product, $\{-,-\}$ is well-defined on $\S$. Since $ab \in \S_{k+l}$ and $ba \in \S_{l+k} = \S_{k+l}$, there exist $u,u^\prime \in \U_{k+l}$ s.t. $ab = \gr_{k+l} (u), ba = \gr_{k+l} (u^\prime) $ and thus $ab - ba = \gr_{k+l} (u - u^\prime) = 0$. This shows that the product in $\S$ is commutative, hence $\S$ is Poisson graded of degree $d$. 
\end{proof} 

\textbf{Module structures on $\A$.} Let $\A$ be a commutative, associative $\Kbb$-algebra with the unit $1$. $\End (\A) = \Hom(\A, \A)$ has two $\A$-module structures, \emph{left} and \emph{right}: $\forall a,x \in \A$
    \begin{align}\label{A-module stuctures on End(A)}
     l_a \varphi(x) : = a\varphi(x), & & r_a \varphi(x) : = \varphi(ax) \ . 
    \end{align}

\begin{lemma}\label{derivation property of delta}
For arbitrary $a \in A$, denote 
    \begin{align}\label{the delta derivation}
        \delta_a : = r_a - l_a .
    \end{align}
Then $\delta_a$ satisfies the Leibniz rule
    \[
        \delta_a ( \varphi \circ \psi) = \delta_a \varphi \circ \psi +  \varphi \circ \delta_a \psi \ , 
    \]
that is, $\delta_a \in \mathcal{D}( \End(\A) ) \subset \Hom (\End (\A), \End (\A) ) $.
\end{lemma}

\begin{proof}
For arbitrary $\varphi, \psi \in \End (\A)$ and $a, u \in \A$, using the definition \eqref{the delta derivation}, we have
    \[
        \delta_a (\varphi \circ \psi) (u) = \delta_a (\varphi (\psi (u) ) ) = \varphi (\psi (au) ) - a\varphi (\psi (u) ) \ .
     \]
On the other hand 
    \begin{align*}
        ( \delta_a \varphi \circ \psi +  \varphi \circ \delta_a \psi ) (u) & = \delta_a \varphi ( \psi (u) ) + \varphi ( \delta_a \psi (u) ) \\
            & =  \varphi ( a \psi (u) ) - a  \varphi ( \psi (u) ) + \varphi ( \psi (au) - a\psi (u) ) \\
            & = \varphi ( \psi (au) ) - a  \varphi ( \psi (u) )  \ . \qedhere 
    \end{align*}
\end{proof}

\begin{lemma}
$[\delta_a, \delta_b] = 0$ for all $a,b \in \A$, where $\delta_a$ is given by \eqref{the delta derivation}.
\end{lemma}

\begin{proof}
The proof is a straightforward computation and we leave it to the reader as an exercise.
\end{proof}

\subsection{Algebra of differential operators}

\begin{definition}\label{algebra of differential operators}
For all $k \in \Nbb_0$, we define
    \[ 
        \Diff_k (\A) : = \bigcap_{\underset{0 \leq i \leq k}{a_i \in \A}} \ker (\delta_{a_0} \circ \dots \delta_{a_i}) 
    \]
and 
    \[ 
        \Diff_* (\A) : = \bigcup_{k \geq 0}  \Diff_k (\A)  \ ,
    \]
which is an abelian group under the addition $+$. Then $ \Diff_* (\A)$ inherits the two $\A$-module structures \eqref{A-module stuctures on End(A)} of $\End (\A)$. We will write $\Diff_*^{(+)} (\A)$ to emphasize the bimodule structure. The elements of $\Diff_k (\A)$ will be called diffferential operators of order $\leq k$ on a commutative algebra $\A$. 
\end{definition}

\begin{observation}
Directly from the above definition we have that for all $k \in \Nbb_0: \Diff_{k-1} (\A) \subset \Diff_k (\A)$. 
\end{observation}

\begin{remark}
One can generalize the above definition to the case $\varphi \colon P \to Q$, where $P$ and $Q$ are projective, finitely generated $\A$-modules ($0 \leq k$) and $\varphi$ is an $\A$-module homomorphism. Then
    \[ 
        \Diff_k (P,Q) : = \{ \varphi \colon P \to Q \ |\  \delta_{a_0} \circ \dots \circ \delta_{a_k} (\varphi) = 0 \text{ for all } a_0, \dots , a_k \in \A \} \  .
    \]
In this notation, $\Diff_k(\A) = \Diff_k(\A, \A)$. 
\end{remark}

\begin{example}[Lie algebra structure on $\Diff_*^{(+)} (\A)$]\label{Lie algebra structure on Diff}
Consider $\End_{\Kbb} (\A)$, equipped with a Lie algebra structure given by the commutator 
    \[
        [\varphi, \psi] = \varphi \circ \psi - \psi \circ \varphi  \ . 
    \]
Using the derivation property from lemma \ref{derivation property of delta}, we have 
    \begin{align}\label{intermidiate computation 2}
        \delta_a [\varphi, \psi] = [ \delta_a \varphi, \psi] + [\varphi, \delta_a \psi]
    \end{align}
for all $a \in \A$ and $\varphi \in \End_{\Kbb} (\A)$. Hence $\delta_a$ acts as a derivation on the commutator. Suppose that $\varphi, \psi \in \Diff_1^{(+)} (\A)$, then from \eqref{intermidiate computation 2} we get 
    \begin{align*}
        \delta_b \circ \delta_a [\varphi, \psi] = [ \delta_a \varphi, \delta_b \psi] + [ \delta_b \varphi, \delta_a \psi] \ ,
    \end{align*}
which does not have to vanish. Hence $[\varphi, \psi] \notin \Diff_1^{(+)} (\A)$, meaning that $\Diff_1^{(+)} (\A)$ is not a Lie algebra. Applying $\delta$ once again yields 
    \begin{align*}
        \delta_c \circ \delta_b \circ \delta_a [\varphi, \psi] = 0 \ . 
    \end{align*}
Thus $[\varphi, \psi] \in \Diff_2^{(+)} (\A)$. Proceeding in a similar fashion one can show that the composition of $\varphi \in \Diff_k^{(+)} (\A)$ and $\psi \in \Diff_l^{(+)} (\A)$ is of order $\leq k+l$, that is $\varphi \circ \psi \in \Diff_{k+l}^{(+)} (\A)$ and the filtered bimodule $\Diff_*^{(+)} (\A)$ is a Lie algebra. 
\end{example}

\begin{remark}
The $\A$-bimodule $\Diff_*^{(+)} (P,P)$ is also filtered, since 
    \[ 
        \delta (\varphi \circ \psi) = \delta (\varphi) \circ \psi + \varphi \circ \delta \psi
    \]
and so the composition of a differential operator of degree $\leq k$ with a differential operator of degree $\leq l$ results in a differential operator of degree $\leq k + l$
    \[ 
        \Diff_k^{(+)} (P,P) \otimes_{\Kbb} \Diff_l^{(+)} (P,P) \to \Diff_{k+l}^{(+)} (P,P)  \ .
    \]
\end{remark}

\begin{example}
Let $\varphi \in \End_{\Kbb} (\A)$, and recall that $\A$ is associative, commutative and unital algebra over $\Kbb$.
    \begin{itemize}
        \item  $\Diff_0 (\A) = \underset{a \in \A}{\cap} \ker \delta_a$. Using the definition \eqref{the delta derivation}, we have 
            \begin{align*}
                \delta_a \varphi (u) = \varphi (au) - a \varphi (u)  \ . 
            \end{align*}
        So $\varphi \in \Diff_0 (\A)$ iff $ \varphi (au) = a \varphi (u)$ for all $a,u \in \A$. Choosing $u = 1$ and writing $a = a1$ gives $ \varphi (a1) = a \varphi (1)$, meaning that $\varphi$ is completely determined by its value on the unit element of $\A$, which gives 
                \[ 
                    \Diff_0 (\A) =  \End_{\A}(\A) = \A  \ .
                \]
        Note that the above case is rather special. If we consider the case of $\A$-modules $P,Q$, then we get 
                \[
                    \Diff_0 (P,Q) = \End_{\A}(P,Q) \ .
                \]   
        \item $\Diff_1 (\A) = \underset{a,b \in \A}{\cap} \ker \delta_b \circ \delta_a $, where 
            \begin{align}\label{intermediate computation 3}
                \delta_b \circ \delta_a \varphi (u) = \varphi (bau) - b \varphi (au) - a\varphi (bu) + ba\varphi (u) \  .
            \end{align}
        \item $\Diff_2 (\A) = \underset{a,b,c \in \A}{\cap} \ker \delta_c \circ \delta_b \circ \delta_a  $, where 
            \begin{align*}
                \delta_c \circ \delta_b \circ \delta_a \varphi (u) = \varphi (bau) - b \varphi (au) - a\varphi (bu) + ba\varphi (u)  \ .
            \end{align*}
    \end{itemize}    
\end{example}

The goal of the following example is to demonstrate that the above given algebraic definition of differential operators on a commutative algebra $\A$ fits with the standard picture of differential operators on functions. 
\begin{example}
Let $\A = C^{\infty} (\Rbb)$ be the algebra of smooth functions of one real variable, the algebra binary operation given by multiplication of functions. Take $\varphi = \del_x : = \frac{\del}{\del x}$. Then
    \begin{align}
        \delta_a \del_x (u) = 0 \iff \del_x (au) = a \del_x u  \ . 
    \end{align}
which is not the case for all $a,u \in C^{\infty}(\Rbb)$. As expected (since $\del_x \notin \A$), the above implies $\del_x \notin \Diff_0 (\A)$. On the other hand, using \eqref{intermediate computation 3}, 
    \begin{align*}
        \delta_b \circ \delta_a \varphi (u) = \del_x (bau) - b \del_x (au) - a \del_x (bu) + ba \del_x u = 0
    \end{align*}
is satisified for all $a,b,u \in C^{\infty}(\Rbb)$, thus $\del_x \in \Diff_1 (C^{\infty}(\Rbb))$. Similarly, if we take $f \del_x $, where $f \in C^{\infty}(\Rbb)$, then
    \begin{align*}
        \delta_b \circ \delta_a f \del_x (u) = f\del_x (bau) - b f\del_x (au) - a f\del_x (bu) + ba f\del_x u = 0 \ ,
    \end{align*}
and so $f\del_x \in \Diff_1 (C^{\infty}(\Rbb))$. Finally, consider $\del^2_x : = \del_x \circ \del_x$. Then
    \begin{align*}
        \delta_b \circ \delta_a \del^2_x = \delta_b ( \delta_a \del_x \circ \del_x + \del_x \circ  \delta_a \del_x) = \delta_a  \del_x  \circ \delta_b \del_x + \delta_b  \del_x  \circ \delta_a \del_x \ ,   
    \end{align*}
which is not vanishing for all $a,b \in C^{\infty}(\Rbb)$. Thus $\del^2_x \notin \Diff_1 (C^{\infty}(\Rbb))$. We can easily check that 
    \[ 
        \delta_c \delta_b \circ \delta_a f \del^2_x = 0 \ ,
    \]
so $f\del^2_x \in \Diff_2 (C^{\infty}(\Rbb))$. One can show that $f\del^i_x \in \Diff_i (C^{\infty}(\Rbb))$, for all $i \in \Nbb_0$ and $f \in C^{\infty}(\Rbb)$, where $\del^i_x : = \underbrace{\del_x \circ \ldots \circ \del_x}_{\text{$i$-times}}$. Since we have the sequence of inclusions 
    \[
        \Diff_0{\A} \xhookrightarrow{} \Diff_1 {\A} \xhookrightarrow{} \ldots \xhookrightarrow{} \Diff_k{\A} \xhookrightarrow{} \ldots \ . 
    \]
Altogether we get 
    \[ 
        D_k : = \sum_{i = 0}^k f_i \del^i_x \in \Diff_k (C^{\infty}(\Rbb)) \ . 
    \]
\end{example}
    
\begin{definition}\label{symbol algebra}
Consider the factor space
    \begin{align*}
        \Smbl_k (\A) : = \Diff_k^{(+)} (\A) / \Diff_{k-1}^{(+)} (\A) \ . 
    \end{align*}
The \emph{symbols algebra of $\A$} is 
    \begin{align*}
        \Smbl_* (\A): = \bigoplus_{k \in \Nbb_0} \Smbl_k (\A)  \ ,
    \end{align*}
with the graded algebra structure
    \begin{align*}
        \Smbl_k \cdot \Smbl_l \subset \Smbl_{k+l} & \{ \Smbl_k , \Smbl_l \} \subset \Smbl_{k+l-1} \ .
    \end{align*}
\end{definition}

\begin{remark}
Recall that to any filtered algebra, we can associate a graded commutative algebra. This is precisely the case of the symbol algebra $\Smbl_*(\A)$ with respect to the filtered algebra $\Diff_*^{(+)}(\A)$.  
\end{remark}

\begin{proposition}
The symbol algebra $\Smbl_*(\A)$ is a Poisson graded algebra of degree $1$.
\end{proposition}


\begin{exercise}
    \begin{enumerate}
        \item $\Diff_k^{(+)}(\A)$ is an $\A$-submodule in $\End(\A)$.
        \item $\Diff_*(\A)$ is a filtered subalgebra in $\End (\A)$.
        \item $\Diff_0(\A) = \A$ and $\Diff_1(\A) = \mathcal{D}(\A) \oplus \A$ .
        \item $\Smbl_0 (\A) = \A$ and $\Smbl_1 (\A) = \mathcal{D}(\A)$.
    \end{enumerate}
\end{exercise}


\section{Intermezzo - Tensor, symmetric and exterior algebras}

\subsection{Tensor algebra of a vector space}

Tensor algebra of a $\Kbb$ vector space, denoted by $T(V)$ is 
    \begin{align*}
        T(V) : = \bigoplus_{k = 0}^{\infty} T^k (V) \ , 
    \end{align*}
where 
    \begin{align*}
        T^k (V) : = \otimes^k V = V \otimes \ldots \otimes V \ .
    \end{align*}
The space $T^k (V) $ consists of $\Kbb$-multilinear mappings 
    \begin{align*}
        \tau \colon V^* \times \ldots \times V^*  \to \Kbb  \ , 
    \end{align*}
where $V^* $ is the vector space dual to $V$. by definition, $ T^0 (V) = \Kbb$. Also, $ T^1 (V) = V$. The algebra product in $T(V)$ is given by the canonical isomorphism defined by the tensor product
    \begin{align*}
        T^k (V) \otimes T^l (V) \to T^{k+l} (V) \ .
    \end{align*}
The tensor algebra satisfies the following universal property: for every $\Kbb$-algebra $\A$ and arbitrary linear map $\psi \colon V \to \A$, there is a uniquely given linear map $\tilde{\psi} \colon T(V) \to \A$ such that the following diagram commutes
    \begin{equation*}
        \begin{tikzcd}
        V \arrow{r}{\iota} \arrow[swap]{dr}{\psi} & T(V) \arrow[dashed]{d}{\exists! \tilde{\psi}}  \\
            & \A
        \end{tikzcd}
    \end{equation*}
That is $\tilde{\psi} \circ \iota = \psi$, where $\iota$ is the canonical embedding of $V$ into $T(V)$. 


\subsection{Symmetric algebra of a vector space}
Let $V$ be a $\Kbb$-vector space, $T(V)$ its tensor algebra. Consider 
    \begin{align}\label{ideal for symmetrization}
        \J : = <u \otimes v - v \otimes u \ |\ u,v \in V > \ ,
    \end{align}
which is a two-sided ideal in $T(V)$. Then the qoutient algebra $S(V) : = T(V) / \J$, called the \emph{symmetric algebra of $V$}, is an associative and commutative algebra, satisfying the following universal property. For any associative, commutative and unital algebra $\A$ and every linear mapping $\psi \colon V \to \A$, there is precisely one unital algebra homomorphism $\tilde{\psi}$ such that the following diagram commutes
    \begin{equation*}
        \begin{tikzcd}
        V \arrow{r}{\iota} \arrow[swap]{dr}{\psi} & S(V) \arrow[dashed]{d}{\exists! \tilde{\psi}}  \\
            & \End(V)
        \end{tikzcd}
    \end{equation*}
where $\epsilon$ is the canonical embedding of $V$ into $S(V)$, given by the composition of the canonical embedding $V \to T(V)$ and the canonical quotient projection $T(V) \to S(V)$. 

\begin{remark} Let us mention some useful properties about the above defined algebras. 
    \begin{itemize}
        \item $S(V)$ is a free, associative, commutative and unital algebra on $n = \dim V$ generators. The product is given as follows. To avoid confusion, we will denote the classes in $S(V)$ with bracket notation. Let $[s] \in S^k(V), [t] \in S^l(V)$. Then the product is 
            \[
                [s][t] : = [s \otimes t] \in S^{k+l}(V)  \ .
            \]
        This product is well defined. To check this, consider different representatives of the equivalence classes $[\tilde{s}] = [s], [\tilde{t}] = [t] $. This means there exist $j_s, j_t \in \J$ such that $s = \tilde{s} + j_s$ and $t = \tilde{t} + j_t$. Then 
            \begin{align*}
                [s \otimes t] = [(\tilde{s} + j_s ) \otimes  \tilde{t} + j_t)] =  [\tilde{s} \otimes  \tilde{t} + \tilde{s} \otimes j_t + j_s \otimes  \tilde{t} + j_s \otimes j_t ]  = [\tilde{s} \otimes  \tilde{t} ] \ ,
            \end{align*}
        where the last equality follows from the fact that $\tilde{s} \otimes j_t + j_s \otimes  \tilde{t} + j_s \otimes j_t  \in \J$, since $\J$ is a two-sided ideal in $T(V)$. 
        \item The ideal $\J$ is homogeneous, meaning that the factor algebra $S(V)$ inherits the grading of $T(V)$. Thus we have $S(V) = \oplus_{k = 0}^\infty S^k(V)$, where $S^k (V) : = T^k(V) / \J^k$, where $\J^k = \J \cap T^k(V)$. 
        \item Consider the canonical projection $\pr_k \colon T^k (V) \to S^k (V)$. We can restrict $\pr_k$ to the linear subspace of symmetric $k$-tensors 
            \begin{align*}
                \pr_k|_{\Sym_k(V)} \colon \Sym_k(V) \to S^k (V)  \  .
            \end{align*}
        Because we always assume $\operatorname{char}\Kbb = O$, the above map can be inverted, yielding a graded vector space isomorphism
            \[
                \Sym (V) = \oplus_{k = 0}^\infty \Sym_k(V) \cong \oplus_{k = 0}^\infty S^k(V) = S(V)
            \]
        Although the space of symmetric tensors and the symmetric algebra are isomorphic as a graded $\Kbb$-vector spaces, it does not make sense to speak about isomorphism of algebras, since $\Sym (V)$ does not posses algebra structure in the sense that the tensor product of two symmetric tensors does not have to be a symmetric tensor. In $\operatorname{char}\Kbb > 0$, we even lose the graded vector spaces isomorphism.
        \item $\Sym(V)$ is a linear subspace of $T(V)$.
        \item $S(V)$ is a not a subalgebra of $T(V)$. 
    \end{itemize}
\end{remark}

\textbf{Symmetrization}. Let $\s_k \colon T^k(V) \to \Sym^k(V)$ be the symmetrization map, given for arbitrary $t \in T^k(V)$ by 
    \begin{align}\label{symmetrization map}
        \s_k (t) : = \frac{1}{k!} \sum_{\sigma \in \mathfrak{S}_k} \sigma \cdot t \ ,
    \end{align}
where $\mathfrak{S}_k$ is the permutation group on $k$-elements and $\sigma \cdot t$ is the action of the permutation group on $k$-tensors, given on $t = t_1 \otimes \ldots \otimes t_k$ by $\sigma \cdot t : = t_{\sigma(1)} \otimes \ldots \otimes t_{\sigma(k)}$ (and we extend $\cdot$ on general $t \in S_k$ by $\Kbb$-linearity).

\begin{proposition}
The ideal \eqref{ideal for symmetrization} is a graded ideal $\J  = \oplus_{k = 0}^\infty \J_k$, where $\J_k : = J \cap T^k(V)$. Moreover, $\ker \s_k = \J_k $ and $T^k(V) = \J_k \oplus \Sym^k(V) \cong \J_k \oplus S^k(V)$
\end{proposition}

Let $\{e_1, \ldots, e_n \} $ be a basis of $V$. The map \eqref{symmetrization map} gives an isomorphism 
    \[ 
        S(V) \cong \Kbb[x_1, \ldots, x_n]
    \]
such that 
    \[
        \s_k (e_{i_1} \otimes \ldots \otimes e_{i_k}) = x_{i_1} \dots x_{i_k}  \ .
    \] 
The special case of PBW theorem says that the monoms $\{e_1^{i_1}, \ldots, e_n^{i_1} \}$ form a~basis of $S(V)$ as a $\Kbb$-vector space and $S(V) \cong \Kbb[e_1, \ldots, e_n]$.


\begin{example}
Let $\g$ be a finite dimensional Lie algebra over $\Kbb = \Rbb$ or $\Kbb = \Cbb$. The symmetric algebra $S(\g^*)$ is isomorphic to $\Kbb[\g^*]$, the polynomial algebra with $\dim \g^*$ variables. The bracket on $\g^*$ makes $\Kbb[\g^*]$ a (commutive) Poisson algebra. 

A simpler version of this construciton is $\A = S(V^*) = \Kbb[V^*]$, where $V$ is a~finite dimensional vector space equipped with a skew-symmetric bilinear form $B \colon V\times V \to \Kbb$, which provides a Poisson brackets on $\Kbb[V^*]$. For instance, let $\dim V = 2$. Consider $X,Y \in V$ linearly independent, so that $V = <X,Y>$. Then $\Kbb[V^*] \cong \Kbb[X,Y]$ and $\{X,Y\} = B(X,Y) : = 1$. In this case, the pair $(V,B)$ is a symplectic plane. 
\end{example}


\subsection{Exterior algebra of a vector space}

Let $V$ be a $n$-dimensional $\Kbb$-vector space. The \emph{exterior algebra} of $V$, denoted $\Lambda (V)$, is defined as a graded subspace in the tensor algebra $T(V)$, formed by completely antisymmetric tensors. Recall that $t \in T^k (V)$ is \emph{completely antisymmetric} (or \emph{alternating}, or \emph{completely skew-symmetric}) if
    \begin{align*}
        \sigma \cdot t : = t_{\sigma(1)} \otimes \ldots \otimes t_{\sigma(k)} = \sgn (\sigma) t_1 \otimes \ldots \otimes t_k
    \end{align*}
for all $k$-permutations $\sigma$ (and we extend $\cdot$ on general $t \in S_k$ by $\Kbb$-linearity). Then $\Lambda^k (V) \subset T^k (V)$ is formed by all such $k$-tensors. Note that $\Lambda^1(V) = V$ and for $k = 0$ we define $\Lambda^0(V) = \Kbb$. The whole algebra is given by the $\Kbb$-vector space
    \[
        \Lambda (V) : = \oplus_{k \geq 0} \Lambda^k (V)  \ ,
    \]
with the product $\wedge$, called \emph{wedge product}, 
    \[
        \Lambda^k (V) \otimes \Lambda^l (V) \to \Lambda^{k+l} (V)
    \]
given by 
   \begin{align}\label{exterior product}
        \omega_1 \wedge \omega_2 : = \frac{(k + l)!}{k! l!} \Alt (\omega_1 \otimes \omega_2) \ ,
   \end{align}
where 
    \begin{align*}
        \Alt \colon T^k(V) \to \Lambda^k(V) \ ,
    \end{align*}
is a projection on the subspace of alternating tensors, called \emph{alternating map}, defined for arbitrary $t \in T^k(V)$ as 
    \begin{align*}
        \Alt (t) : = \frac{1}{k!} \sum_{\sigma \in \mathfrak{S}_k} \sgn (\sigma) \sigma \cdot t \ . 
    \end{align*}
\begin{example}
Consider $t \in T^2(V)$, given by $t = v \otimes w - w \otimes v$. Then $\Alt (t) = \frac{1}{2} (v \otimes w - w \otimes v - w \otimes v + v \otimes w = t$, i.e. $t$ is already an element of $\Lambda^2 (V)$. 
\end{example}
The wedge product satisfies
    \begin{align*}
         \omega_1 \wedge \omega_2 & =  (-1)^{pq} \omega_2 \wedge \omega_1, \\
          \omega_1 \wedge (\omega_2 \wedge \omega_3)  & = (\omega_1 \wedge \omega_2) \wedge \omega_3
    \end{align*}
meaning that $\Lambda (V)$ is a graded, associative $\Kbb$-algebra with the unit $1$. 

Similarly as in the case of symmetric algebra, we can define $\Lambda(V)$ as a~quotient of the tensor algebra. Consider a two-sided ideal $\J \subset T(V)$, given by 
    \[ 
        \J : = < t_1 \otimes t_2 + t_2 \otimes t_1 \ | \ t_1 , t_2 \in T(V) >
    \]
Then we define $ \Lambda(V) : = T(V) / \J $. Moreover, $\Lambda(V)$ satisfies the following universal property. For every associative, unital algebra $\A$ and any $\Kbb$-linear map $\varphi \colon V \to \A$, such that $j(v)^2 = 0$ for all $v \in V$, there is a unique algebra homorphism $\tilde{\varphi} \colon \Lambda(V) \to \A$, such that the following diagram commutes
    \begin{equation*}
        \begin{tikzcd}
        V \arrow{r}{\iota} \arrow[swap]{dr}{\varphi} & \Lambda(V) \arrow[dashed]{d}{\exists! \tilde{\varphi}}  \\
        & \A
        \end{tikzcd}
    \end{equation*}  
Using the universal property, one can show that the two above construction of $\Lambda(V)$, either as a subspace of alternating tensors or the quotient algebra, are isomorphic (in a unique way).


\subsection{Poisson structure on a symmetric algebra $S(\g)$}

\begin{theorem}
$S(\g)$ is a graded Poisson algebra of degree $1$ or of degree $2$. 
\end{theorem}

\begin{proof} (Degree 1.) Let $\g $ be a $\Kbb$-Lie algebra. We will show that the Lie bracket extends in a unique way to a Poisson bracket of degree $1$ on $S(\g)$. Let $g \in \g$ be arbitrary. Consider the endomorphism\footnote{Notice that using the action of $\g$ on itself given by multiplication from the left, we can identify elements in $\g$ as endomorphisms of $\g$, $h \mapsto l_h$. Then we have $\delta_g (l_h) (u) = h(gu) - gh(u) = (ad_g h) u$, thus $\delta_g|_{\g} \equiv \ad_g$.} $\ad_g \equiv \delta_g \in \End_{\Kbb} (\g)$, given by
    \begin{align*}
        \ad_g (h) = [g,h] \ , h \in \g \ .
    \end{align*}
Then $\ad_g$ can be extended in a unique way to the whole $S(\g)$ to a derivation on $S(\g)$ as follows. Consider a \emph{decomposable} $t \in S(\g)$, i.e. $t$ can be written as $t = t_1 \cdot \ldots \cdot t_k$, where $\cdot$ denotes the (symmetric) product in $S(\g)$ and $t_1, \ldots, t_k \in \g$. Then 
    \begin{align*}
        \ad_g(t) : = \sum_{i = 1}^k t_1 \cdot \ldots \cdot t_{i-1}  \cdot  ad_x(t_i) \cdot t_{i+1} \cdot \ldots \cdot t_k \ .
    \end{align*} 
A general element of $S(\g)$ is a $\Kbb$-linear combination of decomposables, so the above definition of the bracket can be extended linearly to the whole $S(\g)$. This extension is also denoted $\ad_g$. For arbitrary $t,u \in  S(\g)$ we have
    \begin{align*}
        \ad_g (t\cdot u) = \ad_g (t) \cdot u + t \cdot \ad_g (u) \ . 
    \end{align*}

(Degree 2.) It is sufficient to define a skew-symmetric bilinear form 
    \begin{align*}
        \Omega \colon \g \times \g \to \Kbb \ .
    \end{align*}
Then there is a unique extension of $\Omega (g,-) \colon \g \to \Kbb$ to a derivation 
    \[ 
        \del \colon S(\g) \to \mathcal{D}(S(\g)) \ ,
    \]
which can be extended to a degree $2$ derivation on $S(\g)$.
\end{proof}

\begin{example}
Let $\Omega$ be non-degenerate, skew-symmetric $2$-form on $\g = \Rbb^{2m}$, that is, we consider $\g$ to be abelian (the Lie bracket is zero). Then there is a~basis, called \emph{symplectic basis},  $e_1, \ldots, e_n$ ($n= 2m$), such that the matrix of $\Omega$ is written in this basis as 
    \[
        \begin{pmatrix}
        0 & I_m  \\
        -I_m & 0 
        \end{pmatrix} \ ,
    \]
where $I_n$ is the identity $m \times m$ matrix. Then for each $i \in \{ 1, \ldots , m \}$, the brackets of basis elements satisfy $\{e_i, e_{i+m} \} = - \{e_{i+m}, e_i \} = 1$ and all other combinations of basis elements have zero brackets. Using this choice of basis, we can identify $S(\g) \cong \Kbb[x_1, \ldots, x_{2m}]$. Note that this is a degree 2 bracket on $\g$, and differs from the Lie bracket $[-,-]_{\g}$, which we assumed is trivial. Let $m = 1$, so $\g = \Rbb^2$ and we consider coordinates $x,y$. Consider $P,Q \in \Kbb[x,y]$ given by
    \begin{align*}
        P(x,y) = \sum_{i,j} a_{ij}x^i y^j \ ,  &&     Q(x,y) = \sum_{p,q} b_{pq}x^p y^q \ . 
    \end{align*}
The bracket is 
    \[ 
        \{ P,Q\} = \sum_{i,j,p,q} a_{ij} b_{pq}\{ x^i y^j, x^p y^q\} = \sum_{i,j,p,q} a_{ij} b_{pq} (jp - iq) x^{i+p-1} y^{q+j-1}. 
    \]
If we pick $P = xy, Q = 4x^2$ and choose the bracket as 
    \[ 
        \{ P,Q\} = \frac{\del P}{\del x} \frac{\del Q}{\del y} - \frac{\del Q}{\del x} \frac{\del P}{\del y} = - 8x^2 \ , 
    \]
with the general formula being
    \[
        \{ P,Q\} = \sum_{i = 1}^m \frac{\del P}{\del x_i} \frac{\del Q}{\del x_{i+m}} - \frac{\del Q}{\del x_{i+m}} \frac{\del P}{\del x_i}
    \]
then this gives a bracket of degree $2$ on $\Kbb[x_1, \ldots, x_{2m}]$. The corresponding degree $2$ filtered Poisson algebra is $\A_{2m}(\g)$. The algebra $\A_{2m}(\g)$ has $2m$ generators $(x_i ,y_i) $, $1 \leq i \leq m$ and relations
    \begin{align*}
        x_i y_i - y_i x_i = 1 \ , 1 \leq i \leq m && x_i y_j - y_j x_i = y_i y_j - y_j y_i = x_i x_j - x_j x_i = 0 \ , i \neq j \ . 
    \end{align*}
$\A_{2m}(\g)$ is called the $2m$-th Weyl algebra.
\end{example}


Let us conclude this section with the following remark, which puts the above algebraic constructions in the context of smooth manifolds.

\begin{remark}
The constructions of tensor algebra, symmetric algebra, and exterior algebra can be extended to the case of smooth manifolds by considering the tangent (or cotangent) space at a given point. This leads to the notion of \emph{tensor fields} as sections of the bundle $\otimes TM \to M$, \emph{symmetric tensor fields} as sections of the bundle $S(TM) \to M$, and \emph{differential forms} as sections of the bundle $\Lambda(TM) \to M$. We will speak more about these objects in chapter concerning the differential calculus on Poisson manifolds.
\end{remark}

\section{Universal enveloping and PBW theorem}

Let $\g$ be a finite dimensional Lie algebra over $\Kbb$. We assume that there is an associative algebra $\A$ such that $\g$ can be embedded in $\A$ (so that one can multiply elements of $\g$), and the Lie bracket of $\g$ is given by the commutator in $\A$, i.e. for all $x,y \in \g$
    \begin{align*}
        [x, y] = xy - yx \ , 
    \end{align*}
where the product on the right-hans side is the product in $\A$. Let us denote by $\{e_i\}_{1 \leq i \leq n}$ a~basis of $\g$ as a vector space. The Lie structure is defined by the  \emph{structure constants} $c_{ij}^k$ 
    \begin{align*}
        [e_i,e_j] = c_{ij}^k e_k .
    \end{align*}
We consider $S \subset \g \times \g$, a set of pairs, defined as follows
    \[
        S = \{ (e_ie_j, e_je_i + \sum_{k = 1}^n c_{ij}^k e_k) \in \g \times \g  \ |\ i > j\} \ .
    \]
It is easy to verify that there is no ambiguity in the presentation of the triple product $e_i e_j e_k$ with $i > j > k$, since
    \begin{align*}
        (e_i e_j)e_k & = (\sum_{l=1}^n c_{ij}^l e_l + e_j e_i)e_k = \sum_{l=1}^n c_{ij}^l e_l e_k + e_j e_i e_k  \ , \\
        e_i (e_j e_k) & = e_i (\sum_{m=1}^n c_{jk}^m e_m + e_k e_j) = \sum_{m=1}^n c_{jk}^m e_i e_m + e_i e_k e_j  \ .
    \end{align*}
Thus the difference vanishes due to the Jacobi identity \eqref{Jacobi identity}
    \[ 
        (e_i e_j)e_k -  e_i (e_j e_k) = \sum_{l=1}^n \sum_{m=1}^n (c_{ij}^m c_{km}^l + c_{jk}^m c_{im}^l + c_{ki}^m c_{jm}^l) e_l = 0  \ .
    \]


\subsection{Universal enveloping algebra}
Let $\g$ be a $n$-dimensional ($n < \infty$) Lie algebra, $\{ e_i\}_{1 \leq i \leq n}$ a basis of $\g$. Let $T (\g) = \oplus_i T^i (\g)$ be the tensor algebra of the underlying vector space of $\g$. 

\begin{proposition}
$T(\g)$ is an associative algebra with respect to the tensor product $\otimes \colon T^p(\g) \times T^q(\g) \to T^{p+q}(\g)$ 
    \begin{align*}
         (x_1 \otimes \dots x_p , y_1 \otimes \dots y_q) \mapsto x_1 \otimes \dots x_p \otimes y_1 \otimes \dots y_q  \ . 
    \end{align*}
\end{proposition}

\begin{remark}
$T(\g) $ is generated as $T(\g) = \Kbb <e_1, \dots, e_n> $ (in the algebra sense). Put differently, the tensor algebra is isomorphic (as a $\Kbb$-algebra) to a~free, associative, non-commutative $\Kbb$-algebra on $n$-generators. 
\end{remark}

Consider a subspace $\I \subset T(\g)$, generated as 
    \begin{align}\label{ideal for the envelping algebra}
        \I : = <t_1 \otimes ([x,y] - x \otimes y + y \otimes x ) \otimes t_2 \ |\  t_1, t_2 \in T(\g), x,y \in \g >  \ .
    \end{align}
Note that $[x,y] - x \otimes y + y \otimes x \in \g \oplus T^2(\g)$.

\begin{proposition}
$\I$ is a (two-sided) ideal in $T(\g)$, i.e. $\forall j \in \I, t \in T(\g): j\otimes t, t \otimes j \in \I$.
\end{proposition}

\begin{proof}
The proof is obvious from the form of generators of $\I$. 
\end{proof}

\begin{remark}
Let $\A$ be an associative algebra. We will denote by $\A_{\Lie}$ the corresponding Lie algebra $\A_{\Lie}$, with the Lie bracket given by the commutator $[x,y]:= xy - yx$, for all $x,y \in \A$. 
\end{remark}

\begin{definition}[universal enveloping algebra]
Let $\g$ be a Lie algebra. The factor space
    \begin{align*}
        \U (\g) : = T(\g)/\I \ ,
    \end{align*}
where $\I$ is given by \eqref{ideal for the envelping algebra}, is called the universal enveloping algebra of $\g$. 
\end{definition}

\begin{proposition}
$\U (\g)$ is an associative algebra. Moreover, $\U (\g)$ can be endowed with a Lie algebra structure, which is compatible with the Lie algebra structure on $\g$.
\end{proposition}

\begin{proof}
Let $u_i = [t_i] \in \U (\g) , i = 1,2 $ and define $u_1 \cdot u_2 .: = [t_1 \otimes t_2]$. There is a~canonical embedding of the field $\Kbb$ and the Lie algebra $\g$ in $\U (\g)$. Consider the canonical embedding $\iota$ of $\g$ in $T (\g)$, and the canonical projection $\pr \colon T(\g) \to \U(\g) $. Then the embedding of $\g$ is given by the composition $\epsilon : = \pr \circ \iota $, 
    \begin{align}\label{canon. embed to univ. envel. alg.}
        \epsilon : = \pr \circ \iota \colon \g \to \U(\g) \ ,
    \end{align}
$x \mapsto u_x : = x + J, J \in \I $ (and in the same way we can embed $\Kbb \to \U(\g)$). We shall identify $x$ and $u_x$. Let $x,y \in \g$. Then $[x,y] - x\otimes y + y \otimes x \in \I$ and we have
    \begin{align*}
        \epsilon ([x,y] - x\otimes y + y \otimes x) & = u_{[x,y]} - u_x u_y + u_y u_x \overset{\text{\tiny{notation}}}{=} [x,y] - xy + yx = 0  \ .
    \end{align*}
Hence $[x,y] = xy - yx$ in $\U(\g)$. This means that the Lie algebra structure on $\U(\g)$, which is given by the commutator, can be restricted to $\g$ and the canonical embedding $\epsilon \colon \g \to \U (\g)_{\Lie}$ is a homomorphism of Lie algebras, i.e. 
    \[ 
        \epsilon ([x,y]) = [ \epsilon (x),  \epsilon (y) ] = \epsilon (x)  \epsilon (y) -  \epsilon (y) \epsilon (x) 
    \]
for all $x,y \in \g$.
\end{proof}

\begin{exercise}
Show that the product in $\U (\g) $ is well-defined.
\end{exercise}

\textbf{Universal property of $\U(\g)$}. Let $\epsilon \colon \g \to \U(\g)$ be the canonical Lie algebra homomorphism \eqref{canon. embed to univ. envel. alg.}. The universal enveloping algebra satisfies the following universal property. For any associative, unital algebra $\A$ (over the same field as $\g$) and any Lie algebra homomorphism $\psi \colon \g \to \A_{\Lie}$, there is a unique associative, unital $\Kbb$-algebra homomorphism $\tilde{\psi} \colon \U(\g)_{\Lie} \to \A_{\Lie} $ s.t. $\psi = \tilde{\psi} \circ \epsilon $, i.e. the following diagram commutes
    \begin{equation}\label{UP of the universal enveloping algebra}
        \begin{tikzcd}
        \g \arrow{r}{\epsilon} \arrow[swap]{dr}{\psi} & \U(\g)_{\Lie} \arrow[dashed]{d}{\exists! \tilde{\psi}}  \\
        & \A_{\Lie}
        \end{tikzcd}
    \end{equation}    

\begin{remark}
For $\g = 0$ we have $\U(\g) = \Kbb$. For $\g$ abelian (i.e. the bracket of $\g$ is trivial), we have $\U(\g) = \Kbb [\g]$, where $\Kbb [\g]$ is a polynomial ring over $\g$.\footnote{Every element of $\g$ serves as a variable. One can think of $\Kbb [\g]$ as $\Kbb[g_1, g_2, \ldots]$ for all $g_i \in \g$.}
\end{remark}

\subsection{Poincaré-Birkhoff-Witt (PBW) theorem}

\begin{theorem}[PBW theorem]
 Let $\g$ be a $\Kbb$-Lie algebra and $\{e_i\}_{1 \leq i \leq n}$ be a basis of $\g$ as a $\Kbb$-vector space. Let $(\U(\g), \epsilon)$ be the universal enveloping algebra. Then $\{1, \epsilon (e_i) \}_{1 \leq i \leq n}$ is a basis of $\U(\g)$ as a $\Kbb$-algebra and the canonical embedding $\epsilon \colon \g \to \U(\g)$ is injective. 
\end{theorem}


\begin{remark}
We can rephrase the above theorem as follows. Let $\{e_i\}_{1 \leq i \leq n}$ be a basis of $\g$. For $k \in \Nbb_0$ consider an index set $I_k = \{ i_1, \dots i_k\}$, $I_0 = \emptyset$. We say that $I_k$ is \emph{increased} if $i_1 \leq \dots \leq i_k$. Denote by $E_{I_k} \in \U(\g) $ the element $E_{I_k} = e_{i_1} \ldots e_{i_k}$ and $E_{I_0} = 1$. Then the following set is a basis for $\U(\g)$ as a~$\Kbb$-vector space
    \begin{align*}
        \{ E_{I_k} \ | \ \text{$I_k$ increasing, } k \in \Nbb_0 \}  \ . 
    \end{align*}
In other words, elements of the form $e_1^{\alpha_1} \dots e_k^{\alpha_k}$, where $ \alpha_1 \leq \ldots \leq \alpha_n$, are linearly independent and generate $\U(\g)$. 

\end{remark}

\begin{proposition}
 $\U(\g)$ is a filtered algebra. The filtration is given by a degree of elements $x \in \U(\g)$, which is defined 
    \begin{itemize}
        \item either as the degree of the polynomial, which describes $x$ (after a choice of a base by PBW),
        \item or by the set $\U_p(\g)$, which is the image of $\oplus_{i = 1}^p T^i(\g)$ under the canonical projection $\pr \colon T(\g) \to \U(\g)$, i.e. $\U_p(\g)$ is the set of elements of degree $\leq p$.
    \end{itemize}
\end{proposition}

\begin{proof}[Sketch of the proof of the PBW theorem]
We will consider the case when $\g$ comes equipped with a faithful representation.\footnote{The kernel of the representation, as a homomorphism from $\g$, is trivial. Example: $\g$ acts faithfully on $C^{\infty}(\g)$. Another example: $\g$ is a matrix algebra ($n \times n$ matrices), then $\g$ acts faithfully on $\Rbb^n$.} 

Consider the following algebra homomorphism 
    \begin{align*}
        \tilde{\Delta} &\colon T(\g) \to T(\g) \otimes T(\g) \ , 
    \end{align*}
such that for all $t \in \g$, $\tilde{\Delta}$ satisfies
    \begin{align*}
        \tilde{\Delta} & (t) = t \otimes 1 + 1 \otimes t  \ , t \in \g \ . 
    \end{align*}
Notice that $1 \otimes T(\g) $ and $T(\g) \otimes 1$, commute with respect to the algebraic structure of $T(\g)$. 

\begin{lemma}
The algebra homomoprhism $\tilde{\Delta}$ descends to a map $\Delta \colon \U(\g) \to \U(\g) \otimes \U(\g) $ ($\Delta$ is a coproduct with respect to the Hopf algebra structure on $\U(\g)$).
\end{lemma}

\emph{Proof of the lemma.}
The map $\Delta$ is given by the following commutative diagram 
 \begin{equation*}
        \begin{tikzcd}
        \g \arrow{r}{\epsilon} \arrow[swap]{dr}{\iota} & \U(\g) \arrow{d}{\tilde{\iota}} \arrow[dashed]{r}{\Delta} & \U(\g) \otimes \U(\g)  \\
        & T(\g) \arrow{r}{\tilde{\Delta}} & T(\g) \otimes T(\g) \arrow[swap]{u}{\pr \otimes \pr}
        \end{tikzcd}
    \end{equation*} 
where $\iota$ and $\epsilon$ are the canonical embeddings of $\g$ into the corresponding tensor algebra $T(\g)$ and the universal enveloping algebra $\U(\g)$, respectively. The map $\tilde{\iota}$ is given by the universal property of $\U(\g)$ (consider $T(\g)$ with trivial bracket), and $\pr$ is the canonical quotient projection. All tensor products are considered over $\Kbb$ (the field over which we consider $\g$). 
\hfill \qedhere

\hfill \newline

Since the basis of $\g$ gives rise to the basis of $S(\g)$, another formulation of the PBW theorem is that $S(\g)$ gives rise to basis of $\U (\g)$. Let $\Sym(\g) : = \oplus_{k = 0}^{\infty} \Sym^k (\g)$ be the space of symmetric tensors, where $t \in \Sym^k (\g)$ if
    \begin{align*}
        t(u_{\sigma(1)}, \ldots, u_{\sigma(k)} = t(u_1, \ldots, u_k)
    \end{align*}
for all $k$-permutations $\sigma$. Clearly we have an injection $\Sym(\g) \xhookrightarrow{j} T(\g)$. Too prove the PBW theorem, we prove that the composition 
    \[
        \begin{tikzcd}
        \Sym(\g) \arrow{r}{j} & T(\g) \arrow{r}{\pr} & \U(\g) 
        \end{tikzcd}
    \]
is injective.

\noindent \textbf{Induction on $k$.} 
    \begin{enumerate}
        \item Suppose $k = 1$. Then we can use the existence of faithful representation of $\g$ on some vector space $V$. This amounts to the existence of an injective $\Kbb$-algebra homomorphism $\rho \colon \g \to \End(V)$. Note that $\Sym_{\leq 1} (\g) = \g$. Using the universal property of $\U(\g)$, we have the following commutative diagram 
            \begin{equation*}
                \begin{tikzcd}
                \g \arrow{r}{\epsilon} \arrow[swap]{dr}{\rho} & \U(\g) \arrow[dashed]{d}{\tilde{\rho}}  \\
                    & \End(V)
                \end{tikzcd}
            \end{equation*}
        Since $\rho = \tilde{\rho} \circ \epsilon$ and $\rho$ is injective, $\epsilon$ is injective. Thus $\pr \circ \epsilon = \pr \circ j|_{\g}$ is injective.   
        \item Take $\Sym_{\leq k} (\g): = \oplus_{i = 1}^k \Sym^i (\g)$ and suppose that the map $\pr \circ j$ is injective on $\Sym_{\leq k}(\g)$. Consider $f \in \Sym_{\leq k + 1}(\g)$.  We want to prove that if
            \begin{align*}
                (\pr \circ j )(f) = 0 \in \U(\g)
            \end{align*}
        then $f = 0$. Define  
            \[
                g : = \tilde{\Delta} (f) - f \otimes 1 - 1 \otimes f \  .
            \]
        This implies that $g \in \Sym_{\leq k}(\g) \otimes \Sym_{\leq k}(\g)$. Thus
            \begin{align*}
                (\pr \otimes \pr ) \circ (j \otimes j )(g) = 0 \ , 
            \end{align*}
        implies $g = 0$ by the induction hypothesis. But $g = 0$ is equivalent to 
            \begin{align*}
                \tilde{\Delta} (f) =  f \otimes 1 + 1 \otimes f \ .
            \end{align*}
        Now suppose $(\pr \otimes \pr ) \tilde{\Delta} (f) = 0$, which is equivalent to $(\pr \otimes \pr )(f \otimes 1 + 1 \otimes f ) = 0 \in \U(\g)$. This can happen if and only if $f = 0$. Hence $\Sym (\g)$ can be injectively embedded into $\U (\g)$.
    \end{enumerate}
\end{proof}

\begin{observation}
Note that the assumption $\operatorname{char} \Kbb = 0$ is necessary in the above proof. To see that, suppose $\operatorname{char} \Kbb = p$ and $\dim \g > 1$. Then, using the binomial theorem, $\tilde{\Delta} (v^p) - v^p  \otimes 1 - 1 \otimes v^p = (v  \otimes 1 - 1 \otimes v)^p - v^p  \otimes 1 - 1 \otimes v^p = 0$ for all $v \in \g$.
\end{observation}

\begin{example}
Let $\dim \g = 2$ and $\{e_1, e_2\}$ be a basis of $\g$ such that $[e_1, e_2] = e_2$. Then we have the relation $e_1e_2 - e_2 e_1 = e_2$ in $\U (\g)$, which means that $E = e_1^{i_1}e_2^{j_1}  e_1^{i_2}e_2^{j_2}  \cdots e_1^{i_k}e_2^{j_k}$ is a linear combination of $e_1^{\alpha_i} e_2^{\alpha_j}$. For example
    \begin{align*}
        E & = e_1 e_2^2 e_1 = e_1 e_2 e_2 e_1 = e_1e_2 (e_1 e_2 - e_2) = e_1 e_2 e_1 e_2 - e_1e_2^2 \\
            & = e_1 (e_1e_2 - e_2) e_2 - e_1e_2^2 = e_1^2e_2^2 - 2 e_1e_2^2
    \end{align*}
\end{example}

\begin{example}
Consider the algebra $\g = \sl_2(\Cbb) = \{ A \in \Mat_2 (\Cbb) \ | \ \tr A = 0 \}$, where $\tr$ is the trace of a matrix and the Lie bracket is given by the commutator. The Chevalley-Eilenberg basis of $\g$ is given by the following matrices 
    \begin{align*}
        e = \begin{pmatrix}
            0 & 1 \\
            0 & 0 
            \end{pmatrix} , & &
        f = \begin{pmatrix}
            0 & 0 \\
            1 & 0
            \end{pmatrix}, & &
        h = \begin{pmatrix}
            1 &  0 \\
            0 & -1 
            \end{pmatrix} 
    \end{align*}
and the brackets are 
    \begin{align*}
        [e,f] = h, && [h,e] = 2e, && [h,f] = -2f \  . 
    \end{align*}
Then $\U(\sl_2(\Cbb))$ is the associative $\Cbb$-algebra with generators $e,f,h$ and relations 
    \begin{align*}
        ef - fe = h, && he - eh = 2e, && hf - fh = -2f \  . 
    \end{align*}
The basis of $\U(\sl_2(\Cbb))$ is given by monoms $e^{\alpha} h^{\beta} f^{\gamma}$, where $\alpha ,\beta, \gamma \in \Nbb_0$. 
\end{example}


\subsection{Universal enveloping and differential operators}

Consider a real, finite dimensional Lie algebra $\g$ and the corresponding universal enveloping algebra $\U(\g)$. There is a natural filtration in $\U (\g)$ given by a sequence of subspaces $\{\U_k \}_{k \geq 0}$ s.t. $\U_k (\g)\subset \U_{k+1}(\g)$ and $\U_k(\g) \cdot \U_l  (\g)\subset \U_{k+l}(\g)$. Moreover, if $\alpha \in \U_k(\g), \beta \in \U_l(\g)$ then the commutator yields
    \begin{align}\label{bracket in the universal enveloping algebra}
        [\alpha, \beta] : = \alpha \beta - \beta \alpha \in \U_{k+l-1}(\g) \ .
    \end{align}
We define 
    \begin{align*}
        \A_k : = \U_k (\g)/ \U_{k-1}(\g) \ , 
    \end{align*}
and 
    \begin{align*}
        \A : = \bigoplus_{k \geq 0} \A_k \ .
    \end{align*}
The associative algebra structure of $\U(\g)$ induces a multiplication in $\A$, which is compatible with the order $\A_k \cdot \A_l \subset \A_{k+l}$. The bracket \eqref{bracket in the universal enveloping algebra} implies commutativity of the multiplication
    \begin{align*}
        \A_k \cdot \A_l = \A_l \cdot \A_k \ , 
    \end{align*}
and defines a bracket operation 
    \begin{align*}
        \{ -,- \} \colon \A_k \times \A_l \to \A_{k+l-1}    
    \end{align*}
on $\A$, making it an associative graded algebra.

\begin{exercise}
Verify that the above operations gives a Poisson algebra structure on $\A$. 
\end{exercise}

Each component $\A_k$ of $\A$ is isomorphic to the space $\Pol_k (\g^*)$ of degree $k$ homogeneous polynomials on $\g^*$ (the dual v. space of $\g$). The algebra $\A$ is isomorphic to 
    \begin{align*}
        \A \cong \Pol (\g^*) : = \bigoplus_{k \geq 0} \Pol_k (\g^*) \ .
    \end{align*}
and the brackets are compatible with the brackets from $\g^*$.

Now consider a smooth manifold $M$. Then the space $C^\infty(M)$ of smooth functions on $M$ is an associative, unital, commutative algebra with respect to addition and multiplication of functions. Hence we can use the definition \ref{algebra of differential operators} to define the filtered, associative algebra of differential operators
    \begin{align*}
        \Diff_* (M) : = \bigcup_{k\geq 0}  \Diff_k (M)  \ ,
    \end{align*}
where $\Diff_k (M)$ is the set of differential operators of order $\leq k$ on $M$
    \begin{align*}
         \Diff_k (M) : = \{ \Delta \colon C^\infty(M) \to C^\infty(M)| \  \delta_{f_0} \circ \ldots \circ \delta_{f_k} (\Delta) = 0  \ , \forall f_i \in C^\infty(M) \} \ , 
    \end{align*}
and for $k = 0$ we define $\Diff_0 (M) = C^\infty(M)$. Notice that $\delta_f (\Delta) = [\Delta, f]$ (see lemma \ref{derivation property of delta} for the definition of $\delta_f$). Recall that the algebra structure of $\Diff_* (M)$ is given by the composition of operators. In example \ref{Lie algebra structure on Diff}, we have seen that the commutator of operators defines a Lie algebra structure on $\Diff_*(M)$. The commutator acts with respect to filtered structure as follows. For $\delta \in \Diff_k, \nabla \in \Diff_l$ we have
    \begin{align*}
        [\delta, \nabla] = \delta \circ \nabla - \nabla \circ \delta \in \Diff_{m+k-1} (M) \ .
    \end{align*}

\begin{remark}\label{diff1 splits into functions and vector fields}
Since $\Diff_0 (M) = C^\infty(M)$, we have $\Diff_1 (M) \cong \Diff_0 (M) \oplus \Gamma (M,TM)$, i.e. $\Diff_1 (M)$ splits into functions and vector fields (seen as differential operators on $C^\infty (M)$. 
\end{remark}

Now we can define the Poisson algebra of degree $1$ as in \ref{symbol algebra}
    \begin{align*}
        \Smbl (M) = \bigoplus_{k\geq 0} \Smbl_k \ ,
    \end{align*}
called the symbol algebra of $C^\infty(M)$, where $\Smbl_k = \Diff_k (M) / \Diff_{k-1} (M)$. 

We shall now give another description of the Poisson algebra $\Smbl (M)$. Consider the space $\mathcal{P}_k (T^*M)$ of functions $p \colon T^*M \to \Rbb$ such that $p \mapsto f(q,p)$ is, for every fixed $q \in M$, a degree $k$ homogeneous polynomial on $T^*_q M$. Take
    \begin{align*}
        \mathcal{P}(T^*M) : = \bigoplus_{k\geq 0} \mathcal{P}_k (T^*M) \subset C^\infty (T^*M) \ , 
    \end{align*}
which is the space of all fibre-wise polynoial functions on $T^*M$. Moreover, $ \mathcal{P}(T^*M) $ is a Poisson subalgebra in $C^\infty (T^*M)$ (seen as a Poisson algebra with the standard Poisson brackets). 

\begin{theorem}
Poisson algebras $\Smbl (M)$ and $ \mathcal{P}(T^*M) $ are isomorphic. 
\end{theorem}

\begin{proof}
It is enough to prove this isomorphism on generators of $\Smbl_0 (M)$ and $\Smbl_1 (M)$ on one side, and on $ \mathcal{P}_0 (T^*M) $ and $ \mathcal{P}_1(T^*M) $ on the other side. It is clear that $ \Smbl_0 (M) \cong \mathcal{P}_0 (T^*M) $ since both of them are equal to $C^\infty (T^*M)$. Further, in remark \ref{diff1 splits into functions and vector fields} we have seen that 
    \begin{align*}
        \Diff_1 (M) \cong \Diff_0 (M) \oplus \Gamma (M,TM) \ , 
    \end{align*}
which gives
    \begin{align*}
        \Smbl_1 (M) = \Diff_1 (M) / \Diff_0 (M) \cong \Gamma (M,TM) \ , 
    \end{align*}
Now consider a vector field $X \in \Gamma (M,TM)$ and define $f_X \in C^\infty(T^*M)$ by 
    \begin{align*}
        f_X (q,p) : = <p,X_q> \in \Rbb \ .
    \end{align*}
This gives us a bijection between $\Gamma (M,TM)$ and $ \mathcal{P}_1 (T^*M) $, since
    \begin{align*}
        < - ,X_q> \colon T^*_q M \to \Rbb
    \end{align*}
is a (homogeneous) $1$-st order polynomial function (for every fixed $q \in M$). Moreover, for all vector fields $X,Y \in \Gamma (M,TM)$
    \begin{align*}
        f_{[X,Y]} : = \{ f_X, f_Y \}  \ , 
    \end{align*}
so the bijection is an algebra homomorphism. For $F \in C^\infty(M)$, the function $f_F \in  C^\infty(T^*M)$ is defined by 
    \begin{align*}
        f_F (q,p) : = F(q) \ . 
    \end{align*}
Then the following identities hold for all $F,G \in C^\infty(M)$ and $ X \in \Gamma(M,TM)$
    \begin{align*}
        \{f_X, f_F \} & = f_{X(F)} \ , \\
        \{f_F, f_G \} & = 0 \ .
    \end{align*}

This isomorphism can be extended to isomorphsisms $ \Smbl_k (M) \cong \mathcal{P}_k (T^*M) $ for all $k$ and hence to isomorphism $ \Smbl (M) \cong \mathcal{P} (T^*M) $. In other words, for all $k \geq 1$, the following sequence is exact
    \begin{equation*}
        \begin{tikzcd}
        0 \arrow{r}{} & \Diff_{k-1} (M) \arrow{r}{\iota} & \Diff_k (M) \arrow{r}{\sigma} & \mathcal{P}_k (T^*M) \arrow{r}{}  & 0 \ ,
        \end{tikzcd}
    \end{equation*}    
where $\iota$ is the inclusion (given by the filtered structure of $\Diff_* (M)$) and $\sigma$ is the \emph{symbol map}. To each $\Delta \in \Diff_k (M)$, the symbol assigns a polynomial function $\sigma(\Delta) \colon T^*M \to \Rbb$, which is fiberwise homogeneous of degree $m$. We will describe it in local coordinates $(\Bar{q}, \Bar{p})$ of $T^*M$, induced by the local coordinates $(\Bar{q})$ on $M$. A differential operator $\Delta \in \Diff_k (M)$ has the form 
    \begin{align*}
        \Delta = \sum_{|\alpha| \leq m} a_\alpha \del^{\alpha} \ , 
    \end{align*}
where $\del^{\alpha} : = \del_{q_1}^{\alpha_1} \circ \ldots \circ  \del_{q_k}^{\alpha_k}, \alpha_i \in \Nbb$ and $a_\alpha \in C^\infty(M)$. Then, the symbol is 
    \begin{align*}
       \sigma( \Delta (q,p)) = \sum_{|\alpha| = m} a_{\alpha} (q) p^{\alpha}  \ , 
    \end{align*}
where $p^{\alpha} : = p_1^{\alpha_1} \dots p_k^{\alpha_k}$. Hence $\sigma( \Delta (q,p))$ is (for every $q$) a homogeneous polynomial function of degree $m$ in the variable $p$. 
\end{proof}

\begin{remark}
Let us describe briefly an \emph{invariant} form of the symbol $\sigma(\Delta)$. Let $F \in C^\infty(M)$. Then $e^{tF} \Delta e^{-tF}$ is a differential operator of order $\leq k$ if $\Delta \in \Diff_{k} (M)$. Consider the following formal expression
    \begin{align*}
        e^{tF} \Delta e^{-tF} = \Delta + t[F, \Delta] + \frac{t^2}{2} [F,[F, \Delta] ] + \ldots + \frac{t^k}{k!}[F, [F , \ldots , [F, \Delta] ] \ldots ] \ .
    \end{align*}
The above expression is finite, since the $i$-th term  has degree $\leq i - k$. Then the symbol map gives the coefficient of the leading term of this expression, seen as a polynomial in $t$
    \begin{align*}
        \sigma (\Delta) (q, d_q \phi) = \frac{1}{k!}[F, [F , \ldots , [F, \Delta] ] \ldots ] (q)  \ .
    \end{align*}
\end{remark} 

\begin{example}
Let $M = G$ be a Lie group with the Lie algebra $\g = \Lie (G)$. Then $\U_k (\g) \subset \Diff_k(G)$ as the left invariant differential operators of order $\leq k$ on $G$. 
\end{example}

\section{Poisson manifolds}

\subsection{Poisson structure on the cotangent bundle}

Let $F,G \in C^\infty (T^*M)$ be smooth functions on the cotangent bundle, and let $(q^i, p_i)$ be local coordinates on $T^*M$. Define the \emph{Poisson brackets} by
    \begin{align}\label{canonical Poisson brackets on the cotangent bundle}
        \{ F,G \} = \sum_{i = 1}^n (\frac{\del F}{\del p_i} \frac{\del G}{\del q^i} - \frac{\del F}{\del q^i} \frac{\del G}{\del p_i} )\ .
    \end{align} 
This is a bilinear mapping (linear in both arguments $F$ and $G$) which is \emph{skew-symmetric}
    \begin{align}\label{skew-symmetry of the Poisson bracket}
        \{ F,G \}  = - \{ G, F \}  \, 
    \end{align}
satisfies the \emph{Leibniz rule}
    \begin{align}\label{Leibniz rule of the Poisson bracket}
        \{ F,GH \} = \{ F,G \} H + G\{ F,H \}  \ , 
    \end{align}
and the \emph{Jacobi identity}
    \begin{align}\label{Jacobi identity of the Poisson bracket}
        \{ F, \{ G,H \} \} + \{ G, \{ H,F \} \} + \{ H, \{ F,G \} \} = 0 \ .
    \end{align}

\subsection{Poisson manifolds}

Generalizing the above idea of Poisson brackets on the cotangent bundle, we can give the following definition.

\begin{definition}
A Poisson bracket on a manifold $M$ is a bilinear mapping
    \begin{align*}
        \{ -,- \} \colon C^\infty (M) \times C^\infty (M) \to C^\infty (M) \ , 
    \end{align*}
such that \eqref{skew-symmetry of the Poisson bracket}, \eqref{Leibniz rule of the Poisson bracket}, and \eqref{Jacobi identity of the Poisson bracket} are satisfied. A manifold $M$ equipped with a~Poisson bracket is called a Poisson manifold.
\end{definition}

\begin{remark}
Every $F \in C^\infty (M)$, the mapping $X_F \colon C^\infty (M) \to C^\infty (M)$, defined via the Poisson bracket as
    \begin{align*}
        X_F : = \{ F, -\} \ , 
    \end{align*}
is a \emph{derivation} on the algebra $ C^\infty (M)$. More precisely, $X_F$ is a $1$st order differential operator, and hence can be considered as a vector field on $M$. This field is called \emph{a Hamiltonian vector field} with \emph{Hamiltonian} $F$. 
\end{remark}

\subsection{Hamiltonian mapping}

For every $F,G \in C^\infty (M)$, The Leibniz rule \eqref{Leibniz rule of the Poisson bracket} gives 
    \begin{align*}
        X_{FG} = F X_G  + G X_F 
    \end{align*}
and one can define a mapping 
    \begin{align*}
        \mathcal{H} \colon T^* M \to TM 
    \end{align*}
by $\mathcal{H} (dF) : = X_F $, or 
    \begin{align*}
        <dG, \mathcal{H} (dF)> : = \{ F,G\} \ ,
    \end{align*}
where the bracket on the left-hand side denotes the evaluation of the $1$-form $dG$ on the vector field $\mathcal{H} (dF)$. The map $\mathcal{H}$ is given in local coordinates $(x^i)$ on $M$ by a matrix 
    \begin{align*}
        H^{kl} : = \{ x^k , x^l \} \ .
    \end{align*}
Denoting $\del_k : = \frac{\del}{\del x^k}$, the Poisson brackets can then be written as
    \begin{align*}
        \{ F, G \} = \sum_{k,l}  H^{kl} \del_k F \del_l G \ , 
    \end{align*}
and the components of the vector field $X_F$ are
    \begin{align*}
        X_F^k = \sum_l H^{lk} \del_k F \ . 
    \end{align*}
Suppose the manifold is the cotangent bundle $T^*M$ and the local coordinates $(x^i)$ are $(q^i , p_i)$, $i = 1, \ldots, n$. Then, the Hamilton map is given by the matrix
    \begin{align*}
        H = \begin{pmatrix}
                0 & -I_n \\
                I_n & 0
            \end{pmatrix} \ , 
    \end{align*}
where $I_n$ is the identity matrix, $2n = \dim T^*M$. In coordinates, the skew-symmetry of the Poisson bracket reads as
    \begin{align*}
        H^{kl} = - H^{lk} \ , 
    \end{align*}
and the Jacobi identity is 
    \begin{align*}
        \sum_i (H^{ki} \del_i H^{lm} + H^{li} \del_i H^{mk} + H^{mi} \del_i H^{kl} ) = 0 \ . 
    \end{align*}
If the matrix $H = (H^{kl})$ is non-degenerate, the mapping $\mathcal{H} \colon T^*M \to TM$ is invertible. If we denote by $(\Omega_{kl})$ the inverse of $H$, then one can define a $2$-form
    \begin{align*}
        \underline{\Omega} = \frac{1}{2} \sum \Omega_{kl} dx^k \wedge dx^l \ , 
    \end{align*}
or in more intrinsic way
    \begin{align*}
         \underline{\Omega}_x (X,Y) : = <\mathcal{H}^{-1}_x (X), Y> \ , 
    \end{align*}
where $x \in M$ and $X,Y \in T_xM$. The Jacobi identity implies that   
    \begin{align*}
        \d \underline{\Omega} = - \frac{1}{6} \sum_{k,l,m} \Omega_{klm} dx^k \wedge dx^l \wedge dx^m \ , 
    \end{align*}    
where 
    \begin{align*}
        \Omega_{klm} : = \del_k \Omega_{lm} + \del_l \Omega_{mk} +  \del_m \Omega_{kl} \ . 
    \end{align*}

The above formula for $d \underline{\Omega}$ can be proved using the following expression for $\Omega_{klm}$
    \begin{align*}
        \Omega_{klm} : = \sum_{p,q,r} \Omega_{kp} \Omega_{lq} \Omega_{mr} \left( \sum_i \mathcal{H}^{pi} \del_i \mathcal{H}^{qr} + \mathcal{H}^{qi} \del_i \mathcal{H}^{rp} + \mathcal{H}^{ri} \del_i \mathcal{H}^{pq} \right) \ . 
    \end{align*}


\subsection{Poisson bracket on a symplectic manifold}

\begin{definition}
A smooth manifold $M$, equipped with a closed $2$-form $\Omega$ ($d\Omega = 0$) which is non-degenerate, i.e. if $X \in T_x U, x \in M$ s.t. 
        \begin{align*}
            \Omega(X,Y) = 0 \ 
        \end{align*}
is called a symplectic manifold. 
\end{definition}
For every symplectic manifold, one can define the Poisson bracket on the algebra $C^\infty(M)$ by 
    \begin{align*}
        \{F,G \} = <dG, \mathcal{H}(dF)> = \Omega \left( \mathcal{H}(dF), \mathcal{H}(dG) \right) \ . 
    \end{align*}
That is, each symplectic manifold is a Poisson manifold. Let us emphasize that the inverse statement is not true, i.e. there are Poisson manifolds which are not symplectic.


\subsection{Examples of Poisson and Symplectic manifolds}

\begin{example}[cotangent bundles] Let $M$ be a smooth manifold, $\dim M = n$. Then the cotangent bundle $T^*M$ has dimension (as a smooth manifold) $\dim T^*M = 2n$. The cotangent bundle comes equipped with the \emph{canonical} (or \emph{tautological}) \emph{Liouville $1$-form}, which is defined as follows. Let $\pi \colon TM \to M$, and $\pi_* \colon T^*M \to M$ denote the projections from the tangent and cotangent bundle, respectively, and by $d\pi_*$ the tangent map to $\pi_*$. Consider the following commutative diagram
    \begin{equation*}
        \begin{tikzcd}
        T (T^*M) \arrow[dashed]{r}{d \pi_*} \arrow[swap]{d}{\Tilde{\pi}} & TM \arrow{d}{\pi}  \\
        T^*M  \arrow[swap]{r}{\pi_*} & M
        \end{tikzcd}
    \end{equation*}
where $\Tilde{\pi}$ is the tangent bundle projection down to $T^*M$ (seen as manifold). Then the Liouville $1$-form $\rho \in \Omega^1(T^*M)$ is defined for arbitrary $X \in T(T^*M)$ as 
    \begin{align*}
        \rho (X) : = <\pi(X), d \pi_* (X)> \ .  
    \end{align*}
Then $ \Omega : = d \rho$ is a symplectic form on $T^*M$. In canonical coordinates, 
    \begin{align*}
        \Omega = d \rho = d( \sum_{i=1}^n p_i dq^i) = \sum_i dp_i dq^i \ .
    \end{align*}
Poisson brackets given by this symplectic structure are given by \eqref{canonical Poisson brackets on the cotangent bundle}. 
\end{example}

\begin{example}[Sphere and projective space]
Let $S^2 \subset \Rbb^3$ be the unit $2$-sphere, i.e. solution of the equation $x^2 + y^2 + z^2 = 1$. Then the $2$-form 
    \begin{align*}
        \Omega = x dy \wedge dz + y dz \wedge dx + z dx \wedge dy \ , 
    \end{align*}
where the values of $x,y,z$ are assumed to satisfy the equation for $S^2$, defines a~symplectic form on $S^2$. Note that one can vies the $2$-sphere can be identified with the complex projective line $S^2 = \mathbb{P}^1 (\Cbb)$). 

Consider now $\Cbb^n$ with the Hermitian form 
    \begin{align*}
        <z,w> = \sum_{k= 1}^n \Bar{z}^k w^k  \  .
    \end{align*}
Define a $1$-form 
    \begin{align*}
        \rho : = i <\Bar{z}, dz> = i \sum_{k= 1}^n \Bar{z}^k dz^k \ , 
    \end{align*}
where $i = \sqrt{-1}$, and the $2$-form
    \begin{align*}
        \Omega = d\rho = i d <\Bar{z},dz> = i\sum_{k= 1}^n d\Bar{z}^k \wedge dz^k \ .
    \end{align*}
The $2$-form $\Omega$ is real, i.e. it can be written as
        \begin{align*}
        \Omega = -2 \sum_{k= 1}^n dx^k \wedge dy^k \ ,
    \end{align*}
for $x^k + i y^k = z^k$. Hence $(\Cbb^n, \Omega)$ is a \emph{real symplectic} manifold. Poisson brackets are 
    \begin{align*}
        \{ z^j, z^k \} =  \{ \bar{z}^j, \bar{z}^k \} = 0 & & \{ z^j, \bar{z}^k \} = i\delta^{jk} \ . 
    \end{align*}

If we fix a quadratic Hamiltonian
    \begin{align*}
        H(z) : = <z,z> = \sum_{k = 1}^n |z^k|^2
    \end{align*}
on $\Cbb^n$, then the level set $H(z) = 1$ is a $(2n-1)$-dimensional sphere $S^{2n-1}$. The Hamiltonian vector field $X_H$ on $\Cbb^n$ restricted to $S^{2n-1}$ defines a vector field $pX_H$. Recall that $\operatorname{U}(1) \cong \Rbb^*$ acts on $S^{2n-1}$ by $(\varphi, z) \mapsto e^{-i\varphi}$ and the quotient by this action is the complex projective space, $S^{2n-1} / \Rbb^* \cong \mathbb{P}(\Cbb^n)$. The $2$-form $\Omega$ induces a $2$-form on $S^{2n-1}$, which is the lift of a unique $2$-form on $\mathbb{P}(\Cbb^n)$, making it a real symplectic manifold. More precisely
    \begin{align*}
        \Omega_{\mathbb{P}(\Cbb^n)} = \phi^* (\Omega|_{S^{2n-1}}) \ , 
    \end{align*}
where $\phi \colon \Cbb^n \setminus \{ 0\} \to S^{2n-1} $, $\phi(z) = |z|^{-1} z$. 
\end{example}

More details about Poisson manifolds can be found, for example, in \cite{Poisson-Structures}.


\subsection{Poisson manifolds and Lie theory}


\textbf{Dual space of a Lie algebra}. Let $\g$ be a finite dimensional Lie algebra over $\Rbb$, say $\dim \g = n$, and $\g^* = \hom (\g, \Rbb)$ the dual vector space. Consider the space of smooth functions $C^\infty (\g^*)$ and observe that one can embed $\g^* \to C^\infty (\g^*) $ as a subspace of linear functions
    \begin{align*}
        \g^* = \{ F_X \in C^\infty (\g^*) | \ F_X (\xi) = \xi(X), X \in \g , \xi \in \g^* \} \ .
    \end{align*}
There is a \emph{unique Poisson bracket} on $\g^*$ such that 
    \begin{align}\label{Poisson bracket on g^*}
        \{ F_X, F_Y \} = F_{[X,Y]} \ ,
    \end{align}
meaning that the assignment $X \mapsto F_X$ yields a Lie algebra homomorphism $\g \to C^\infty (\g^*)$.

\textbf{Description in coordinates - brackets}. Let $[X^k,X^l] = c^{kl}_m X^m$, where $X^k \in \g$ for all $k$. Then functions $\xi^1 = F_{X^1}, \ldots \xi^n = F_{X^n}$ form a system of linear coordinates on $\g^*$, and we can define a Poisson bracket on $\g^*$ via the Poisson brackets on $C^\infty (\g^*)$
    \begin{align*}
        \{ F,G\} = \sum_{k,l,m} C^{kl}_m \frac{\del F}{\del \xi^k} \frac{\del G}{\del \xi^l} \xi^m
    \end{align*}
so that 
    \begin{align*}
        \{ \xi^k, \xi^l\} = C^{kl}_m \xi^m \ .
    \end{align*}


\textbf{Description in coordinates - gradient}. For any $F \in C^\infty (\g^*)$, its \emph{gradient} 
    \begin{align*}
        \nabla F \colon \g^* \to \g 
    \end{align*}
is a $\g$-valued function on $\g^*$ defined by
    \begin{align*}
        <\nabla F (\xi), \eta > : = \frac{d}{dt}|_{t = 0} F(\xi + t \eta) \ , 
    \end{align*}
where $\xi, \eta \in \g^*$. In the coordinates
    \begin{align*}
         \nabla F = \sum_k  \frac{\del F}{\del \xi^k} X^k \ , 
    \end{align*}
which implies 
    \begin{align*}
        \{ F, G \} (\xi) = <\xi, [\nabla F (\xi), \nabla G (\xi)]> \ .
    \end{align*}

\textbf{G-invariant mappings on the $T^*G$}. Let $\g = \Lie (G) $ be a finite dimensional Lie algebra of a Lie group $G$. Every Lie group is parallelizable, i.e. there is a vector bundle isomorphism between $T^*G $ and $G \times \g^*$ as vector bundles over $G$. This is expressed by the following commutative diagram
                \begin{equation*}
                \begin{tikzcd}
                T^*G \arrow{r}{\varphi} \arrow[swap]{dr}{\pi} & G \times \g^* \arrow{d}{\pi_1}  \\
                    & G
                \end{tikzcd}
            \end{equation*}    
where $\varphi$ is the bundle isomorphism and $\pi_1 $ is the projections on the $1$st factor. We will denote by $\pi_2 \colon  G \times \g^* \to \g^*$ the projection on the second factor. Let $p \in T^*G$. Then 
    \begin{align}\label{local description of T^*G}
        \varphi (p) = (h, \xi) 
    \end{align}
for some $h \in G$ and $\xi \in \g^*$. Now consider the left translation automorphism $l_g \colon G \to G$, given by $l_g (h) = gh$. The tangent map of $l_g$ is denoted $(l_g)_* \colon TG \to TG$ and the cotangent map by $(l_g)^* \colon T^*G \to T^*G$. Notice that for $p \in T^*G$ we have
    \begin{align}\label{intermediate step - Lie theory}
        \left(\pi_2 \circ \varphi \circ (l_g)^* \right) (p) = \pi_2 (g^{-1}h, \xi) = \xi \ .
    \end{align}
This follows from the fact that $\g^*$ can be identified with the space of left-invariant $1$-forms on $G$, i.e. $1$-forms $\eta$ such that $ (l_g)^* \eta = \eta$. Thus from \eqref{local description of T^*G} and \eqref{intermediate step - Lie theory}, we have
    \begin{align}\label{intermediate step 2 - Lie theory}
        \pi_2 \circ \varphi \circ (l_g)^* =  \pi_2 \circ \varphi \ . 
    \end{align}
    
For the following considerations, we recall that any smooth map between manifolds, say $F \colon M \to N$, induces the corresponding algebra homomorphism in the opposite direction, $F^\bullet \colon C^\infty(N) \to C^\infty(M)$, which is given by the precomposition. Because $\g^*$ is a smooth manifold (as a vector space with global coordinates), we have the homomorphism
    \begin{align*}
        (\pi_2 \circ \varphi)^\bullet \colon C^\infty(\g^*) \to C^\infty(T^*G) \ .
    \end{align*}
We want to show that the algebra $C^\infty(\g^*)$ can be identified with the $G$-invariant subspace of the algebra $C^\infty(T^*G)^G$, i.e. 
   \begin{align*}
        C^\infty(\g^*) \cong C^\infty(T^*G)^G \ , 
    \end{align*}
where 
    \begin{align}\label{G invariance}
        C^\infty(T^*G)^G = \{ F \in C^\infty(T^*G)^G | \ \forall g \in G:  F \circ (l_g)^* = F \} \ . 
    \end{align}
To see this, consider $F \in C^\infty (\g^*)$ and denote 
    \begin{align*}
        \Tilde{F} : = (\pi_2 \circ \varphi)^\bullet F = F \circ \pi_2 \circ \varphi \in  C^\infty(T^*G) \ . 
    \end{align*}
Then $\Tilde{F}$ is $G$-invariant in the sense of \eqref{G invariance}. Indeed, using \eqref{intermediate step - Lie theory}, we obtain 
    \begin{align*}
         \Tilde{F} \circ (l_g)^* = F \circ \pi_2 \circ \varphi \circ (l_g)^* =  F \circ \pi_2 \circ \varphi =  \Tilde{F} \ ,
    \end{align*}
so we have an injection $C^\infty(\g^*) \xhookrightarrow{} C^\infty(T^*G)^G$. On the other hand, consider $\Tilde{H} \in C^\infty(T^*G)^G$. Then we can define $H \in C^\infty(\g^*) $
    \begin{align*}
        H (\xi) : =  \Tilde{H} \left( \varphi^{-1} ( e, \xi) \right) , \xi \in \g^* \ ,
    \end{align*}
where $\varphi$ is the trivialization of $T^*G$ as described in the above commutative diagram. The above definition of $H$ is unambiguous due to $G$-invariance of $ \Tilde{H}$. This assignment $\Tilde{H} \mapsto H$ is obviously injective, hence the bijection between $C^\infty(\g^*)$ and $C^\infty(T^*G)^G$. The brackets are preserved as well
    \begin{align*}
        \{F,H \} \circ \pi_2 = \{ F \circ \pi_2 , H \circ \pi_2\} 
    \end{align*}
for all $F,H \in C^\infty(\g^*)$. 

\begin{remark}
For a Lie group $G$, the cotangent bundle $T^*G$ is a symplectic manifold with non-degenerate Poisson structure (since $G$ is a smooth manifold).
\end{remark}


\begin{remark}
$\g^*$ does not have a structure of a Lie algebra, althoug it is a~dual vector space to the Lie algebra $\g$, it is not a coalgebra. On the other hand, $C^\infty(\g^*)$ is a Lie algebra.
\end{remark}

\textbf{One more example of a symplectic manifold}. Let $G$ be a Lie group,  $\g$ its Lie algebra, and $\g^*$ the corresponding dual vector space. The group $G$ acts on $\g$ by the \emph{adjoint action} $\Ad \colon G \to \operatorname{Aut} (\g)$. At $g \in G$, the map 
    \begin{align*}
        \Ad_g \colon \g \to \g
    \end{align*}
is the derivative at the identity element $e \in G$ of the conjugation map $G \to G$, given by $h \mapsto g h g^{-1}$. One can differentiate $\Ad$ at $e$, to obtain the \emph{adjoint action of $\g$ on $\g$}, i.e the map
    \begin{align*}
        \ad \colon \g \to \End (\g) \ .
    \end{align*}
In fact, one can show that $\ad \colon \g \to \D (\g) $, where $\D (\g)$ is the algebra of derivations on $\g$, which is the Lie algebra of $\operatorname{Aut}(\g)$. For $X,Y \in \g$, the $\ad$ map acts as the Lie bracket
    \begin{align*}
        \ad_X (Y) = [X,Y] \ .
    \end{align*}
Using $\Ad$, one can define the \emph{coadjoint action of $G$ on $\g^*$}, denoted  $\Ad^* \colon G \to \Aut (\g^*)$, by 
    \begin{align*}
        \Ad^*_g (\xi) (X) = \xi (\Ad_{g^{-1}} (X)), 
    \end{align*}
where $\xi \in \g^*$. Finally, using $\Ad^*$, we can define the \emph{coadjoint action of $\g$ on $\g^*$}, denoted $\ad^* \colon \g \to \End (\g^*)$, as
    \begin{align*}
       (\ad^*_X (\xi) (Y) ) : = - \xi ([X,Y]) \ ,  
    \end{align*}
$X,Y \in \g$ and $\xi \in \g^*$.


\subsection{Symplectic foliation on $\g^*$}

The brackets on $\g^*$ are non-symplectic. 

\textbf{Symplectic structure on the coadjoint orbit}. Now suppose that $G$ is a~connected Lie group and denote by $\mathcal{O} \subset \g^*$ the $G$-orbit of the coadjoint action $\Ad^*$. Then for all $\xi \in \mathcal{O}$, the tangent spaces of $\mathcal{O}$ at $\xi$ is given by 
    \begin{align*}
        T_\xi \mathcal{O} = \{ \ad^*_X (\xi) \in \g^* | X \in \g \} \ . 
    \end{align*}
One can define a symplectic structure on $\mathcal{O}$:
    \begin{align*}
        \Omega_{\xi} \colon T_\xi \mathcal{O} \times T_\xi \mathcal{O} \to \Rbb \ ,
    \end{align*}
as follows. For arbitrary $(X, \xi), (Y, \xi) \in  T_\xi \mathcal{O}$, define
    \begin{align}\label{symplectic structure on the coadjoint orbit}
         \Omega_{\xi} ((X, \xi), (Y, \xi)) : = \xi (\ad_X(Y)) = \xi([X,Y]) \ .
    \end{align}
This structure is usually atributed to Kostant \cite{KostantBertram}, Kirillov \cite{Kirillov1968, Kirillov2004}, and Souriau \cite{Souriau1976}. For brevity, we will refere to this symplectic structure as KKS. 

\subsubsection{Coadjoint invariant functions}

A function $F \in C^\infty (g^*)$ such that  
    \begin{align*}
        \{ F,H \} = 0 \ \forall H \in C^\infty (g^*) 
    \end{align*}
is called a \emph{Casimir function}. Let $F$ be a Casimir function. Then for arbitrary $X \in \g$, the corresponding linear function\footnote{This is the evaluation map $F_X \colon \g^* \to \Rbb$ (defined above), $F_X (\xi) = \xi (X)$.} $F_X$ satisfy 
    \begin{align*}
        \{F, F_X\} = - \{ F_X , F\} = - \xi_{F_X} (F) = \ad^*_X(F) = 0\ , 
    \end{align*}
where $\xi_{F_X} = \{ F_X, - \}$ is the Hamiltonian function corresponding to $F_X$. This means that the Casimir functions are \emph{coadjoint-invariant functions}.

\begin{remark}
If $G$ is a semisimple Lie group, then there is an open denset subset $U \subset \g^*$, stable under $G$, and the orbits are separated by Casimir functions. This yields a foliation of $\g^*$. Let us note that this foliation is not a smooth fibration and has a rather complicated structure.
\end{remark}    


\textbf{Poisson bracket on $\g^*$ and $\mathcal{O}$}. There is a link between the Poisson bracket on $\g^*$ \eqref{Poisson bracket on g^*} and on the coadjoint orbit $\mathcal{O}$. Denote by $\iota \colon \mathcal{O} \to \g^*$ the embedding of the orbit into $\g^*$, and by $\iota^* \colon C^\infty(\g^*) \to C^\infty(\mathcal{O})$ the corresponding algebra map. Then we have
    \begin{align*}
        \iota^*\{ f, g \}_{ C^\infty(\g^*)}  = \{ \iota^*f ,\iota^* g \}_{C^\infty(\mathcal{O} )} \ .
    \end{align*}
The left-hand side is the Poisson bracket on $C^\infty(\g^*)$, restricted to $\mathcal{O}$. On the right-hand side is the \emph{symplectic Poisson bracket} on $C^\infty(\mathcal{O} )$, given by the symplectic structure $\Omega$ on $\mathcal{O}$, given by \eqref{symplectic structure on the coadjoint orbit}. We will now describe two examples of this construction. 

\begin{example}
Let $G = \SO(3)$. The Lie algebra is $\g = \so (3)$ and it is isomorphic with its dual $\so (3) \cong \so^* (3)$ via the Killing form. Moreover, $\so^* (3) \cong \Rbb^3$. The orbits of the coadjoint action are described by the equation $x^2 + y^2 + z^2 = R^2$. Hence for $R = 0$ we have a $0$-dimensional symplectic leaf (a  e point). For $R > 0$ we have $2$-dimensional symplectic leaves (the coadjoint orbits). On $\Rbb^3$ we have the volume form $dx \wedge dy \wedge dz$, which if restricted to the orbits, yields
    \begin{align*}
        \Omega_R = \frac{1}{R^2}(x dy \wedge dz + y dz \wedge dx + z dx \wedge dy )
    \end{align*}
\end{example}
The Casimir functions on $\so(R)$ are given by $F = x^2 + y^2 + z^2$. Hence the Casimir elements are parametrized by $R > 0$. 

\begin{example}
Consider $G = \Sl_2 (\Rbb)$, given by 
    \begin{align*}
        \Sl_2 (\Rbb) = \{ A \in \Mat_2 (\Rbb) | \ \det A = 1 \} \ , 
    \end{align*}
The Lie algebra is 
    \begin{align*}
        \sl_2(\Rbb) = \{ A \in  \Mat_2 (\Rbb) | \tr A = 0 \}  \cong \Rbb^3 \ . 
    \end{align*}
The $\Sl_2 (\Rbb)$-orbits in $\Rbb^3$ are cones and hyperboloids, with the symplectic structure given by KKS structure. 
\end{example}

\begin{remark}
For a finite dimensional Lie algebra $\g$ and the corresponding linear extension of the bracket to the symmetric algebra $S(\g)$, the Casimir elements can be identified with the $\g$-invariant subspace of $S(\g)$
    \begin{align*}
        \Cas (\g) : = \operatorname{Z}(\U (\g)) \cong S(\g)^{\g}
    \end{align*}
 where $\operatorname{Z}(\U (\g))$ denotes the center of the universal envelopping algebra of $\g$ (see \ref{UP of the universal enveloping algebra}). For a reductive $\g$ 
    \begin{align*}
        HP^k \left(S(\g) \right) \cong H_{CE}^k (\g) \otimes_{\Kbb} \Cas (\g) \ , 
    \end{align*}
where $H_{CE}$ is the Chevalley-Eilenberg cohomology. 
\end{remark}

\section{Differential Calculus on Poisson Manifolds}

Let $M$ be a smooth manifold. Then there is a natural isomorphism $P_m \cong \Gamma (M, S^m (TM))$, where $S^m  (TM)$ is the $m$-symmetric power of the tangent bundle $TM$. Then the Poisson brackets on $C^\infty(T^*M)$ can be reformulated as Poisson brackets on symmetric tensor fields.

We also consider \emph{m-vector fields} over $M$ \cite{NRoy-Properties-of-Multisympl-mfld}, i.e. sections of the $m$-th exterior bundle
    \begin{align*}
        \Lambda^m (TM) \to M \ , 
    \end{align*}
The set of all $m$-vector fields is denoted $\mathfrak{X}^m(M) : = C^\infty(M,\Lambda^m (TM))$. For $m = 1$ we obtain the standard notion of vector fields, i.e. $\mathfrak{X}^1(M) = C^\infty(M, TM)$. Taking the \emph{Whitney sum} of all the exterior powers of $TM$ gives the exterior bundle 
    \begin{align*}
        \Lambda(TM) : = \oplus_{m \geq 0} \to M \ .
    \end{align*}
Ssections of this bundle are \emph{multivector fields}. The set of all multivector fields, $\mathfrak{X}(M): = C^\infty(M,\Lambda (TM))$, is an algebra with respect to the exterior product \eqref{exterior product}. There is a dual construction leading to differential forms over $M$. If we start with $T^*M$ instead of $TM$, we get \emph{differential $m$-forms over $M$} (shortly just \emph{$m$-forms}) as sections of the bundle 
    \begin{align*}
        \Lambda^m (T^*M) \to M \ .
    \end{align*}
The set of all differential $m$-forms is denoted $ \Omega^m (M) : =  C^\infty(M,\Lambda^m (T^*M))$. Note that $m$-vector fields are skew-symmetric polylinear functions on $1$-forms. Taking sections of the bundle of all $m$-forms, $0 \leq m \leq \dim M$, 
    \begin{align*}
        \Lambda (T^*M) : = \oplus_m \Lambda^m (T^*M) \to M \ ,
    \end{align*}
we obtain the algebra (with respect to the exterior product) of differential forms over $M$ 
    \begin{align*}
        \Omega (M) =  \oplus_{m \geq 0} \Omega^m (M) \ ,  
    \end{align*}

\begin{remark}
Both algebras $\mathfrak{X}(M)$, and $\Omega(M)$, are $C^\infty$-modules.
\end{remark}
  
Now suppose that $(M, \{ \cdot , \cdot \} )$ is a Poisson manifold, $\mathcal{H} \colon T^*M \to TM$ the corresponding Hamiltonian mapping. One can consider a bivector field $\pi_* \in \mathfrak{X}^2(M)$, such that 
    \begin{align*}
        \{ F,G\} : = \sum_{k,l} H^{kl} \del_k F \del_l G \ .
    \end{align*}
We would like to see the above studied Jacobi identity in a more invariant, intrinsic way using certain structure on the space of multivector fields $\mathfrak{X}(M)$. This structure is called a \emph{Lie super-Schouten bracket}, or \emph{Schouten-Nijenhuis bracket}, or shortly just \emph{Schouten bracket},
    \begin{align*}
        [\![- , - ]\!] \colon \mathfrak{X}^p(M) \times \mathfrak{X}^q(M) \to \mathfrak{X}^{p+q-1} (M) \ .
    \end{align*}
We define it inductively using the following properties of multivector fields
    \begin{enumerate}
        \item $s \wedge t = (-1)^{|s||t|} t \wedge s$, where $|s|$ is the degree\footnote{meaning $s \in \mathfrak{X}^{|s|} (M)$} of $s$,
        \item $ [\![s , t]\!] = (-1)^{|s|(|t|+1)+|t|}  [\![t , s]\!] $ (graded commutativity),
        \item $ [\![s , t \wedge r ]\!] =  [\![s , t]\!] \wedge r + (-1)^{|t|( |s| - 1 )} t \wedge  [\![s , r]\!]$,
        \item $ [\![s \wedge t,  r ]\!] =  s \wedge [\![ t,  r ]\!] + (-1)^{|t|( |r| - 1 )} [\![s , r]\!] \wedge t $. 
    \end{enumerate}
For degrees $\leq 1$ we further define
    \begin{enumerate}
        \item $|s| = |t| = 1 : [\![s , t]\!] : = [s,t] $, where $[-,-]$ is the Lie bracket
        \item $|s| =1, |t| = 0 : [\![s , t]\!] : = s(t) = \mathcal{L}_s(t) $, where $\mathcal{L}$ is the Lie derivative
        \item $|s| = |t| = 0 : [\![s , t]\!] : = 0 $. 
    \end{enumerate}
Multivector fields $s$ and $t$ are called decomposable if
    \begin{align*}
        s = s_1 \wedge \ldots \wedge s_p && t = t_1 \wedge \ldots \wedge t_q \ ,
    \end{align*}
for some $p,q$, and where all $s_i, t_j \in  \mathfrak{X}^1 (M)$. In this case 
    \begin{align*}
         [\![s , t]\!] = \sum_{k=1}^p \sum_{l=1}^q  [s_k, t_l] \wedge s_1 \wedge \ldots \wedge \hat{s_k} \wedge \ldots \wedge s_p \wedge t_1 \wedge \ldots \wedge \hat{t_l} \wedge \ldots \wedge t_q \ .
    \end{align*}

There is at most one Schouten bracket satisfying the above properties.


Once we properly defined the space of multivectors on a manifold and the notion of Schouten bracket, we can introduce the notion of Poisson tensor.  
\begin{definition}[Poisson tensor]\label{Poisson tensor}
Let $M$ be a Poisson manifold with Poisson bracket $\{-,- \}$ and with the Schouten bracket $ [\![- , -]\!]$. A $2$-vector field $\pi \in \mathfrak{X}(M)$ is called a Poisson tensor, if for all $f,g \in C^\infty(M)$
    \begin{align*}
        \{f,g \} = [\![ [\![\pi , f ]\!], g  ]\!] \ .
    \end{align*}
\end{definition}

\begin{theorem}
The Poisson bracket $\{ -,-\}$ satisfies the Jacobi identity iff $[\![ \pi , \pi]\!] = 0$. 
\end{theorem}

\subsection{Coordinate-free construction of the Schouten bracket}

Consider the exterior algebra $\Omega(M)$. Let $\omega \in \Omega^p(M)$ and $X \in \mathfrak{X}^1(M)$. We associate with $X$ two operators on $\Omega(M)$
    \begin{align*}
        \iota_X & \colon \Omega^p(M) \to \Omega^{p-1}(M) &&  \text{(interior product)}  \ , \\
        \mathcal{L}_X &\colon \Omega^p(M) \to \Omega^p(M) && \text{(Lie derivative)}  \ . 
    \end{align*}
The interior product is defined by 
    \begin{align*}
         (\iota_X \omega) (X_2, \ldots, X_p) : = \omega(X_1, X_2, \ldots, X_p)\ ,
    \end{align*}
where all $X_i \in \mathfrak{X}^1(M)$. The Lie derivative is defined as
    \begin{align*}
        \mathcal{L}_X \omega : = \frac{d}{dt}|_{t = 0} \phi^*_t(\omega) \ ,
    \end{align*}
where $\phi^*_t$ denotes the pullback along the flow $\phi_t$ of $X$. This means 
    \begin{align*}
        (\mathcal{L}_X \omega) (X_1, \ldots, X_p ) =  \mathcal{L}_X (\omega (X_1, \ldots, X_p ) ) - \sum_k \omega(X_1, \ldots, [X,X_k], \ldots, X_p) \ . 
    \end{align*}
We denote by $d$ the de Rham differential $d \colon \Omega^p(M) \to \Omega^{p+1}(M) $, given by 
    \begin{align*}
        d \omega (X_1, \ldots, X_{p+1}) & = \sum_k \mathcal{L}_{X_k}   (\omega (X_1, \ldots,\hat{X_k} , \ldots, X_{p+1} ) ) \\
            & + \sum_{k < l} (-1)^{k+l} \omega( [X_k, X_l], X_1, \ldots,\hat{X_k} , \ldots, \hat{X_l} , \ldots, X_{p+1} ) \ . 
    \end{align*}
    
\begin{theorem}[Cartan triple $(\iota_X, \mathcal{L}_X, d)$]\label{Cartan triple}
The following identities are always satisfied
    \begin{enumerate}
        \item $ \iota_X \circ \iota_Y + \iota_Y \circ \iota_X = 0 $,
        \item $[\mathcal{L}_X, \mathcal{L}_Y] = \mathcal{L}_{[X,Y]}$, 
        \item $[\mathcal{L}_X,\iota_Y] = \iota_{[X,Y]}$,
        \item $[\mathcal{L}_X, d] = 0$,
        \item $d^2 = d \circ d = 0$,
        \item $\mathcal{L}_X = \iota_x \circ d + d \circ \iota_X$.
    \end{enumerate}
\end{theorem}

We want now to define the analogue of this "differential calculus" for multivectors. For decomposable $t \in \mathfrak{X}^p (M), t = X_1 \wedge \ldots \wedge X_p$, we define     
    \begin{align*}
        \iota_t = \iota_{X_1} \circ \ldots \circ \iota_{X_p} \ . 
    \end{align*}
The above can be extend by linearity for arbitrary $t \in \mathfrak{X}^p (M)$. This definition is correct because of the first property of the above theorem and yields an operator
    \begin{align*}
        \iota_t \colon \Omega^m (M) \to \Omega^{m-p} (M) \ ,
    \end{align*}
which further satisfies 
    \begin{align*}
        \iota_{t \wedge s} = \iota_t \circ \iota_s \ . 
    \end{align*}
    
To define the Lie derivative $\mathcal{L}_t$ along a multivector field, one can start with the notion of a \emph{graded operator} 
    \begin{align*}
        \del \colon \Omega^m (M) \to \Omega^{m+r} (M) .
    \end{align*}
In this case, $\del$ is called \emph{a graded operator of degree} $\deg \del = |\del| = r$. If in addition $\del$ satisfies 
    \begin{align*}
        \del (\omega \wedge \theta) = \del \omega \wedge \theta + (-1)^{mr} \omega \wedge \del \theta,
    \end{align*}
where $\omega \in \Omega^m (M)$, then $\del$ is called a \emph{graded derivative of degree $r$}. For example, $\iota_x, \mathcal{L}_x, d$ are graded derivatives of the following degrees $|\iota_x| = -1, |\mathcal{L}_x| = 0, |d| = 1$. 

We recall that \emph{graded brackets} of two operators are defined by 
    \begin{align*}
        [\Delta, \nabla] = \Delta \circ \nabla - (-1)^{|\Delta | |\nabla|} \nabla \circ \Delta \ . 
    \end{align*}

\begin{remark}
If $\del$ and $D$ are graded operators of degrees $|\del|$ and $|D|$, respectively, then  $[\del, D]$ is a graded derivation of degree $|\del| + |D|$.
\end{remark}

The property 6. of the previous theorem now reads 
    \begin{align*}
         [\iota_X,d] = \iota_X \circ  d + d \circ \iota_X = \mathcal{L}_X 
    \end{align*}
for a vector field $X$. Using the graded bracket we now define the Lie derivative along a multivector field $t$ 
    \begin{align*}
        \mathcal{L}_t : = [\iota_t, d] \ . 
    \end{align*}
If $t \in \mathfrak{X}^p(M)$ then 
    \begin{align*}
        \mathcal{L}_t \colon  \Omega^m (M) \to \Omega^{p-m+1} (M) \ ,
    \end{align*}
meaning that $ \mathcal{L}_t$ is a graded derivation of degree $| \mathcal{L}_t| = m-1$. 

We can proceed with the definition of a bracket on the algebra of multivectors (denoted by the same symbol as the Schouten bracket)
    \begin{align*}
         [\![- , - ]\!] \colon \mathfrak{X}^p (M) \times \mathfrak{X}^q(M) \to \mathfrak{X}^{p+q-1} (M) \ .
    \end{align*}
If $s \in \mathfrak{X}^p (M), t \in \mathfrak{X}^q(M)$, then there is a unique multivector field $[\![t , s ]\!] \in \mathfrak{X}^{p+q-1} (M)$ satisfying 
    \begin{align*}
        \iota_{[\![s , t ]\!]} = [\mathcal{L}_s, \iota_t] \ ,
    \end{align*}
where the right-hand side is given by the graded commutator, i.e.
    \begin{align*}
        [\mathcal{L}_s, \iota_t] = \mathcal{L}_s \circ \iota_t - (-1)^{(p-1)q} \iota_t \circ  \mathcal{L}_s \ . 
    \end{align*}
When the multivectors are decomposable as $s = s_1 \wedge \ldots \wedge s_p$ and $ t = t_1 \wedge \ldots \wedge t_q $ then 
    \begin{align*}
         [\![s , t]\!] = \sum_{k=1}^p \sum_{l=1}^q  [s_k, t_l] \wedge s_1 \wedge \ldots \wedge \hat{s_k} \wedge \ldots \wedge s_p \wedge t_1 \wedge \ldots \wedge \hat{t_l} \wedge \ldots \wedge t_q \ .
    \end{align*}
This defines $ [\![- , - ]\!]$ on the whole $\mathfrak{X}(M)$ by bilinear extension. Moreover, it satisfies the \emph{graded commutativity}
    \begin{align*}
        [\![s , t]\!] = -(-1)^{(|s|-1)(|t|-1)} [\![t , s]\!] \ , 
    \end{align*}
as well as the analogy of the second property of the \ref{Cartan triple} theorem
    \begin{align*}
         [\mathcal{L}_s, \mathcal{L}_t] = \mathcal{L}_{[\![s , t]\!]} \ ,
    \end{align*}
and also satisfies the \emph{graded Jacobi identity}
    \begin{align*}
        (-1)^{(|s|-1)(|r|-1)} [\![s , [\![ t, r ]\!] ]\!] + (-1)^{(|t|-1)(|s|-1)} [\![t , [\![ r, s ]\!] ]\!] + (-1)^{(|r|-1)(|t|-1)} [\![r , [\![ s, t ]\!] ]\!] = 0 \ .
    \end{align*}
    
For more details about the Nijenhuis-Schouten bracket, see for example \cite{Ib_ez_1998, Krasil'shchik1988}

\section{Modified Double Poisson Brackets}
The main reference for this section is work of S. Arthamonov (2017) \cite{ARTHAMONOV2017212}.

\subsection{Poisson Brackets for General Associative Algebras}.

Let $\A$ be an associative algebra. Conventional definition of Poisson bracket becomes too restrictive when $\A$ is essentially non-commutative.

The standard "set" axioms (cf. definition 3.1) meets with a problem of "non-trivial" example existence. 
This chapter is devoted to some constructions of "non-commutative" Poisson algebra
structures. This subject which has started with the paper of Ping Xu \cite{Ping-Xu}, where the author introduces a notion of Poisson structure on noncommutative algebras, and studies some of its properties and applications. Given an associative algebra $\A$ demonstrated that the Hochschild cohomology $HH^*(\A,\A)$ can be provided with a graded Lie algebra structure by means of the so-called $G-$bracket. This bracket, which was first introduced by M. Gerstenhaber, is the analogue of the Schouten bracket for multivector fields. A Poisson structure on $\A$ is then defined as an element of $HH^2(\A,\A)$  whose  $G$-bracket  with  itself  vanishes.   It  was  shown  that  such a Poisson structure induces an ordinary Poisson bracket on the center of $\A$.

We shall discuss the drawback of the naive definition of a Poisson structure on a non–commutative algebra $\A$.

First, we describe the following important lemma which appeared in the paper \cite{Voronov1995OnTP} and therefore (by the famous Arnold's statement) is attributed to Victor Ginzburg. 

\subsubsection{Ginzburg-Voronov lemma}.  
\begin{lemma}
If $\A$ is any Poisson algebra, then for all $a,b,c,d \in \A$ the following identity holds
    \begin{align*}
        [a,c] \{b,c\} = \{ a,c\}[b,d] \ .
    \end{align*}
\end{lemma}

\begin{proof}
Take the Poisson bracket \{ ab,cd\}. By the derivation property of the bracket we have
    \begin{align*}
         \{ ab,cd\} = a\{b,cd \} + \{ a,cd\} b = ac\{b,d \} + a\{b,c \}d + c\{a,d \}b + \{ a,c\}db \ .
    \end{align*}
On the other hand
    \begin{align*}
        \{ ab,cd\} = c\{ ab,d\} + \{ ab,c\}d =  ca\{b,d\} + c\{a,d\}b + a\{ b,c\}d + \{ a,c\}bd \ .
    \end{align*}
Subtracting the two equations yields the result.
\end{proof}

\begin{definition}
An algebra $\A$ is \emph{prime}, if the product of nonzero ideals in nonzero.
\end{definition}
\begin{definition}
An algebra $\A$ is \emph{simple} if it has no non-trivial two-sided ideals and the algebra product is non-trivial. 
\end{definition}

For example, the algebra given by
    \begin{align*}
        \{  \begin{pmatrix}
            0 & a \\
            0 & 0 
            \end{pmatrix} | \  a \in \Rbb \}     
    \end{align*}
is not simple, as the matrix product is always trivial in this case. 

The following theorem is due to D. R. Farkas and G. Letzter \cite{Farkas1998RingTF}.

\begin{theorem}
Let $\A$ be a prime and simple noncommutative Poisson algebra. Then for all $c,d \in \A$
    \begin{align*}
        \{c,d\} = \lambda [c,d]
    \end{align*}
for some $\lambda \in \operatorname{Z}_{P}(\A)$\footnote{See def. \eqref{Poisson center} for the definition of Poisson center $\operatorname{Z}_{P}(\A)$}
\end{theorem}

\begin{definition}[\cite{CRAWLEYBOEVEY2011205}]
A map
    \begin{align*}
        \{,\} \colon \A \otimes\A \rightarrow \A
    \end{align*}
is an $H_0$-Poisson bracket if for all $a,b,c \in \A$
    \begin{enumerate}
        \item $\{a,bc\}=b\{a,c\}+\{a,b\}c$ (Right Leibnitz identity),
        \item $\{a,\{b,c\}\}-\{b,\{a,c\}\}=\{\{a,b\},c\}$ (Left Loday-Jacobi identity),
        \item $\{a,b\}+\{b,a\}\equiv0\bmod [\A,\A]$,
        \item $\{ab,c\}=\{ba,c\}$.
    \end{enumerate}
\end{definition}

\begin{corollary}[\cite{CRAWLEYBOEVEY2011205}]
An $H_0$-Poisson bracket induces a Lie Algebra structure $\{\_\}^{Lie} \colon \mathcal \A_\natural\otimes \A_\natural\rightarrow\A_\natural$ on \emph{abelianization} $\A_\natural : =\A/[\A,\A]$ of $\A$.
\end{corollary}

\subsubsection{Representation scheme.}

Following philosophy by M.~Kontsevich \cite{Kontesvich1993}, any algebraic property that makes geometric sense is mapped to its commutative counterpart by \emph{Representation Functor}
    \begin{gather*}
        \mathrm{Rep}_N:\quad\mathrm{fin.\;gen.\;Associative\; algebras} \rightarrow\mathrm{Affine\; schemes} \ ,\\
        \mathrm{Rep}_N(\mathcal A)=Hom(\mathcal A,Mat_N(\mathbb C)) \ .
    \end{gather*}

It assigns to a finitely generated associative algebra $\mathcal A=\langle x^{(1)},\dots,x^{(k)}\rangle/\mathcal R$ a scheme of its' $N\times N$ matrix representations. Let
    \begin{align}
        \varphi(x^{(i)})=\left(\begin{array}{ccc}
        x^{(i)}_{11}&\dots&x^{(i)}_{1N}\\
        \vdots&&\vdots\\
        x^{(i)}_{N1}&\dots&x^{(i)}_{NN}
        \end{array}\right) \ .
    \end{align}
Representations of $\mathcal A$ then form an affine scheme $\mathcal V$  with a coordinate ring $\mathbb C[\mathcal V]:=\mathbb C\left[x^{(i)}_{j,k}\right]/\varphi(\mathcal R)$. Denote as $\mathbb C_{\mathcal V}$ - the corresponding sheaf of rational functions.

\subsubsection{Moduli Space of Representations}

Change of basis corresponds to the action $GL_N(\mathbb C)\circlearrowleft Mat_N(\mathbb C)$,
    \begin{align*}
         M\rightarrow gMg^{-1} \ .
    \end{align*}
It induces $GL_N(\mathbb C)\circlearrowleft\mathbb C[\mathcal V]$.
The invariant subalgebra $\mathbb C[\mathcal V]^{GL_N(\mathbb C)}\subset\mathbb C[\mathcal V]$ is then a coordinate ring of the corresponding moduli space of representations.
    \begin{align*}
        \varphi_0:\mathcal A_\natural\rightarrow \mathbb C[\mathcal V]^{GL_N(\mathbb C)},\quad \varphi_0(x)=\textrm{Tr}\,\varphi(x) \ .
    \end{align*}
    
\begin{lemma}[Procesi, 1976]
Subset $\varphi_0(\mathcal A_\natural)$ generates $\mathbb C[\mathcal V]^{GL_N(\mathbb C)}$.
\end{lemma}

\begin{proposition}[Crawley-Boewey, 2011]
An $H_0$-Poisson bracket induces a~conventional Poisson bracket
    \begin{align*}
        \{,\}^{inv}:\mathbb C[\mathcal V]^{GL_N(\mathbb C)}\otimes \mathbb C[\mathcal V]^{GL_N(\mathbb C)}\rightarrow \mathbb C[\mathcal V]^{GL_N(\mathbb C)} \ .
    \end{align*}
\end{proposition}

\subsection{Double Poisson Brackets}

\begin{definition}\label{main}[M.~Van~den~Bergh, 2008]
A map $\db{,}:\mathcal A\otimes\mathcal A\rightarrow \mathcal A\otimes\mathcal A$ is a~double Poisson bracket if for all $a,b,c\in\mathcal A$:
    \begin{enumerate}
        \item $\db{a,b}=-\db{b,a}^{op},$
        \item $\db{ab,c}=(1\otimes a)\db{b,c}+\db{a,c}(b\otimes1),$
        \item $\db{a,bc}=(b\otimes1)\db{a,c}+\db{a,b}(1\otimes c),$
        \item $R_{12}R_{23}+R_{31}R_{12}+R_{23}R_{31}=0,$ where $R_{m,n}(a_1\otimes\dots\otimes a_k)=a_1\otimes\dots\otimes a_{m-1}\otimes\db{a_m,a_n}'\otimes\dots\otimes
            \db{a_m,a_n}''\otimes\dots\otimes a_k.$
    \end{enumerate}
\end{definition}

\begin{proposition}[\cite{DoublePoissonBergh}]
Double Poisson bracket induces a conventional Poisson bracket
    \begin{align*}
        \{,\}^{\mathcal V}:\mathbb C_{\mathcal V}\otimes\mathbb C_{\mathcal V}\rightarrow\mathbb C_{\mathcal V}\qquad \left\{x^{(m)}_{ij},x^{(n)}_{kl}\right\}^{\mathcal V}=\varphi\left(\db{x^{(m)}\otimes x^{(n)}}\right)_{(kj),(il)} \ .
    \end{align*}
\end{proposition}

If $\A=\mathbb C<x_1,\ldots,x_m>$  is the free associative algebra, then $\mathbb C[{\rm Rep}_n(\A)] = \mathbb C[x_{i,\alpha}^j]$
where $1\leq \alpha \leq m.$ 

If $\db{x_{\alpha},x_{\beta}}$ is a double Poisson bracket on $\A=\mathbb C<x_1,\ldots,x_m>$, then, using the Sweedler 
convention and drop the sign of sum, we obtain the conventional  Poisson brackets on $\mathbb C[{\rm Rep}_n(\A)] $:
$$\lbrace x_{i,\alpha}^j,x_{k,\beta}^l\rbrace = \db{x_{\alpha},x_{\beta}}_{k}^{'j}\db{x_{\alpha},x_{\beta}}_{i}^{''l}$$

 \subsection{Quadratic double Poisson brackets}

Let $\A=\mathbb C<x_1,\ldots,x_m>$ be the free associative algebra. If double brackets $\db{x_i, x_j}$ between all generators are fixed, then the bracket between two arbitrary elements of $\A$ 
is uniquely defined by identities (\ref{main}). It follows from  (\ref{main}) that constant, linear, and quadratic double brackets are defined by 
\begin{equation} \label{dconst}
\db{x_i,x_j} = c_{ij} 1\otimes 1, \qquad c_{i,j}=-c_{j,i},
 \end{equation}
\begin{equation} \label{dlin}
\db{x_i,x_j} = b_{ij}^k x_k\otimes 1 - b_{ji}^k1\otimes x_k,
 \end{equation}
and  
\begin{equation} \label{dquad}
\db{x_{\alpha}, x_{\beta}} =r_{\alpha \beta}^{u v} \, x_u \otimes x_v+a_{\alpha \beta}^{v u} \, x_u x_v\otimes 1-a_{\beta \alpha}^{u v} \,1\otimes  x_v x_u,   \end{equation}
where
\begin{equation}\label{r1}
r^{\sigma \epsilon}_{\alpha\beta}=-r^{\epsilon\sigma}_{\beta\alpha},
\end{equation}
correspondingly.  The summation with respect to repeated indexes is assumed.

It is easy to verify that the bracket (\ref{dconst}) satisfies (\ref{main}) for any skew-symmetric tensor $c_{ij}$. For the bracket (\ref{dlin}) the condition (\ref{main}) is equivalent 
to the identity 
\begin{equation}\label{r0}
b^{\mu}_{\alpha \beta} b^{\sigma}_{\mu \gamma}=b^{\sigma}_{\alpha \mu} b^{\mu}_{\beta \gamma},
\end{equation}
which means that  $b^{\sigma}_{\alpha \beta}$
are structure constants of an associative algebra ${\cal A}$.  

\begin{proposition}
 The bracket (\ref{dquad}) satisfies (\ref{main}) iff the following relations hold:

\begin{equation}\label{r2}
r^{\lambda\sigma}_{\alpha\beta}
r^{\mu\nu}_{\sigma\tau}+r^{\mu\sigma}_{\beta\tau} r^{\nu\lambda}_{\sigma\alpha}+r^{\nu\sigma}_{\tau\alpha} r^{\lambda\mu}_{\sigma\beta}=0,
\end{equation}

\begin{equation}\label{r3}
a^{\sigma\lambda}_{\alpha\beta} a^{\mu\nu}_{\tau\sigma}=a^{\mu\sigma}_{\tau\alpha} a^{\nu\lambda}_{\sigma\beta},
\end{equation}
\begin{equation}\label{r4}
a^{\sigma\lambda}_{\alpha\beta} a^{\mu\nu}_{\sigma\tau}=a^{\mu\sigma}_{\alpha\beta} r^{\lambda\nu}_{\tau\sigma}+a^{\mu\nu}_{\alpha\sigma}
r^{\sigma\lambda}_{\beta\tau}
\end{equation}
and
\begin{equation}\label{r5}
a^{\lambda\sigma}_{\alpha\beta} a^{\mu\nu}_{\tau\sigma}=a^{\sigma\nu}_{\alpha\beta} r^{\lambda\mu}_{\sigma\tau}+a^{\mu\nu}_{\sigma\beta}
r^{\sigma\lambda}_{\tau\alpha}.
\end{equation}
\end{proposition}
The conventional Poisson bracket corresponding to any double Poisson bracket (\ref{dquad})  can be defined on  $\mathbb C[{\rm Rep}_n (\A)]$ by the following way \cite{odesskii2012bi}: 
\begin{equation}\label{Poisson}
\{x^j_{i,\alpha},x^{j^{\prime}}_{i^{\prime},\beta}\}=
r^{\gamma\epsilon}_{\alpha\beta}x^{j^{\prime}}_{i,\gamma}x^j_{i^{\prime},\epsilon}+
a^{\gamma\epsilon}_{\alpha\beta}x^k_{i,\gamma}x^{j^{\prime}}_{k,\epsilon}\delta^j_{i^{\prime}}-
a^{\gamma\epsilon}_{\beta\alpha}x^k_{i^{\prime},\gamma}x^{j}_{k,\epsilon}\delta^{j^{\prime}}_i
\end{equation}
where $x^j_{i,\alpha}$ are entries of the matrix $x_{\alpha}$ and $\delta^{j}_i$ is the Kronecker delta-symbol. Relations  (\ref{r1}), (\ref{r2})-(\ref{r5}) hold iff (\ref{Poisson}) 
is a Poisson bracket.

We may interpret the four index tensors $r$ and $a$ as:
 
 1) operators on $V\otimes V$, where $V$ is an $m$-dimensional vector space; 
 
 2)  elements of $Mat_m(\mathbb C)\otimes Mat_m(\mathbb C)$;
 
  3) operators on $Mat_m(\mathbb C)$.

For the first interpretation let $V$ be a linear space with a basis $e_\alpha,~\alpha=1,...,m$. Define linear
operators $r,~a$ on the space $V\otimes V$ by 
$$r(e_\alpha\otimes
e_\beta)=r^{\sigma \epsilon}_{\alpha \beta}e_\sigma\otimes e_\epsilon,\qquad a(e_\alpha \otimes e_\beta)=a^{\sigma \epsilon}_{\alpha \beta}e_\sigma\otimes e_\epsilon.
$$ 
Then the identities (\ref{r1}), (\ref{r2})-(\ref{r5}) 
can be written as 
\begin{equation}
\begin{array}{c} \label{rraa}
r^{12}=-r^{21},~~~r^{23}r^{12}+r^{31}r^{23}+r^{12}r^{31}=0,\\[5mm]
a^{12}a^{31}=a^{31}a^{12},\\[5mm]
\sigma^{23}a^{13}a^{12}=a^{12}r^{23}-r^{23}a^{12},\\[5mm]
a^{32}a^{12}=r^{13}a^{12}-a^{32}r^{13}.
\end{array}
\end{equation}
Here all operators act in $V\otimes V\otimes V$,  $\sigma^{ij}$ 
means the transposition of $i$-th and $j$-th components of the tensor
product, and $a^{ij},~r^{ij}$ mean operators $a,~r$ acting in the
product of the $i$-th and $j$-th components. 

Note that first two relations mean that the tensor $r$ should be skew-symmetric solution of the classical associative Yang-Baxter equation \cite{Agui}.

In the second interpretation we consider the following elements from  $Mat_m(\mathbb C)\otimes Mat_m(\mathbb C)$:
$r=r^{km}_{ij} e^{i}_{k} \otimes e^{j}_{m}, \quad a=a^{km}_{ij} e^{i}_{k} \otimes e^{j}_{m},$ where $e^{i}_{j}$ are the matrix unities:
$e^j_i e^m_k=\delta^j_k e^m_i.$  Then (\ref{r1}), (\ref{r2})-(\ref{r5}) are equivalent to (\ref{rraa}), where 
tensors belong to $Mat_m(\mathbb C)\otimes Mat_m(\mathbb C)\otimes Mat_m(\mathbb C).$ Namely,   $r^{12}=r^{mk}_{ij} e^{i}_{k} \otimes e^{j}_{m}\otimes 1$ and so on. The element $\sigma$ is 
given by $\sigma =e^j_i\otimes e^i_j $.

For the third interpretation, we shall define operators $r, a, \bar r, a^{*}: Mat_{N}\to Mat_{N}$ by 
 $\quad r(x)^p_q=r^{m p}_{n q} x^{n}_{m}$,  $\quad a(x)^p_q=a^{m p}_{n q} x^{n}_{m},$ $\bar r(x)^p_q=r^{p m}_{n q} x^{n}_{m}, \quad   a^{*}(x)^p_q=a^{p m}_{q n} x^{n}_{m}.$ 
 
 Then (\ref{r1}), (\ref{r2})-(\ref{r5}) provide the following operator identities:  
$$ 
\begin{array}{c}
 r(x)=-r^*(x), \qquad r(x) r(y)=r(x r(y))+r(x) y),\\[2mm]
 \bar r(x)=-\bar r^*(x), \qquad \bar r(x) \bar r(y)=\bar r(x \bar r(y))+\bar r(x) y),\\[2mm]
  a(x) a^{*}(y)= a^{*}(y) a(x), \\[2mm]
 a^*(y a(x))=r(x   a^*(y))-r(x)  a^*(y),\\[2mm]
 a(x)a(y)=-a(r(y) x)-a(y r(x)),\\[2mm]
 a^*(a(x) y)=r( a^*(y) x)-   a^*(y) r(x),\\[2mm]
 a(y a^*(x))=-\bar r(x   a(y))+\bar r(x)  a(y),\\[2mm]
 a^*(x)a^*(y)=a^*(\bar r(y) x)+a^*(y \bar r(x)),\\[2mm]
a(a^*(x) y)=-\bar r(a(y) x)+ a(y) \bar r(x)
\end{array}
$$ for any $x,y$.  First two of these identities mean that  operators  $r$ and $\bar r$ satisfies the Rota-Baxter equation \cite{Rota} and this  fact implies also
that the new matrix multiplications $\circ_r$ and $\circ_{\bar r}$ defined by
$$
x\circ_{r}y= r(x)y + xr(y),\quad x\circ_{\bar r}y= {\bar r}(x)y + x{\bar r}(y)
$$
are associative.

 \subsection{Examples and classification of low dimensional quadratic double Poisson brackets}

It is easy to see that for $m=1$ non-zero quadratic double Poisson brackets does not exist. In the simplest non-trivial case $m=2$ the system of algebraic 
equations (\ref{r1}), (\ref{r2})-(\ref{r5}) can be straightforwardly solved . 

\begin{theorem}\label{classifD2}
	Let $m=2.$ Then the following Cases \textbf{1}-\textbf{7} form a complete
	list of quadratic double Poisson brackets up to  equivalence given  a linear change of the generators. 
	We present non-zero
	components of the tensors $r$ and $a$ only.
	
	{\bf Case 1.}  $r^{21}_{22}=-r^{12}_{22}=1$.
	The corresponding (non-zero) double brackets read 
	$$\db{v,v} = v\otimes u - u\otimes v;$$
	
	{\bf Case 2.} $r^{21}_{22}=-r^{12}_{22}=1,$   $a^{11}_{21}=a^{12}_{22}=1$.
	The corresponding (non-zero) double brackets:
	$$\db{v,v} = v\otimes u - u\otimes v + vu\otimes 1 - 1\otimes vu, \, \db{v,u} = u^2 \otimes 1,\,
	\db{u,v} = - 1\otimes u^2 ;$$
	
	{\bf Case 3.}  $r^{21}_{22}=-r^{12}_{22}=1,$   $a^{11}_{12}=a^{21}_{22}=1$.
	The corresponding (non-zero) double brackets:
	$$\db{v,v} = v\otimes u - u\otimes v + uv\otimes 1 - 1\otimes uv, \, db{u,v} = u^2 \otimes 1,\,
	\db{v,u} = - 1\otimes u^2 ;$$
	
	{\bf Case 4.} $r^{22}_{21}=-r^{22}_{12}=1$.  The corresponding (non-zero) double brackets:
	$$\db{v,u} = v\otimes v, \,  \db{u,v} = - v\otimes v;$$
	
	{\bf Case 5.} $r^{22}_{21}=-r^{22}_{12}=1,$;   $a^{21}_{11}=a^{22}_{12}=1$.
	The corresponding (non-zero) double brackets:
	$$\db{v,u} = v\otimes v - 1\otimes v^2, \,  \db{u,v} = - v\otimes v + v^2 \otimes 1,\, \db{u,u} = uv\otimes 1 - 1\otimes uv;$$
	
	{\bf Case 6.} $r^{22}_{21}=-r^{22}_{12}=1,$;   $a^{12}_{11}=a^{22}_{21}=-1$. 
	The corresponding (non-zero) double brackets:
	$$\db{v,u} = v\otimes v - v^2 \otimes 1, \,  \db{u,v} = - v\otimes v + 1 \otimes v^2,\, \db{u,u} = - vu\otimes 1 + 1\otimes vu;$$
	
	{\bf Case 7.}   $a^{11}_{22}=1$.  The corresponding (non-zero) double brackets:
	$$\db{v,v} = u^2\otimes 1 - 1\otimes u^2.$$
\end{theorem} 

For a proof of  (\ref{classifD2}) see \cite{OdRubSok_1}.

\begin{remark}{\rm Cases \textbf{2}} and {\bf 3} as well as {\rm Cases \textbf{5}} and {\bf 6} are linked via the  involution    

\end{remark} 

\begin{remark} {\rm Case\textbf{ 1}} is equivalent to the double bracket from Example 1 with $m=2.$\end{remark}

\begin{remark} It is easy to verify (see \cite{Agui}) that there exist only two non-isomorphic anti-Frobenius subalgebras in $\operatorname{Mat}_2(\mathbb{C})$. They are matrices with one zero 
	column and matrices with one zero row.  {\rm Cases \textbf{1}} and {\bf 4} correspond to them.\end{remark}

\begin{remark} Notice that the trace Poisson brackets for {\rm Cases \textbf{2}} and {\bf 4} are non-degenerate. Corresponding symplectic forms can be found in \cite{biel} 
	(Example 5.7 and Lemma 7.1).\end{remark}  

\begin{remark} The corresponding Lie algebra structures on the trace space $\A/[\A,\A]$  are trivial (abelian) in all cases, except the cases \textbf{2}, \textbf{3} and \textbf{4} :
	$$ 
	[\bar u,\bar v] = -\bar u^2 \quad ({\rm Case}\, 2), \quad [\bar u,\bar v] = \bar u^2\quad  ({\rm Case}\, 3),\quad [\bar u,\bar v] = -\bar v^2\quad ({\rm Case}\,  4).
	$$
	These cases give the isomorphic Lie algebra structures on $\A/[\A,\A]$ with respect to the involutions
	$u\to v,\quad v\to u$ and $u\to u\quad v\to -v.$
\end{remark}

\begin{example} Consider the trace Poisson bracket (\ref{Poisson}) corresponding to case \textbf{6}. Its Casimir functions are given by
	$$
	\mbox{tr}\,v^k, \qquad \mbox{tr}\,u v^k,  \qquad k=0,1,...
	$$
	where $u=x_1, v=x_2.$ Functions $\mbox{tr} \, u^i$ and $\mbox{tr} \,v u^i$, where $i=2,3,...$ commute each other with 
	respect to this bracket.
\end{example}




\bibliographystyle{plain}

\end{document}